\documentclass{article}

\usepackage{arxiv}

\usepackage[utf8]{inputenc} 
\usepackage{fontenc,xcolor}    
\usepackage{hyperref}       
\usepackage{url}            
\usepackage{booktabs}       
\usepackage{nicefrac}       
\usepackage{microtype}      
\usepackage{graphicx}
\usepackage{natbib}
\usepackage{amsmath,amssymb,bm}
\usepackage[shortlabels]{enumitem}
\usepackage{setspace}
\usepackage{ntheorem}

\graphicspath{{figures/}}

\newcounter{hypothesis}

\newtheorem{hyp}[hypothesis]{}

\numberwithin{equation}{section}
\newtheorem{theorem}{Theorem}[section]
\newtheorem{lemma}[theorem]{Lemma}
\newtheorem{remark}[theorem]{Remark}
\newtheorem{proposition}[theorem]{Proposition}

\theoremstyle{nonumberplain}
\newtheorem{proof}{Proof}

\newcommand{\as}[1]{\vert#1\vert}
\newcommand{\magg}[1]{\vert#1\vert ^{2}}
\newcommand{\norm}[1]{\left\Vert#1\right\Vert}
\newcommand{\grad}{\nabla}
\newcommand{\p}{\partial}
\newcommand{\dx}{{\rm d}x}
\newcommand{\dy}{{\rm d}y}

\DeclareMathOperator{\sign}{sign}
\DeclareMathOperator*{\Span}{span}
\DeclareMathOperator{\ind}{ind}
\DeclareMathOperator{\ran}{Ran}

\DeclareMathOperator{\card}{card}
\DeclareMathOperator*{\argmin}{arg\,min}

\DeclareMathOperator{\Div}{div}
\DeclareMathOperator{\Dim}{dim}

\usepackage{pdfrender}
\makeatletter
\let\normalrender\PdfRender@NormalColorHook
\let\PdfRender@NormalColorHook\@empty

\makeatother
\pdfrender{StrokeColor=black,TextRenderingMode=2,LineWidth=0.25pt}

\title{Long-time behaviour and bifurcation analysis of a two-species aggregation-diffusion system on the torus}

\author{
        {Jos{\'e} A. Carrillo}\\ 
	Mathematical Institute\\
	University of Oxford\\
	Oxford, England \\
	\texttt{carrillo@maths.ox.ac.uk}
    \AND
        {Yurij Salmaniw}\\ 
	Mathematical Institute\\
	University of Oxford\\
	Oxford, England \\
	\texttt{yurij.salmaniw@maths.ox.ac.uk}
}

\renewcommand{\shorttitle}{J. A. Carrillo \& Y. Salmaniw}

\hypersetup{
pdftitle={Long-time behaviour and bifurcation analysis of a two-species aggregation-diffusion system on the torus},
pdfsubject={},
pdfauthor={J. Carrillo \& Y. Salmaniw},
pdfkeywords={aggregation-diffusion systems, bifurcation from simple eigenvalue, long-time behaviour, supercritical bifurcations, subcritical bifurcations, exchange of stability}
}

\begin{document}

\maketitle

\begin{abstract}
We investigate stationary states, including their existence and stability, in a class of nonlocal aggregation-diffusion equations with linear diffusion and symmetric nonlocal interactions. For the scalar case, we extend previous results by showing that key model features, such as existence, regularity, bifurcation structure, and stability exchange, continue to hold under a mere bounded variation hypothesis. For the corresponding two-species system, we carry out a fully rigorous bifurcation analysis using the bifurcation theory of Crandall \& Rabinowitz. This framework allows us to classify all solution branches from homogeneous states, with particular attention given to those arising from the self-interaction strength and the cross-interaction strength, as well as the stability of the branch at a point of critical stability. The analysis relies on an equivalent classification of solutions through fixed points of a nonlinear map, followed by a careful derivation of Fr{\'e}chet derivatives up to third order. An interesting application to cell-cell adhesion arises from our analysis, yielding stable segregation patterns that appear at the onset of cell sorting in a modelling regime where all interactions are purely attractive. 
\end{abstract}

\keywords{aggregation-diffusion systems \and bifurcation from simple eigenvalue \and  long-time behaviour \and supercritical bifurcations \and subcritical bifurcations \and exchange of stability}

\mscCodes{35B32 \and 35Q92 \and 35R09 \and 35R05 \and 35P05}



\section{Introduction}

We consider the following $n$-species aggregation-diffusion equation \cite{carrillo2019aggregation}
\begin{align}\label{eq:general_system}
\begin{cases}
        \frac{\p u_i}{\p t} = \grad \cdot \left( \sigma_i \grad u_i + u_i \sum_{j=1}^n \alpha_{ij} \grad (  W_{ij} * u_j ) \right)  \\
        u_i(x,0) = u_{i0} (x) ,
        \end{cases}
\end{align}
where $*$ denotes a convolution
\begin{align*}
    W_{ij} * u_j (x,t) := \int_\Omega W_{ij}(x-y) u_j (y,t) {\rm d}y 
\end{align*}
over a spatial domain $\Omega \subset \mathbb{R}^d$ (typically either $\mathbb{R}^d$ or $\mathbb{T}^d$), for some prescribed interaction kernels $W_{ij}$. Here, $\sigma_i>0$ is the diffusivity of the $i^{\textup{th}}$ population, while $\alpha_{ij} \geq 0$ describes the strength of the interaction from population $i$ to population $j$ governed by the kernel $W_{ij}$. For example, if $W_{ij}$ is coordinate-wise even and non-decreasing from the origin, $\alpha_{ij}$ describes the strength of attraction of population $i$ to population $j$; when it is non-increasing from the origin, $\alpha_{ij}$ describes the strength of repulsion of population $i$ from population $j$.

There has been a growing literature describing the qualitative (existence, uniqueness, regularity) and quantitative (stationary solution profiles, local stability, global asymptotic stability) behaviour of solutions to problem \eqref{eq:general_system}. Though not presently our primary concern, we highlight some relevant efforts concerning the well-posedness of the problem. This has been answered in several instances, typically depending on the regularity of the interaction kernels, structural requirements, or conditions on the initial mass. In the scalar case, well-posedness is proven in \cite{Carrillo2020} for kernels $W \in W^{2,\infty}(\mathbb{T}^d)$ using the iterative-scheme approach of \cite[Theorem 4.5]{Chazelle2017WellPosedness}. More recently, well-posedness of the $n$-species system for kernels $W_{ij} \in W^{2,\infty}(\mathbb{T}^d)$ was proven in \cite{giunta2022local} using a semigroup theory approach. For kernels that are merely $L^p(\mathbb{T}^d)$ (under some conditions on $p$), well-posedness for the $n$-species system was proven in \cite{jungel2022nonlocal} for positive-definite kernels satisfying a detailed balance condition (see \eqref{condition:detailed_balance}). This was achieved using entropy methods. In \cite{carrillo2024wellposedness}, the authors prove the well-posedness of the $n$-species system on the whole space and on the torus with limited conditions on the kernels. First, they prove a global existence result for $W_{ij} \in L^1 (\mathbb{R}^d) \cap L^\infty (\mathbb{R}^d)$  by assuming that the kernels are of Bounded Variation and satisfy a detailed balance condition. These results were also obtained using entropy estimates and the compactness lemma of Aubin-Lions.

In this work, we will focus on the novelty of the bifurcation branches with respect to suitable parameters brought up by the multiple populations aspect under consideration. For this reason, we reduce to the one-dimensional case in order to focus on the main goals related to interspecies interactions, and not be bothered by other symmetry considerations as in \cite[Remark 4.6]{Carrillo2020}.

\subsection{The scalar equation revisited}\label{subsec:intro_scalar_equation}

In the scalar case $n=1$, \eqref{eq:general_system} has a gradient-flow structure whenever the interaction kernel $W$ is even \cite[Eq. (1.2)]{Carrillo2020}. We briefly review existing results for this case, which reads in one dimension:
\begin{align}\label{eq:scalar_system_1}
\begin{cases}
        \frac{\p u}{\p t} =  \left( \sigma u_x + \alpha u ( W * u )_x \right)_x ,  \\
        u(x,0) = u_{0} (x) \geq 0.
        \end{cases}
\end{align}
This problem belongs to a larger class of dissipative partial differential equations with gradient flow structure in the sense of probability measures, see \cite{Jordan1998, Otto2001,CMV03,AGS08} and the survey papers \cite{carrillo2019aggregation,G24}. We first present and extend some existing results in one spatial dimension, most of which were initially obtained in \cite{Carrillo2020}, as this will prepare us nicely for the analysis of the multi-species system. We only need to define the cosine transform of a given kernel $W$: 
\begin{align}\label{eq:cos_transform_basis_function}
    \widetilde W(k) := \int_\mathbb{T} W(x) w_k(x)\, \dx, \quad w_k(x) :=  (2/L)^{1/2} \cos (2 \pi k x/L),
\end{align}
for $k \geq 1$ (see Section \ref{sec:preliminaries} for further details). Before describing the bifurcation structure, we first identify conditions under which the homogeneous state $u_\infty = L^{-1}$ is globally asymptotically stable. Notice that the zero$^{\textup{th}}$ Fourier mode of $W$ can always be assumed to be zero by shifting the interaction potential $W$.

\begin{theorem}[Global asymptotic stability of homogeneous state, scalar case]\label{thm:global_stability_scalar}
    Let $u(x,t)$ be a classical solution to equation \eqref{eq:scalar_system_1} with smooth initial data and a smooth, even interaction kernel $W$. Then the following hold.
    \begin{enumerate}
        \item If $0 < \alpha < \tfrac{2 \pi \sigma}{3 L \norm{W_x}_{L^\infty}}$, then $\norm{u(\cdot,t) - \tfrac{1}{L}}_{L^2} \to 0$ exponentially as $t \to \infty$;
        \item If $\widetilde W(k) \geq 0$ for all $k \in \mathbb{Z}$, or if $0 < \alpha < \tfrac{2 \pi^2 \sigma}{L^2 \norm{W_{xx}}_{L^\infty}}$, then $\mathcal{H} (u(\cdot,t) \vert \tfrac{1}{L}) \to 0$ exponentially as $t \to \infty$, where
        $$
\mathcal{H} (u(\cdot,t) \vert \tfrac{1}{L}) := \int_\mathbb{T} u(\cdot,t) \log \left( L \, u(\cdot,t)\right) \dx
        $$
        denotes the relative entropy.
    \end{enumerate}
\end{theorem}
Importantly, this result tells us that the homogeneous state can fail to be the unique stationary solution only if the potential has negative Fourier modes. In particular, the notion of $H$-stability \cite{CCP15,Carrillo2020} or positive-definite kernels \cite{jungel2022nonlocal} becomes relevant in the following sense: a kernel $W \in L^2(\mathbb{T})$ is called \textit{$H$-stable} or \textit{positive-definite} if $\widetilde W(k) \geq 0$ for all $k \in \mathbb{Z}$. We therefore conclude that a necessary condition for the existence of an inhomogeneous stationary state to problem \eqref{eq:general_system_SS} for a single population requires that the interaction kernel $W$ has negative Fourier modes.

The authors of \cite{Carrillo2020} were able to describe the emergence of inhomogeneous stationary states through a bifurcation analysis of the homogeneous state. For this problem, it is possible to characterise all stationary states in terms of fixed points (or zeros) of a nonlinear map (see Theorem \ref{thm:existenceregularitySS}). In this setting, the bifurcation theory of Crandall \& Rabinowitz \cite{CR71} leads to the following result.
\begin{theorem}[Description of local bifurcations, scalar case (\cite{Carrillo2020})]\label{thm:local_bifurcations_scalar}
    Suppose $d=1$ and let $W \in H^1(\mathbb{T})$ be an even kernel. Denote by $(1/L, \alpha)$ the trivial branch of solutions to the stationary problem of \eqref{eq:scalar_system_1}. Then, every $k^* \geq 1$ such that
    \begin{enumerate}[i.)]
        \item $\card \{ k \in \mathbb{N} : \widetilde W(k) = \widetilde W(k^*) \} = 1$;
        \item  $\widetilde W(k^*) < 0$,
    \end{enumerate}
    leads to a bifurcation point $(1/L, \alpha_{k^*})$ of equation \eqref{eq:scalar_system_1}, where $\alpha_{k^*}$ is given by the formula
    \begin{align*}
        \alpha_{k^*} = - \frac{\sigma \sqrt{2L} }{\widetilde W(k^*)} .
    \end{align*}
    In particular, there exists a branch of solutions $(u,\alpha) = (u^*(s),\alpha(s))$ having the following form:
    $$
u^* = u^*(s) = \frac{1}{L} + s \sqrt{\frac{2}{L}} \cos \left( \frac{2 \pi k^* x}{L}  \right) + o(s), \quad s \in (-\delta, \delta),
    $$
    for some $\delta > 0$, where $\alpha : (-\delta, \delta) \mapsto V$ is a twice continuously differentiable function in a neighbourhood $V$ of $\alpha_{k^*}$ satisfying
$$
\alpha(0) = \alpha_{k^*}\, , \quad \alpha^\prime (0) = 0\, , \quad \alpha^{\prime \prime} (0) = \frac{L}{2} \alpha_{k^*} \left[ 1 - \left( \frac{\widetilde W(2k)}{ \widetilde W(k) - \widetilde W(2k)} \right) \right].
$$
    Consequently, when $\widetilde W(2k^*) < \widetilde W(k^*)$ or $\widetilde W(2k^*) > \tfrac{1}{2} \widetilde W(k^*)$, the bifurcation is a supercritical pitchfork bifurcation; when $\widetilde W(k^*)< \widetilde W(2k^*)<\tfrac{1}{2}\widetilde W(k^*)$, the bifurcation is a subcritical pitchfork bifurcation. In either case, $u^*$ is the only inhomogeneous state near $(1/L, \alpha_{k^*})$.
\end{theorem}
The formula for $\alpha^{\prime \prime}(0)$ given in the proof of \cite[Theorem 4.2]{Carrillo2020} was incomplete, and so we rectify this here. More precisely, a correction term was missing, which introduces a resonance term appearing as a contribution from $\widetilde W(2k^*)$. Since $\widetilde W(k^*) < 0$ always holds, any kernel with $\widetilde W(2k^*) \geq 0$ will yield a supercritical bifurcation. Therefore, a necessary condition for the existence of a subcritical bifurcation is that $\widetilde W(2k^*) < 0$, i.e., $k = 2k^*$ must also be a candidate bifurcation point. We direct interested readers to \cite[Ch. 3]{BH21} for some local and global bifurcation results on $\mathbb{T}$ for a related nonlocal adhesion model.

We refer to the points $(1/L, \alpha_{k^*})$ identified in Theorem \ref{thm:local_bifurcations_scalar} as \textbf{\textit{bifurcation points}}. These points produce a nontrivial stationary solution of problem \eqref{eq:scalar_system_1}, but we cannot discuss their stability in general. However, we can determine the stability of the \textit{first} point of bifurcation, assuming at least one such $\alpha_{k^*}$ exists as described in Theorem \ref{thm:local_bifurcations_scalar_2}. It is useful to first define the sets $\mathcal{K}^\pm := \{ k \in \mathbb{N} : \pm \widetilde W(k) > 0 \}$ for a given kernel $W$. Then, there exists a critical value at which the homogeneous stationary state loses stability, which is given by
\begin{align}\label{critical_alpha_positive}
    \alpha^* (W) := \begin{cases}
     -\frac{\sigma \sqrt{2L}}{\min_{k \geq 1} \{  \widetilde W (k)  \}} > 0, \quad \text{ whenever } \quad \mathcal{K}^- \neq  \emptyset ; \cr
        \quad\quad +\infty, \quad\quad\quad\quad\quad\quad \text{ otherwise,}
    \end{cases} 
\end{align}
in the sense that the homogeneous state is linearly stable whenever $\alpha \in [0, \alpha^*(W))$, and is linearly unstable for $\alpha > \alpha^*(W)$. 

We refer to $\alpha^*(W)$ defined in \eqref{critical_alpha_positive} as the \textbf{\textit{point of critical stability}} for the kernel $W$. When $\mathcal{K}^- \neq \emptyset$, we then denote by $k_W := \argmin_{k \geq 1} \{ \widetilde W(k) \}$ the associated \textbf{\textit{critical wavenumber}}, whenever it is unique. Since every branch is found to be supercritical, we can extend Theorem \ref{thm:local_bifurcations_scalar} here by showing that an exchange of stability occurs at the point of critical stability. This exchange of stability follows from properties of the semiflow generated by the time-dependent problem paired with spectral properties of the linearised operator.

\begin{theorem}[Point of critical stability \& stability exchange, scalar case]\label{thm:local_bifurcations_scalar_2}
        Suppose the hypotheses of Theorem \ref{thm:local_bifurcations_scalar} hold. When $\widetilde W(k) < 2 \widetilde W(2k)$ or $\widetilde W(2k) < \widetilde W(k)$, the bifurcation is supercritical, and the homogeneous solution $u_\infty = 1/L$ and the emergent solution $u^*$ with frequency $k_W$ exchange stability at $\alpha = \alpha^*(W)$: $u_\infty$ is locally asymptotically stable for $\alpha \in [0, \alpha^*(W))$ and is unstable for $\alpha \in (\alpha ^*(W), \infty)$; $u^*$ is locally asymptotically stable for $\alpha \in (\alpha^*(W),\alpha^*(W) + \delta_0)$, for some $\delta_0>0$. When $\widetilde W(k)< \widetilde W(2k)<\tfrac{\widetilde W(k)}{2}$, the bifurcation is subcritical, the emergent branch is unstable, and no exchange of stability occurs.
\end{theorem}

\begin{remark}\label{remark:scalar_bifurcation_remark}\
\begin{itemize}

    \item The results of Theorems \ref{thm:local_bifurcations_scalar} and \ref{thm:local_bifurcations_scalar_2} still hold even if we relax the $H^1$-regularity assumption on the kernel $W$ to a bounded variation condition (see Hypothesis \textbf{\ref{hyp:kernel_shape}} and Appendix \ref{sec:BV_functions}). This includes, for example, the commonly used top-hat kernel, see e.g., \cite{Potts2016, potts2016territorial, Fagan2017, pottslewis2019, wangsalmaniw2022, giunta2022detecting, giunta2024weakly}, which is of bounded variation in any dimension \cite{carrillo2024wellposedness}, but does not belong to $H^1$.

    \item  In the statement of Theorem \ref{thm:local_bifurcations_scalar_2}, it is understood that the instability of the homogeneous state beyond $\alpha^*$ holds in the nonlinear sense. In particular, the Principle of Linearised Stability holds, and linear instability of $u_\infty$ implies nonlinear instability of $u_\infty$.
    
    \item We depict the results of Theorems \ref{thm:local_bifurcations_scalar} and \ref{thm:local_bifurcations_scalar_2} in Figure \ref{fig:scalar_bif_1}. For this example, we fix $L=2 \pi$, $\sigma = 1$, and we choose $W(x)$ to be the ``repulsive" top-hat kernel 
    \begin{align*}
        W = \begin{cases}
        \tfrac{1}{2R}, \quad \as{x} \leq R, \cr
        0, \quad\quad \text{otherwise},
    \end{cases}
    \end{align*}
   with radius $R=L/10$. 
    
    \item In the left panel of Figure \ref{fig:scalar_bif_1}, we plot several Fourier coefficients $\widetilde W(k)$ of the kernel $W$, and the \textup{possible} bifurcation points $\alpha_k$ identified in Theorem \ref{thm:local_bifurcations_scalar}, noting that only those $\alpha_k > 0$ yield a bifurcation point. We emphasise the connection between the minimal value of $\widetilde W(k)$ (occurring in $\mathcal{K}^-$) and the point of critical stability $\alpha^*(W)$; more generally, we emphasize that 
    $$
    \alpha_k > 0 \iff \widetilde W(k) < 0,
    $$
    so that only those wavenumbers $k$ such that $\widetilde W(k) < 0$ can yield a bifurcation point.
    
    \item In the right panel of Figure \ref{fig:scalar_bif_1}, we display a typical bifurcation diagram for the scalar case. First, we observe that each wavenumber $k$ such that $\widetilde W(k) < 0$ leads to a bifurcation point, the point is given by $\alpha = \alpha_k$. The point of critical stability (in this case, at $\alpha = \alpha^*(W) = \alpha_1$) produces a branch with frequency $k= k_W = 1$, and since $\widetilde W(2k_W) = \widetilde W(2) > 0$, the bifurcation is supercritical and an exchange of stability occurs between the homogeneous branch and the first emergent branch. A secondary bifurcation point (in this case occurring at $\alpha = \alpha_3$) produces another supercritical branch with frequency $k = 3$, and its stability is unknown. Wavenumbers such that $\widetilde W(k) > 0$ (e.g., $\alpha_2$ in this example) do not lead to a bifurcation point.

    \item As we will find for the two-species case, the point of critical stability for the potential $-W$, namely $\alpha^*(-W)$, will also play a key role. Therefore, in Figure \ref{fig:scalar_bif_1} we highlight the connection between the maximal value of $\widetilde W(k)$ (occurring in $\mathcal{K}^+$) and the point of critical stability $\alpha^*(-W)$.
\end{itemize}
\end{remark}

\begin{figure}[ht]
    \centering
    \includegraphics[width=0.95\linewidth]{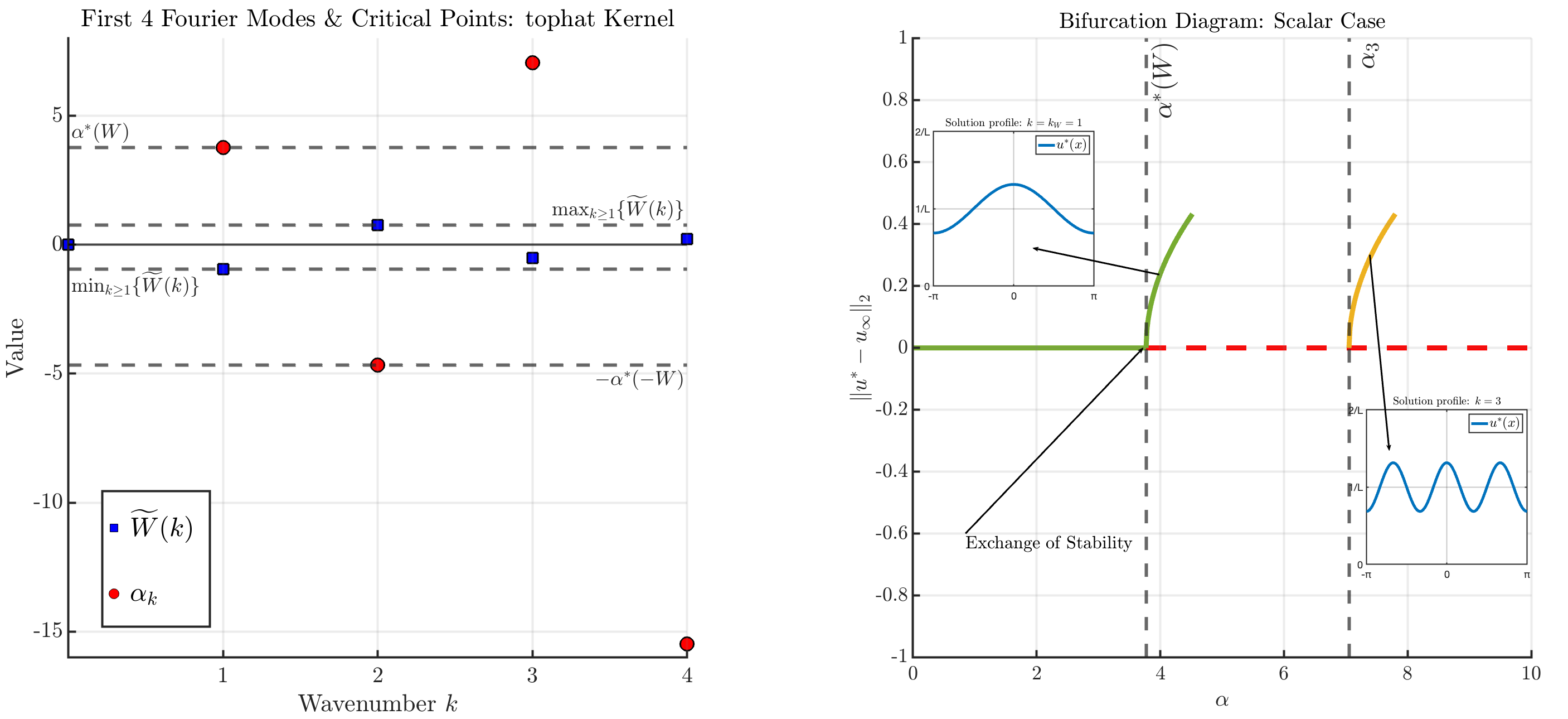}
    \caption{A depiction of the results of Theorem \ref{thm:local_bifurcations_scalar} and \ref{thm:local_bifurcations_scalar_2}. The left panel displays the Fourier coefficients of $W$ (blue squares) and the associated possible bifurcation points $\alpha_k$ (red dots). The right panel displays a typical bifurcation diagram. Green lines denote a stable branch, dashed red lines denote an unstable branch, and the stability of the yellow branches is unknown. See Remark \ref{remark:scalar_bifurcation_remark} for further discussion.}
    \label{fig:scalar_bif_1}
\end{figure}

\subsection{The multi-species system}

In the case of several interacting populations, some works have investigated the structure of solutions to problem \eqref{eq:general_system} in one spatial dimension. Early contributions focused on the linear stability of the homogeneous state, with results depending on the structure of the interaction matrix $\alpha_{ij}$. For instance, \cite{pottslewis2019} analysed several biologically motivated nonlocal models, with particular attention to top-hat interaction kernels. More recently, \cite{giunta2022detecting} combined properties of the associated energy functional (akin to the free energy functional in \eqref{free_energy_functional} below) with numerical methods and asymptotic analysis in the local limit, where the interaction kernel $W(x)$ tends to a Dirac mass. In a further development, \cite{giunta2024weakly} employed a weakly nonlinear analysis approach to study the bifurcation structure near the first threshold of instability. In \cite{JKME23}, the authors investigated the influence that spatial dimension has on the stability of the homogeneous state. To our knowledge, the present work is the first to rigorously characterise all bifurcations from the homogeneous state via simple multiplicity eigenvalues using the theory of Crandall–Rabinowitz.

In the bifurcation analysis of the scalar case, the equivalence between stationary solutions and fixed points of a nonlinear map is a very useful property, and this is always possible in the scalar case when a gradient-flow structure is available. For $n \geq 2$, system \eqref{eq:general_system} maintains a gradient-flow structure under additional constraints. For example, if $W_{ij} = W$ for all $i,j=1,\ldots,n$, then system \eqref{eq:general_system} has a gradient-flow structure if the coefficient matrix $\{ \alpha_{ij} \}_{i,j=1}^n$ is symmetric. More generally, the \textit{detailed balance} condition, originally introduced by Boltzmann in his development of the kinetic theory of gases \cite{boltzmann1872}, is necessary and sufficient to maintain a gradient-flow structure:
\begin{align}\label{condition:detailed_balance}
    \exists \pi_i >0 : \pi_i \alpha_{ij} W_{ij} (x-y) = \pi_j \alpha_{ji} W_{ji} (y-x) \quad \forall i,j=1,\ldots,n, \, x,y \in \mathbb{T}.
\end{align}
Heuristically, the detailed balance condition \eqref{condition:detailed_balance} says that $\alpha_{ij} W_{ij}$ is symmetrisable through multiplication with a vector $(\pi_1, \ldots, \pi_n) \in \mathbb{R}^n$. This condition was used in, e.g., \cite{jungel2022nonlocal, carrillo2024wellposedness}, to establish the well-posedness of solutions. In the setting of linear stability, this condition ensures that the linearised problem yields a self-adjoint operator, from which we conclude the spectrum is real.

Here, we are primarily concerned with the long-term behaviour of solutions to \eqref{eq:general_system} as well as the properties of solutions to the stationary problem on the one-dimensional torus $\mathbb{T}$. To simplify exposition, we assume hereafter that
\begin{align}\label{interaction_assumptions}
\begin{cases}
        &\sigma_i = \sigma > 0; \quad \alpha_{ii} W_{ii} := \alpha_i W_i; \quad i=1,2;  \\
   & \alpha_{ij} W_{ij} := \gamma W, i \neq j , 
\end{cases}
\end{align}
where $\alpha_i, \gamma \geq 0$, $i=1,2$. In this way, self-interaction is governed by $W_i$ with strength $\alpha_i \geq 0$, while the (symmetric) cross-interaction is governed by $W$ with strength $\gamma\geq 0$. Our stationary problem then reads
\begin{align}\label{eq:general_system_SS}
        0 =  \left( \sigma ( u_i )_x + u_i (  \alpha_{i}  W_{i} * u_i + \gamma W * u_j )_x \right)_x , \quad\quad \quad i=1,2,\quad i \neq j, \quad\quad x \in \mathbb{T}.
\end{align}

The goal of this paper is to first extend the scalar equation results of \cite{Carrillo2020}, and then to apply the theory of Crandall-Rabinowitz to the multi-species system \eqref{eq:general_system}. First, we show a global asymptotic stability result for smooth interaction potentials with limited structural requirements (see Theorem \ref{thm:global_stability_system}). Then, under the structural criteria of \eqref{interaction_assumptions} and Hypothesis \textbf{\ref{hyp:kernel_shape}}, we extend the bifurcation analysis described in Section \ref{subsec:intro_scalar_equation} to the two-species stationary problem \eqref{eq:general_system_SS}. These results are contained in Theorems \ref{thm:bifurcations_alpha1_1}-\ref{thm:main_results_local_stability_alpha_system} for bifurcations with respect to $\alpha_1$, while Theorems \ref{thm:bifurcations_gamma_1}-\ref{thm:main_results_local_stability_gamma_system} contain the bifurcation results with respect to $\gamma$. 

We highlight some of the novelty of our findings as follows. In the scalar case, we have improved the results of \cite{Carrillo2020} first by allowing kernels merely of bounded variation, rather than the typical $H^1$-regularity assumption found in most existing works. This is sufficient for most of our scalar results, including well-posedness results, characterisations of stationary states, bifurcation analyses, and exchanges of stability; whether this regularity is sufficient for the global asymptotic stability result to hold is left open. Weakening the regularity requirements of the interaction kernel is motivated primarily by the use of the top-hat kernel in much of the ecology literature, see e.g., \cite{wangsalmaniw2022, painter2023biological, pottslewis2019, giunta2022detecting, giunta2024weakly}. However, it is also an interesting mathematical challenge to reduce the regularity of the kernels considered, see \cite{carrillo2024wellposedness}. 

For the two-species system, we provide the first rigorous application of the Crandall-Rabinowitz bifurcation theory to understand the (local) bifurcation structure. Importantly, our analysis is general in that it depends only on the Fourier coefficients of the kernel, rather than on a particular choice of kernel. We treat in detail bifurcation with respect to $\alpha_1$, which corresponds to quantitative changes in solution behaviour due to self-interaction forces, and with respect to $\gamma$, which corresponds to quantitative changes in solution behaviour due to cross-interaction forces. We can identify both sub- and supercritical bifurcation branches and determine the local stability of the branches emerging at a point of critical stability. The understanding of the local stability of the homogeneous state shown in Proposition \ref{prop:local_stability} also appears to be new, highlighting the importance of both critical values $\alpha^*(\pm W)$. We provide in full detail a precise calculation of the Fr{\'e}chet derivatives in the Appendix. While essentially a technical result, these formulas hold for any number of interacting populations, and may be useful for future researchers when considering bifurcations without the explicit structure of \eqref{interaction_assumptions}, or when considering cases $n\geq3$.

Finally, in Section \ref{subsec:application}, we deduce from our analysis that it is possible to observe a stable segregation pattern in a fully attractive regime, indicative of the onset of cell-sorting behaviour \cite{BDFS18,carrillo2019adhesion,FBC24}. This has a direct connection with the differential adhesion hypothesis \cite{foty2005differential,carrillo2019adhesion,FBC24}, where cell sorting is observed through a sufficient difference in adhesive strengths of two cell populations.

Before we state precisely our main findings, it is useful to first describe some general comparisons that we can make between: (A) the scalar equation versus the two-species system, and (B) bifurcations in the two-species case with respect to $\alpha_1$ (self-interaction strength) versus with respect to $\gamma$ (cross-interaction strength). In the following discussion, we implicitly assume that a bifurcation point occurs at a simple eigenvalue, thereby avoiding technicalities that may distract from the bigger picture.

\underline{Single-species versus multi-species cases.}
\begin{itemize}

    \item \textbf{Global asymptotic stability.} In the single-species case $n=1$, Theorem \ref{thm:global_stability_scalar} shows that for a sufficiently regular kernel $W$ and sufficiently small aggregation strength $\alpha$, the homogeneous state is globally asymptotically stable. 
    
    The same result holds for the $n$-species system if one assumes all interaction kernels are sufficiently smooth, and all interaction parameters $\alpha_{ij}$ are sufficiently small. This is what is presented in Theorem \ref{thm:global_stability_system} for $n=2$ interacting populations; the statement is also valid for $n$ interacting populations with no further modification.
    
    \item \textbf{Positive-definite kernels \& existence of bifurcation points.} In the single-species case, there are no bifurcation points if the interaction kernel is positive-definite (one whose Fourier coefficients are all non-negative); the homogeneous state $u_\infty$ is locally asymptotically stable for all $\alpha \geq 0$ (in fact, $u_\infty$ is \textit{globally} asymptotically stable if the interaction kernel is sufficiently smooth). Equivalently, a necessary and sufficient condition for the existence of a bifurcation point is that $\mathcal{K}^- \neq \emptyset$ so that there exists a wavenumber $k \geq 1$ such that $\widetilde{W} (k) <0$. 
    
    For the multi-species case, we find that the same result holds only when \textit{all} interaction kernels $W_{ij}$ are positive definite. For example, if $W_{ij} = W$ for all $i,j = 1, \ldots, n$ and $\widetilde W(k) \geq 0$ for all $k \geq 1$, then the $n$-species system has no bifurcation point. However, if $W_{11}$ is not positive definite, we can guarantee the existence of a bifurcation point, even if all other interaction kernels remain positive definite. This is what is found in Theorems \ref{thm:bifurcations_alpha1_1} and \ref{thm:bifurcations_gamma_1}.
    
    \item \textbf{Branch direction \& stability exchange.} In the single-species case, there \textit{always} exists a point of critical stability whenever $\mathcal{K}^- \neq \emptyset$, and the region of linear stability is an interval of the form $[0, \alpha^*)$ for some $\alpha^*>0$. At the critical wavenumber $k^*$, there always holds $\widetilde W(k^*) < \widetilde W (2k^*)$, and so the bifurcation is supercritical provided the competing harmonic at $2k^*$ is displaced sufficiently far to the right of the critical mode. In particular, all bifurcation points $\alpha_{k^*}$ (not just the first one) are found to be supercritical whenever $\widetilde W(2k^*) \geq 0$. In such a case, we find that an exchange of stability must occur at the point of critical stability.
    
    In the multi-species case, both aspects (existence and stability exchange) change significantly: depending on the bifurcation parameter of interest, there may be one, two, or no points of critical stability (see Theorem \ref{thm:main_results_local_stability_gamma_system}, where a unique point of critical stability exists, versus Theorem \ref{thm:main_results_local_stability_alpha_system}, where up to two points of critical stability exist). As a result, sub- and supercritical bifurcations are now possible, even when $\widetilde W(2k^*) = 0$ holds. However, when a point (or points) of critical stability exists, we can still establish an exchange of stability result (this is also found in Theorems \ref{thm:main_results_local_stability_alpha_system} and \ref{thm:main_results_local_stability_gamma_system}).
    
    \item \textbf{Wavenumber at point of critical stability.} In the single-species case, we identify the point of critical stability $\alpha^*(W)$ as defined in \eqref{critical_alpha_positive}, and the first branch emerges with frequency $k = k_W$; there is no other possibility. 
    
    In the multi-species case, the point of critical stability for the kernel $-W$, namely $\alpha^*(-W)$, now plays a significant role. Whenever one exists (there may be more than one!), a point of critical stability may now occur at wavenumber $k = k_{W}$ or $k=k_{-W}$, and the particular wavenumber where this happens depends on other parameters in the model and the relative magnitudes of $\alpha^*(\pm W)$. Both situations are possible, and, aside from some degenerate cases, we have provided precise conditions to identify whether the critical wavenumber is $k_W$ or $k_{-W}$. This is what is described in Theorems \ref{thm:main_results_local_stability_alpha_system} and \ref{thm:main_results_local_stability_gamma_system}; see also Figure \ref{fig:stability_region_1} and Proposition \ref{prop:local_stability}.

    \item \textbf{Phase relationships.} In the scalar equation, there is only one solution component, and so it does not make sense to inquire about the phase of the solution. For two (or more) interacting populations, the emergent solution branches obtained may feature solution components that are \textit{in phase} (i.e., troughs align with troughs, peaks align with peaks) or \textit{out of phase} (i.e., troughs align with peaks). For all bifurcation branches identified in Theorems \ref{thm:bifurcations_alpha1_1} and \ref{thm:bifurcations_gamma_1}, we identify a simple analytical expression to determine the phase relationship between the solution components.
\end{itemize}

\underline{Bifurcation with respect to $\alpha_1$ versus $\gamma$ in the multi-species case.} 

\begin{itemize}

    \item \textbf{Point(s) of critical stability.} When considering bifurcation with respect to $\gamma \geq 0$, the situation is most comparable to the scalar equation in the following sense. Assuming a bifurcation point exists, there always exists a critical $\gamma^*>0$ so that the homogeneous state is locally asymptotically stable for $\gamma \in [0, \gamma^*)$ and is unstable for all $\gamma > \gamma^*$. Therefore, a point of critical stability always exists, and this first bifurcation is always supercritical whenever $\as{\widetilde W(2k^*)} \ll 1$. This is what we find in Theorem \ref{thm:main_results_local_stability_gamma_system}. 
    
    Much different is bifurcation with respect to $\alpha_1 \geq 0$: fixing $\alpha_1 = 0$ is insufficient to ensure the local stability of the homogeneous state, and a region of local stability may not exist. This is because the parameter $\gamma$ may be chosen large enough to destabilise the homogeneous state, independent of $\alpha_1$. Hence, a more careful description of the linear stability of the homogeneous state is necessary to uncover the bifurcation structure with respect to $\alpha_1$. Precise analytical criteria for the linear stability of the homogeneous state are what is found in Proposition \ref{prop:local_stability}.

    \item \textbf{Branch direction and wavenumber at point of critical stability.} When bifurcating with respect to $\gamma$, we need only to consider the case $\gamma \geq 0$ due to the assumed symmetry of the system. When $\as{\widetilde W(2k^*)} \ll 1$, the bifurcation is shown to always be supercritical, but it can occur at wavenumber $k_{W}$ or $k_{-W}$, depending on the $\alpha_i$'s and the relative magnitudes of $\alpha^*(\pm W)$. This is what is shown in Theorem \ref{thm:main_results_local_stability_gamma_system}, and can be understood through the linear stability result of Proposition \ref{prop:local_stability}.
    
    In contrast, when bifurcating with respect to $\alpha_1$, we may now bifurcate by increasing $\alpha_1$ \textit{or} by decreasing $\alpha_1$. This is precisely what happens when two points of critical stability exist, and we provide precise criteria for when this holds in Theorem \ref{thm:main_results_local_stability_alpha_system}. When two such points exist, the bifurcation in the increasing direction always occurs at wavenumber $k_{W}$, and the branch is supercritical if $\as{\widetilde W(2k^*)} \ll 1$; in the decreasing direction, the bifurcation always occurs at wavenumber $k_{-W}$, and the branch is subcritical if $\as{\widetilde W(2k^*)} \ll 1$. An exchange of stability still occurs in both directions.

\end{itemize}

The remainder of this manuscript is organised as follows. In Section \ref{sec:statement_of_results}, we provide the full rigorous statements of our key results. As the statements are technically dense, we include several remarks to provide further heuristic understanding. Several figures are included to provide a visual depiction of our key results for some exemplary cases. In Section \ref{subsec:application}, we demonstrate how our results have an interesting connection with the Differential Adhesion Hypothesis \cite{foty2005differential, carrillo2019adhesion}, and how our model captures the onset of cell-sorting observed experimentally. This is particularly interesting, as it shows that segregation patterns can occur in a purely adhesive setting. In Section \ref{sec:preliminaries}, we introduce relevant notations and conventions, and also provide a statement for the well-posedness of the time-dependent problem using the recent results of \cite{carrillo2024wellposedness}, which significantly weaken the regularity requirement on the interaction kernels. In Section \ref{sec:GAS}, we prove the global asymptotic stability result of the homogeneous state for the two-species system. In Section \ref{sec:stationary_states_charaterisation}, we establish several equivalent characterisations of the stationary states; this is particularly important to our approach, as it allows us to work with a pointwise nonlinear map, rather than a differential operator directly. In Section \ref{sec:linear_analysis}, we fully describe the spectral and linear stability of the single-species and two-species problems, concluding with a proof of Proposition \ref{prop:local_stability}. This identifies relevant bifurcation points and prepares us for subsequent exchange-of-stability results. Finally, in Section \ref{sec:bif_proofs_1}, we complete the bifurcation analysis and conclude with the proofs of all bifurcation results of Section \ref{sec:statement_of_results}. We include several well-known results in the Appendix, including a brief description of functions of Bounded Variation, explicit computation of all Fr{\'e}chet derivatives, and the proofs of the results presented in Section \ref{sec:stationary_states_charaterisation}.

\section{Statement of main results}\label{sec:statement_of_results}

Our goal is to extend this approach to the two-species system \eqref{eq:general_system_SS} with interaction potentials of the form in \eqref{interaction_assumptions} to describe the local bifurcation structure with respect to the self-interaction strength $\alpha_1$, or the (symmetric) cross-interaction strength $\gamma$. From the symmetry of the system, we may, without loss of generality, ignore bifurcations with respect to $\alpha_2$. We further assume that the mass of each population is normalized to one (i.e., $\int_\mathbb{T} u_i \dx = 1$, $i=1,2$). We may then write the free energy functional
\begin{align}\label{free_energy_functional}
    \mathcal{F} (\mathbf{u}) &=  \sigma \sum_{i=1}^2 \int_\mathbb{T} u_i \log (u_i)\ \dx + \frac{1}{2} \sum_{i\neq j} \int_\mathbb{T} u_i ( \alpha_i W_i * u_i + \gamma W * u_j )\ \dx \nonumber \\
    &=:  \sigma \mathcal{S}(\mathbf{u}) + \frac{1}{2} \mathcal{E} (\mathbf{u}, \mathbf{u}) ,
\end{align}
where $\mathbf{u} := (u_1, u_2)$, and $\mathcal{S} (\mathbf{u})$ and $\mathcal{E} (\mathbf{u}, \mathbf{u})$ represent the entropy and total interaction energy, respectively. Similar to, e.g., \cite{carrillo2024wellposedness}, one can write
\begin{align}
    f_i :=   \sigma \log (u_i) + \alpha_i W_i * u_i + \gamma W * u_j, \quad i\neq j,
\end{align}
so that $(u_i)_t = \tfrac{\partial}{\partial x}\left( u_i \tfrac{\partial f_i}{\partial x} \right)$ for each $i=1,\ldots,n$. Notice that $f_i=\tfrac{\delta \mathcal{F}}{\delta u_i}$, that is the variation of the free energy functional with respect to the component $u_i$. Formal computation then yields
\begin{align}\label{eq:energy_dissipation}
    \frac{{\rm d}}{{\rm d}t} \mathcal{F} (\mathbf{u})  + \sum_{i=1}^n \int _\mathbb{T} u_i  \as{\tfrac{\partial f_i}{\partial x}}^2 \dx = 0.
\end{align}
The restrictions \eqref{interaction_assumptions} on the matrix of interactions potentials are sufficient to ensure that any solution of problem \eqref{eq:general_system_SS} can be identified with a zero (or fixed point) of a nonlinear map (see \eqref{eq:nonlinear_map_1}-\eqref{eq:nonlinear_map_2}), and every such fixed point is a solution of the problem \eqref{eq:general_system_SS} (see Theorem \ref{thm:existenceregularitySS} and Proposition \ref{prop:equivalencies1}). As discussed above, necessary and sufficient conditions on the matrix of interaction kernels guaranteeing this property are the so-called detailed balance conditions \eqref{condition:detailed_balance}.

We first have the following result, an analogue of Theorem \ref{thm:global_stability_scalar} in the scalar case, which gives sufficient conditions for global asymptotic stability to hold in relative entropy.

\begin{theorem}[Global asymptotic stability of homogeneous state, system case]\label{thm:global_stability_system}
Suppose $\mathbf{u}(x,t)$ is a classical solution to system \eqref{eq:general_system} subject to \eqref{interaction_assumptions} with $d=1$, $n=2$, with smooth initial data and even interaction potentials $W$, $W_i \in W^{2,\infty}(\mathbb{T})$, $i=1,2$. If there holds
$$
0 < \frac{2 \pi^2 \sigma}{L^2} -  \gamma \norm{W_{xx}}_{L^\infty} -  \max_{i=1,2} \left\{\alpha_i \norm{(W_{i})_{xx}}_{L^\infty} \right\},
$$
then $\mathcal{H} (\mathbf{u} \, \vert \, \mathbf{u_\infty}) \to 0$ exponentially as $t \to \infty$, where $\mathcal{H} (\mathbf{u} \, \vert \, \mathbf{u_\infty})$ denotes the total relative entropy
\begin{align}
    \mathcal{H}(\mathbf{u}\, \vert \, \mathbf{u_\infty}) = \sum_{i=1}^2 \int u_i (\cdot,t)\log \left( \frac{u_i (\cdot,t)}{u_\infty} \right)\, {\rm d}x.
\end{align}
Moreover, if for $i=1,2$ we have $\widetilde W_i (k) \geq 0$ for all $k \geq 1$,  then the same convergence result holds so long as $0 \leq \gamma < \frac{2 \pi^2 \sigma}{ L^2 \norm{W_{xx}}_{L^\infty}}$. 
\end{theorem}
\begin{remark}\ 
\begin{itemize}
    \item While the smoothness requirement on $W, W_i$ can be weakened to a Bounded Variation condition for our bifurcation analysis (see Hypothesis \textbf{\ref{hyp:kernel_shape}}), it is not obvious how to weaken the smoothness assumption of Theorem \ref{thm:global_stability_system} to obtain the same result.
    \item While the statement is presented for kernels satisfying conditions \eqref{interaction_assumptions}, this is not necessary; in fact, all that is required is smoothness and evenness of the kernels. Moreover, the same result holds for the $n$-species case with appropriate adjustments to the bounds obtained. Therefore, the global asymptotic stability of the homogeneous state when the interaction strengths are sufficiently small is robust to any number of interacting populations and any sufficiently regular kernels, so long as they are even.
\end{itemize}
\end{remark}

Before presenting our bifurcation results for the two-species system, we first define some key quantities that will appear often throughout the remainder of the manuscript.
Fix $\sigma, L>0$ and consider a kernel $W \in L_s^2(\mathbb{T})$. We first define $h_k : \{ k \in \mathbb{N} : \widetilde W(k) \neq 0\} \mapsto \mathbb{R}$ by
\begin{align}\label{eq:h_k_relation}
    h_k := \frac{\sigma \sqrt{2L}}{\widetilde W(k)}.
\end{align}
Notice that these $h_k$'s are closely related to the bifurcation points identified in Theorem \ref{thm:local_bifurcations_scalar}; in particular, we have that $\alpha^*(W) = - h_{k_W}$ and $\alpha^*(-W) = h_{k_{-W}}$. 

For our bifurcation analysis, our main assumption on the kernels $W, W_i$ is as follows.
\begin{hyp}\label{hyp:kernel_shape}
    $W \in \textup{BV}(\mathbb{T}) \cap L^\infty(\mathbb{T})$ is even with zero mean, and for each $i=1,2$ there holds $W_i = \chi_i W$, where $\chi_i \in \{ 1, -1 \}$.
\end{hyp}

\begin{remark} \
    \begin{itemize} 
        \item In one spatial dimension, $\textup{BV} (\Omega) \subset L^\infty (\Omega)$ for any $\Omega \subset \mathbb{R}$, and so the second inclusion is redundant; in higher dimensions, this is no longer true, and boundedness is required independent of $\textup{BV}$ inclusion. We briefly introduce the class of \textup{BV} functions and some useful properties in Appendix \ref{sec:BV_functions}.
        
        \item Hypothesis \textbf{\ref{hyp:kernel_shape}} weakens the more typical $H^1$-regularity assumption; this is a minimally sufficient requirement such that a unique classical solution solving the time-dependent problem exists (see Theorem \ref{thm:wellposed_time}), is sufficient to ensure that the stationary states are necessarily smooth (see Theorem \ref{thm:existenceregularitySS}), and is sufficient to ensure that the associated linearised operator is sectorial (see, e.g., the proof of Theorem \ref{thm:local_bifurcations_scalar_2}). In particular, the top-hat kernel satisfies Hypothesis \textbf{\ref{hyp:kernel_shape}} but does not belong to $H^1$.

        \item The conditions of Hypothesis \textbf{\ref{hyp:kernel_shape}} include Hypotheses \textbf{(H2)} and \textbf{(H4)} of \cite{carrillo2024wellposedness}; in fact, by Lemma \ref{lemma:BV_TV_embeddings}, if $W \in \textup{BV} \subset L^\infty$, then $W * W \in H^1$ and $\norm{( W * W)_x}_{L^2} \leq \norm{DW}_{\textup{TV}} \norm{W}_{L^2}$ and Hypothesis \textbf{(H4)} of \cite{carrillo2024wellposedness} follows from Hypothesis Hypotheses \textbf{(H2)} of \cite{carrillo2024wellposedness}.

        \item Notice that if the interaction potential is increasing, respectively decreasing, as a function of the radius, the interaction forces between particles are attractive, respectively repulsive. 
        
        \item The introduction of $\chi_i$ is meant to distinguish between the sign of the self-interaction kernels $W_i$ in relation to the cross-interaction kernel $W$. The assumption that $W_i = \chi_i W$ means that all interactions are governed by the same \textit{shape} but may differ in whether they are attractive or repulsive. This hypothesis dramatically simplifies the presentation of our results and allows one to be more precise about the particular bifurcation structure as it reduces the problem to four possible combinations of intraspecies interactions (e.g., attractive/attractive, repulsive/attractive, attractive/repulsive, repulsive/repulsive). This could be generalised to more general $W_i$ that maintain the gradient-flow structure of the problem, e.g., satisfying the detailed balance condition \eqref{condition:detailed_balance}, see \cite{carrillo2024wellposedness}, but we do not explore this possibility further here.
    \end{itemize}
\end{remark}

We first present the following Proposition, containing necessary and sufficient conditions for the homogeneous state to be linearly stable. To avoid confusion, by linear stability we refer to the spectrum of the linearised operator; as our resultant eigenvalue problem is self-adjoint, notions of linear and spectral stability are equivalent.
\begin{proposition}\label{prop:local_stability}
    Assume Hypothesis \textbf{\ref{hyp:kernel_shape}} holds. Then, the homogeneous state $\mathbf{u}_\infty$ is linearly stable if and only if there holds
    \begin{align}\label{eq:gamma_minimum_1}
        0 < \min \left\{ [\alpha^*(W) - \chi_1 \alpha_1][\alpha^*(W) - \chi_2 \alpha_2] - \gamma^2, \ [\alpha^*(-W) + \chi_1 \alpha_1][\alpha^*(-W) + \chi_2 \alpha_2] - \gamma^2 \right\}.
    \end{align}
    More precisely, $\mathbf{u}_\infty$ is linearly stable if and only if
    \begin{align}
        \begin{cases}
            0 < [\alpha^*(W) - \chi_1 \alpha_1][\alpha^*(W) - \chi_2 \alpha_2] - \gamma^2 \quad \text{ whenever } \quad S^* < 0, \cr
            0 < [\alpha^*(-W) + \chi_1 \alpha_1][\alpha^*(-W) + \chi_2 \alpha_2] - \gamma^2 \quad \text{ whenever } \quad S^* > 0,
        \end{cases}
    \end{align}
    where
    \begin{align}\label{eq:S_star}
        S^* := \alpha^*(W) - \alpha^*(-W) - ( \chi_1 \alpha_1 + \chi_2 \alpha_2 ).
    \end{align}
    In particular, a necessary condition for linear stability is $- \alpha^*(W) < \chi_i \alpha_i < \alpha^*(W)$ for $i=1,2$.
\end{proposition}
\begin{remark}\label{remark:local_stability_remark} \ 
\begin{itemize}
    \item In Proposition \ref{prop:local_stability}, we identified a necessary condition for linear stability to be $- \alpha^*(-W) < \chi_i \alpha_i < \alpha^*(W)$ for $i=1,2$, which is equivalent to requiring linear stability of the associated scalar equations with interaction kernel $W$ or $-W$. This says that for the two-species case, the $\alpha_i$'s cannot interact in a way that increases the size of the stability region of the scalar case.

    \item A depiction of this region can be found in Figure \ref{fig:stability_region_1} for several values of $\gamma$. When $\gamma = 0$, the stability region is the entire rectangle $(-\alpha^*(-W), \alpha^*(W)) \times (-\alpha^*(-W), \alpha^*(W))$. As $\gamma$ increases, the stability region shrinks (i.e., the progressively darker regions shrink). The quantity $S^*$ defined in Proposition \ref{prop:local_stability} gives a dividing line between the upper and lower curves that define this region.
    
    \item The quantity $S^*$ appears several times throughout subsequent results, and so we define it explicitly here. It is useful to note that the line defined by $S^*$ has negative intercepts if and only if $\alpha^*(W) < \alpha^*(-W)$; one may compare Figure \ref{fig:stability_region_1} (repulsive top-hat kernel) with the top panel of Figure \ref{fig:system_bif_alpha1_adhesion_case} (attractive top-hat kernel), where the second case has positive intercepts.

    \item Importantly, the linear stability conditions of Proposition \ref{prop:local_stability} yield necessary and sufficient conditions for the existence of a point of critical stability for the two-species system (see Theorems \ref{thm:main_results_local_stability_alpha_system} and \ref{thm:main_results_local_stability_gamma_system}).
\end{itemize}
\end{remark}

\begin{figure}[ht]
    \centering
    \includegraphics[width=0.95\linewidth]{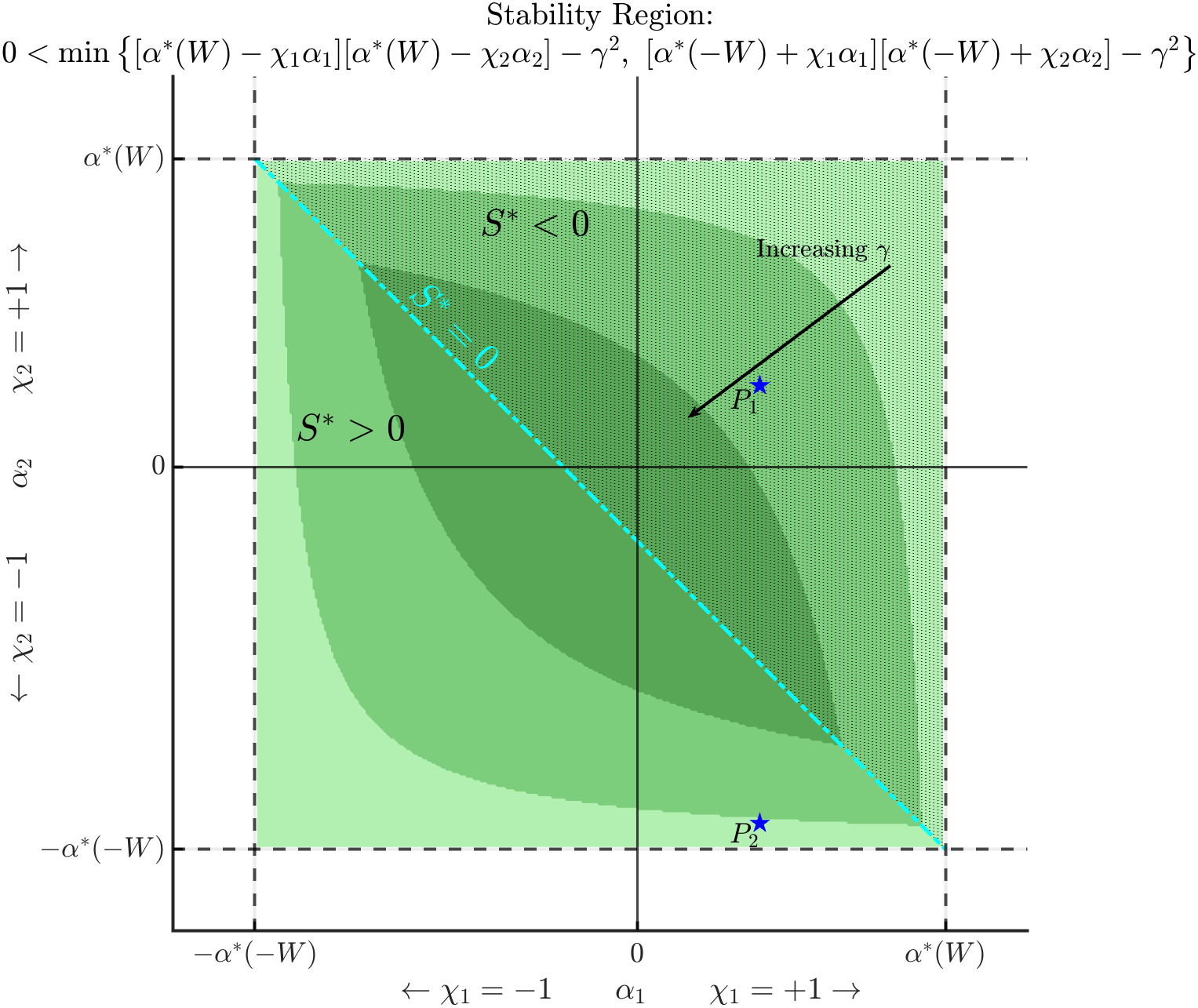}
    \caption{A visualisation of the stability region from Proposition \ref{prop:local_stability} in the $(\chi_1 \alpha_1, \chi_2 \alpha_2)$-plane for $\gamma=0$ (lightest shade of green), $\gamma = 1.5$ (darker shade of green), and $\gamma = 3.0$ (darkest shade of green). The line $S^*=0$ (the cyan-colored line) is defined as in \ref{eq:S_star}; it divides the stability region according to where the minimum of \eqref{eq:gamma_minimum_1} is achieved. The speckled area corresponds to the region $S^*<0$, while the untextured area corresponds to the region $S^* > 0$. For a fixed value of $\gamma>0$, those points $(\chi_1 \alpha_1, \chi_2 \alpha_2)$ falling within these shaded regions correspond with local asymptotic stability of the homogeneous solution. Whenever there holds $\alpha^*(\pm W) < \infty$, the stability region vanishes for $\gamma$ sufficiently large. Bifurcation points then occur at the boundary of these regions. The points $P_1$, $P_2$ are reference points for subsequent bifurcation diagrams. See Remark \ref{remark:local_stability_remark} for further discussion.}
    \label{fig:stability_region_1}
\end{figure}

We now discuss the case of bifurcation from the homogeneous state with respect to $\alpha_1$, as this is most readily connected to the bifurcation analysis of the scalar case. We first establish the existence of bifurcation points, an analogue of Theorem \ref{thm:local_bifurcations_scalar} for the scalar case.

\begin{theorem}[Description of local bifurcations w.r.t. $\alpha_1 \geq 0$]\label{thm:bifurcations_alpha1_1} Assume Hypothesis \textbf{\ref{hyp:kernel_shape}} holds. Fix $\sigma, L, \gamma > 0$, and $\alpha_2 \geq 0$. Denote by $(\mathbf{u}_\infty, \alpha_1) = (L^{-1}, L^{-1},\alpha_1)$ the homogeneous solution branch. Define $\alpha_{1,k}$ by
\begin{align}\label{eq:thm_alpha_crit_1}
    \alpha_{1,k} := - \chi_1 \left(  h_k - \frac{\gamma^2}{(h_k + \chi_2 \alpha_2) } \right).
\end{align}
Then, every $k^* \geq 1$ such that 
\begin{enumerate}[i.)]
    \item $\card \{ k \in \mathbb{N}: \alpha_{1,k} = \alpha_{1,k^*} \} = 1$,
    \item $h_{k^*} + \chi_2 \alpha_2 \neq 0$,
    \item $\alpha_{1,k^*} > 0$,
\end{enumerate}
leads to a bifurcation point $(\mathbf{u}_\infty, \alpha_{1,k^*})$ of system \eqref{eq:general_system_SS}. The emergent branch at $\alpha_{1,k^*}$ is of the form $(u_1(s), u_2(s), \alpha_1(s))$ with $s \in (-\delta, \delta)$ for some $\delta>0$, where $\alpha_1 (0) = \alpha_{1,k^*}$, $\alpha_1^\prime (0) = 0$, and $\alpha_1 ^{\prime \prime}(0) \neq 0$. In particular, the bifurcation is of pitchfork type, and the emergent branch takes the form
\begin{align}
    (u_1, u_2) = \begin{pmatrix}
        L^{-1} \\ L^{-1}
    \end{pmatrix} + s \begin{pmatrix}
    1 \\ c_{\alpha_{1,k^*}}
\end{pmatrix} \, w_{k^*} (x) + o(s) \begin{pmatrix}
    1 \\ 1
\end{pmatrix}, \quad \alpha_1 (s) = \alpha_{1,k^*} + \alpha_{1,k^*}^{\prime \prime} (0) \frac{s^2}{2} + o(s^2),
\end{align}
where $w_{k^*}$ is a basis element as defined in \eqref{eq:cos_transform_basis_function}, the coefficient $c_{\alpha_{1,k^*}}$ is given by
\begin{align}\label{eq:phase_relation_alpha_1_case}
    c_{\alpha_{1,k^*}} = - \frac{\gamma}{h_{k^*} + \chi_2 \alpha_2},
\end{align}
and $\alpha_1^{\prime \prime}(0)$ is given by
\begin{align}\label{eq:alpha_1_primeprime}
   \alpha_{1,k^*}^{\prime \prime}(0) =&\  - \frac{\chi_1 L h_{k^*}}{2 ( 1 + c_{\alpha_{1,k^*}}^2 )} \left[ 1 + c_{\alpha_{1,k^*}}^4 + \frac{( \delta_1 \chi_1 \alpha_{1,k^*} + \delta_2 \gamma + c_{\alpha_{1,k^*}}^2 (\delta_1 \gamma + \delta_2 \chi_2 \alpha_2)  )}{\det (M)} \right],
\end{align}
where $\det (M) = (1 + \chi_1 \alpha_{1,k^*} / h_{2k^*})(1 + \chi_2 \alpha_{2} / h_{2k^*}) - \gamma^2 / h_{2k^*}^2 \neq 0$, and
\begin{align}
    \begin{pmatrix}
    \delta_1 \\ \delta_2
\end{pmatrix}
:=
\begin{pmatrix}
    1 + ( \chi_2 \alpha_2 - c_{\alpha_{1,k^*}}^2 \gamma ) / h_{2k} \\
    c_{\alpha_{1,k^*}}^2 [ 1 + ( \chi_1 \alpha_1 - c_{\alpha_{1,k^*}}^{-2}\gamma ) / h_{2k} ] 
\end{pmatrix} .
\end{align} 
When $\alpha^{\prime \prime}_{1,k^*} (0) > 0$ $(< 0)$, the bifurcation is supercritical (subcritical). In particular, when $\as{\widetilde W(2k^*)} \ll 1$ the bifurcation at $k^*$ is a supercritical pitchfork bifurcation (subcritical pitchfork bifurcation) if $\chi_1 h_{k^*} <0$ ($\chi_1 h_{k^*} > 0$).
\end{theorem}

Finally, we can determine the stability properties of the first point of bifurcation, at least near the bifurcation point. What is essential is that the spectrum of the linearised system lies to the left of the complex axis, and the first critical bifurcation point passes through as a simple eigenvalue. 
Using Proposition \ref{prop:local_stability} and Theorem \ref{thm:bifurcations_alpha1_1}, we prove the following.

\begin{theorem}[Point of critical stability \& stability exchange, $\alpha_1$ case]\label{thm:main_results_local_stability_alpha_system}
    Fix $\sigma, L > 0$ and assume Hypothesis \textbf{\ref{hyp:kernel_shape}} holds. Fix $(\chi_2 \alpha_2, \gamma)$ so that
    \begin{align}\label{eq:alpha_bif_thm_2_cond_1}
    (\chi_2 \alpha_2, \gamma) \in (-\alpha^*(-W), \alpha^*(W)) \times (0, [\alpha^*(W) + \alpha^*(-W)]/2 ),
    \end{align}
    where $\alpha^*(\pm W)$ is the point of critical stability for the scalar case as defined in \eqref{critical_alpha_positive} for the kernel $\pm W$, with associated wavenumbers $k_{\pm W}$. Recall also $\alpha_{1,k}$ as defined in Theorem \ref{thm:bifurcations_alpha1_1}. Then, for any such $(\chi_2 \alpha_2, \gamma)$, the region of linear stability is non-empty, and we have the following three cases.
    \begin{enumerate}
        \item Suppose $\alpha_{1,k_{-W}} < 0 < \alpha_{1,k_W}$. Then, there exists $\alpha_1^* > 0$ so that $\mathbf{u}_\infty$ is locally asymptotically stable for all $\alpha_1 \in [0, \alpha_1^*)$ and is (nonlinearly) unstable for all $\alpha_1 > \alpha_1^*$. Moreover, $\alpha_1^*$ is given by
    \begin{align}
        \alpha_1^* = \begin{cases}
            \alpha_{1,k_W} \quad \text{ when } \chi_1 = +1; \cr
            -\alpha_{1,k_{-W}}, \quad \text{ when } \chi_1 = -1.
        \end{cases},
    \end{align}
    Consequently, under the assumptions of Theorem \ref{thm:bifurcations_alpha1_1}, the first bifurcation point is supercritical whenever $\alpha^{\prime \prime}_{k_{\pm W}} (0) > 0$, and the Principle of Exchange of Stability holds: $\mathbf{u}_\infty$ loses stability at $\alpha_1 = \alpha_1^*$, and the emergent branch is locally asymptotically stable. In particular, when $\as{\widetilde W(2 k_{\pm W})} \ll 1$, the bifurcation is always supercritical, and an exchange of stability occurs. When $\alpha^{\prime \prime}_{k_{\pm W}} (0) < 0$, the bifurcation is subcritical and no exchange of stability occurs. Furthermore, when $\chi_1 = +1$ (i.e., when bifurcation occurs at $k_{ W}$), the solution components are \textup{in phase}; when $\chi_1 = -1$ (i.e., when bifurcation occurs at $k_{-W}$), the solution components are \textup{out of phase}.
    
    \item Suppose $0< \alpha_{1,k_{-W}} < \alpha_{1,k_W}$. Then, when $\chi_1 = -1$, there is no point of critical stability and $\mathbf{u}_\infty$ is unstable for all $\alpha_1 \geq 0$. When $\chi_1 = +1$, $\mathbf{u_\infty}$ is locally asympototically stable for all $ \alpha_1 \in ( \alpha_{1,k_{-W}} , \alpha_{1,k_W} )$. Then, there are two points of critical stability given by $\alpha_{1,k_{-W}}$ and $\alpha_{1,k_{W}}$, and the following hold:
    \begin{itemize}
        \item at $\alpha_1 = \alpha_{1,k_W}$, there is a supercritical bifurcation when $\alpha_{1,k_W}^{\prime \prime}(0) > 0$ and the Principle of Exchange of Stability holds: $\mathbf{u}_\infty$ loses stability at $\alpha_1 = \alpha_{1,k_W}$, and the emergent branch is locally asymptotically stable for $\alpha_1 \in (\alpha_{1,k_W}, \alpha_{1,k_W} + \delta_0)$ for some $\delta_0>0$. In particular, when $\as{\widetilde W(2 k_W)} \ll 1$,  $\alpha_{1,k_W}$ is always supercritical and an exchange of stability occurs. Otherwise, the bifurcation is subcritical and no exchange of stability occurs. Furthermore, the solution components of the emergent branch are \textup{in phase}.
        \item at $\alpha_1 = \alpha_{1,k_{-W}}$, there is a subcritical bifurcation when $\alpha_{1,k_{-W}}^{\prime \prime}(0) < 0$ and the Principle of Exchange of Stability holds: $\mathbf{u}_\infty$ loses stability at $\alpha_1 = \alpha_{1,k_{-W}}$, and the emergent branch is locally asymptotically stable for $\alpha_1 \in (\alpha_{1,k_{-W}} - \delta_0, \alpha_{1,k_{-W}})$ for some $\delta_0>0$. In particular, when $\as{\widetilde W(2 k_{-W})} \ll 1$,  $\alpha_{1,k_{-W}}$ is always subcritical and an exchange of stability occurs. Otherwise, the bifurcation is supercritical and no exchange of stability occurs. Furthermore, the solution components of the emergent branch are \textup{out of phase}.
    \end{itemize}

    \item Suppose $\alpha_{1,k_{-W}} < \alpha_{1,k_W} < 0$. Then, Case 2. is reversed: when $\chi_1 = +1$, there is no point of critical stability and $\mathbf{u}_\infty$ is unstable for all $\alpha_1 \geq 0$. When $\chi_1 = -1$, $\mathbf{u_\infty}$ is locally asymptotically stable for all $\alpha_1 \in (- \alpha_{1,k_W}, - \alpha_{1,k_{-W}})$. Then, there are two points of critical stability $-\alpha_{1,k_{-W}}$ and $-\alpha_{1,k_{W}}$, and there holds 
    \begin{itemize}
        \item at $\alpha_1 = -\alpha_{1,k_{-W}}$, there is a supercritical bifurcation when $-\alpha^{\prime \prime}_{1, k_{-W}} (0) > 0$ and the Principle of Exchange of Stability holds: $\mathbf{u}_\infty$ loses stability at $\alpha_1 = -\alpha_{1,k_{-W}}$, and the emergent branch is locally asymptotically stable for $\alpha_1 \in (-\alpha_{1,k_{-W}}, -\alpha_{1,k_{-W}} + \delta_0)$ for some $\delta_0>0$. In particular, when $\as{\widetilde W(2 k_{-W})} \ll 1$,  $-\alpha_{1,k_{-W}}$ is always supercritical and an exchange of stability occurs. Otherwise, the bifurcation is subcritical and no exchange of stability occurs. Furthermore, the solution components of the emergent branch are \textup{in phase}.
        \item at $\alpha_1 = -\alpha_{1,k_W}$, there is a subcritical bifurcation when $- \alpha^{\prime \prime}_{1,k_W} < 0$ and the Principle of Exchange of Stability holds: $\mathbf{u}_\infty$ loses stability at $\alpha_1 = -\alpha_{1,k_{-W}}$, and the emergent branch is locally asymptotically stable for $\alpha_1 \in ( -\alpha_{1,k_{-W}} - \delta_0, -\alpha_{1,k_{-W}} )$ for some $\delta_0>0$. In particular, when $\as{\widetilde W(2 k_{W})} \ll 1$,  $-\alpha_{1,k_{W}}$ is always subcritical and an exchange of stability occurs. Otherwise, the bifurcation is supercritical and no exchange of stability occurs. Furthermore, the solution components of the emergent branch are \textup{out of phase}.
    \end{itemize}
 \end{enumerate}

    Finally, if \eqref{eq:alpha_bif_thm_2_cond_1} is violated, no point of critical stability exists and $\mathbf{u}_\infty$ is unstable for all $\alpha_1 \geq 0$.
\end{theorem}

\begin{remark}\label{remark:alpha1_bifurcations_remark}\
    \begin{itemize}
        \item When $\gamma = 0$, one finds that all quantities of Theorem \ref{thm:bifurcations_alpha1_1} reduce to the scalar case of Theorem \ref{thm:local_bifurcations_scalar}.

        \item Different from the scalar case, subcritical bifurcations are now possible, even when $\widetilde W(2k) = 0$. In particular, when $\as{\widetilde W(2k)} \ll 1$ there is no longer a correspondence between the sign of $\alpha_{1,k}$ and the sign of $h_k$: it is possible that $\alpha_{1,k} > 0$ while $h_k>0$, which is precisely the criteria for a subcritical bifurcation to occur when $\chi_1 = +1$. In fact, any wavenumber $k \in \mathcal{K}^+$ such that $\alpha_{1,k}>0$ will produce a subcritical bifurcation.
        
        \item We visualise Theorem \ref{thm:bifurcations_alpha1_1} and two cases of Theorem \ref{thm:main_results_local_stability_alpha_system} in Figure \ref{fig:system_bif_alpha1}. For this example, we again fix $W$ and all parameters as described in Remark \ref{remark:scalar_bifurcation_remark} for the scalar example depicted in Figure \ref{fig:scalar_bif_1}. We then fix $\chi_1 = +1$ and consider two points $(\chi_1 \alpha_1, \chi_2 \alpha_2, \gamma)$ in parameter space:
        \begin{align*}
            P_1 = (1.5,1.0, 1.5), \quad \quad \quad P_2 = (1.5,-4.35, 1.5).
        \end{align*}
        These are depicted as blue stars in Figure \ref{fig:stability_region_1} in the $(\chi_1 \alpha_1, \chi_2 \alpha_2)$-plane. (Note carefully that $\chi_1 \alpha_1$ does not play a role at this stage, as we are considering $\alpha_1$ as the bifurcation parameter; the value of $\chi_1 \alpha_1$ will be relevant when discussing bifurcations with respect to $\gamma$). Notice that $P_1$ lies within the intermediate green region, while $P_2$ lies outside of the intermediate green region.
        
        \item In the left panels of Figure \ref{fig:system_bif_alpha1}, we plot several Fourier coefficients of the kernel $W$ (the same as in the left panel of Figure \ref{fig:scalar_bif_1} for the scalar case), along with the associated candidate bifurcation points of the scalar case with respect to $\alpha$ (i.e., the $\alpha_k$'s defined in Theorem \ref{thm:local_bifurcations_scalar}), and the associated candidate bifurcation points of the two-species system with respect to $\alpha_1$ (i.e., the $\alpha_{1,k}$'s defined in Theorem \ref{thm:bifurcations_alpha1_1}). The top-left panel corresponds to the point $P_1$, while the bottom-left panel corresponds to the point $P_2$.
        
        \item The top-right and bottom-right panels of Figure \ref{fig:system_bif_alpha1} depict the associated bifurcation diagrams for the two-species case with respect to $\alpha_1$: the top-right panel corresponds with the point $P_1$, while the bottom-right panel corresponds with the point $P_2$.
        
        \item For bifurcation from the point $P_1$, it falls into Case 1. of Theorem \ref{thm:main_results_local_stability_alpha_system}, and so the behaviour is almost identical to the bifurcation behaviour observed in Figure \ref{fig:scalar_bif_1} for the scalar case: both branches displayed are supercritical, and an exchange of stability occurs at the first branch. At both bifurcation points depicted, we compute $c_{\alpha_{1,k}} > 0$, and so the components $u_i^*(x)$ are in phase with each other. 
        
        \item For bifurcation from the point $P_2$, the behaviour changes significantly according to Theorem \ref{thm:main_results_local_stability_gamma_system}. Since the point $P_2$ lies outside of the stability region for $\gamma = 1.5$ as depicted in Figure \ref{fig:stability_region_1}, the homogeneous state is unstable for \textit{any} $\alpha_1 \geq 0$; therefore, no exchange of stability occurs. However, from Theorem \ref{thm:bifurcations_alpha1_1}, we still describe several of the emergent branches: the first branch is a subcritical bifurcation, and we calculate $c_{\alpha_{1,2}} < 0$ so that the solution components are out of phase with each other. The second two branches are found to be supercritical with $c_{\alpha_{1,1}}$ and $c_{\alpha_{1,3}}$ both positive, from which we conclude the solution components are in phase with each other.
    \end{itemize}
\end{remark}

\begin{figure}[ht]
    \centering
    \includegraphics[width=0.9\linewidth]{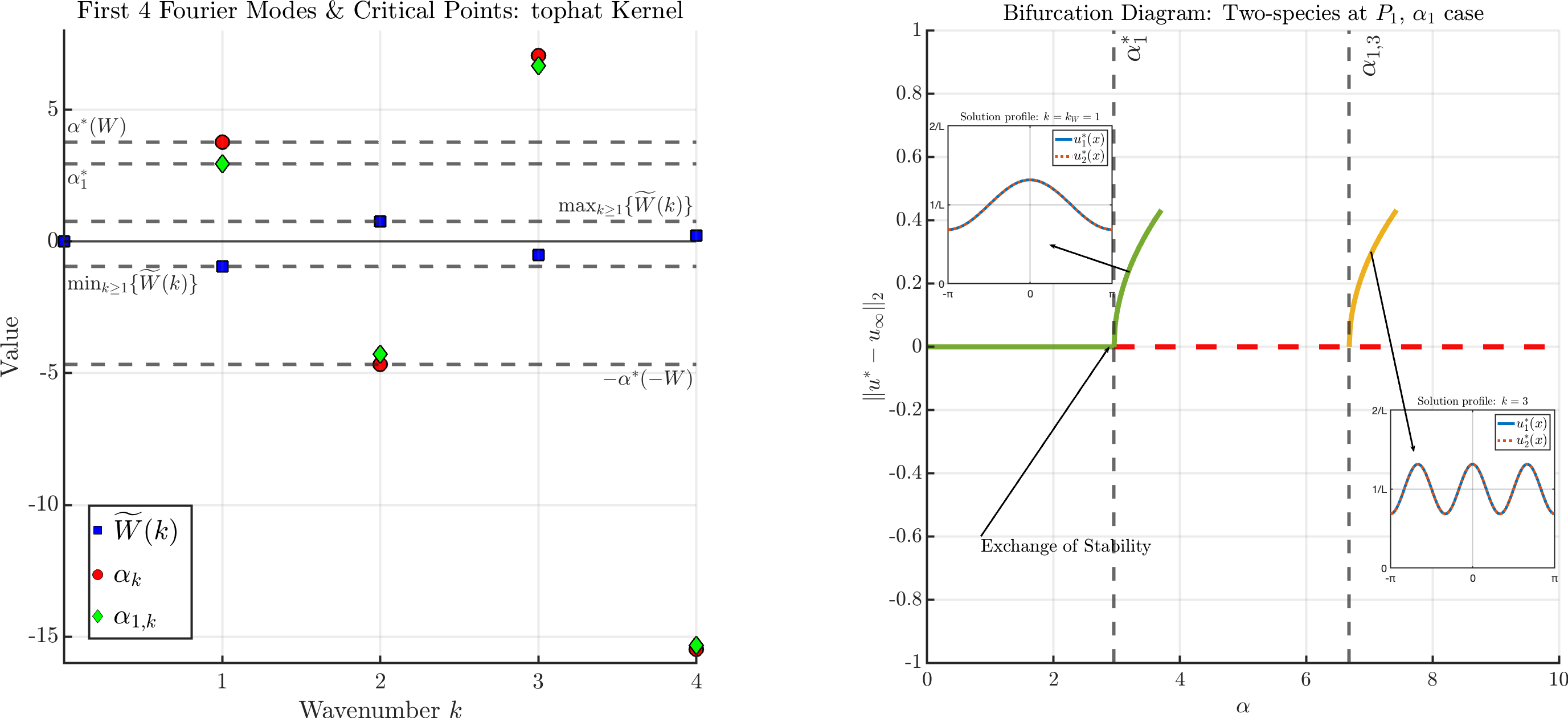} \\
    \includegraphics[width=0.9\linewidth]{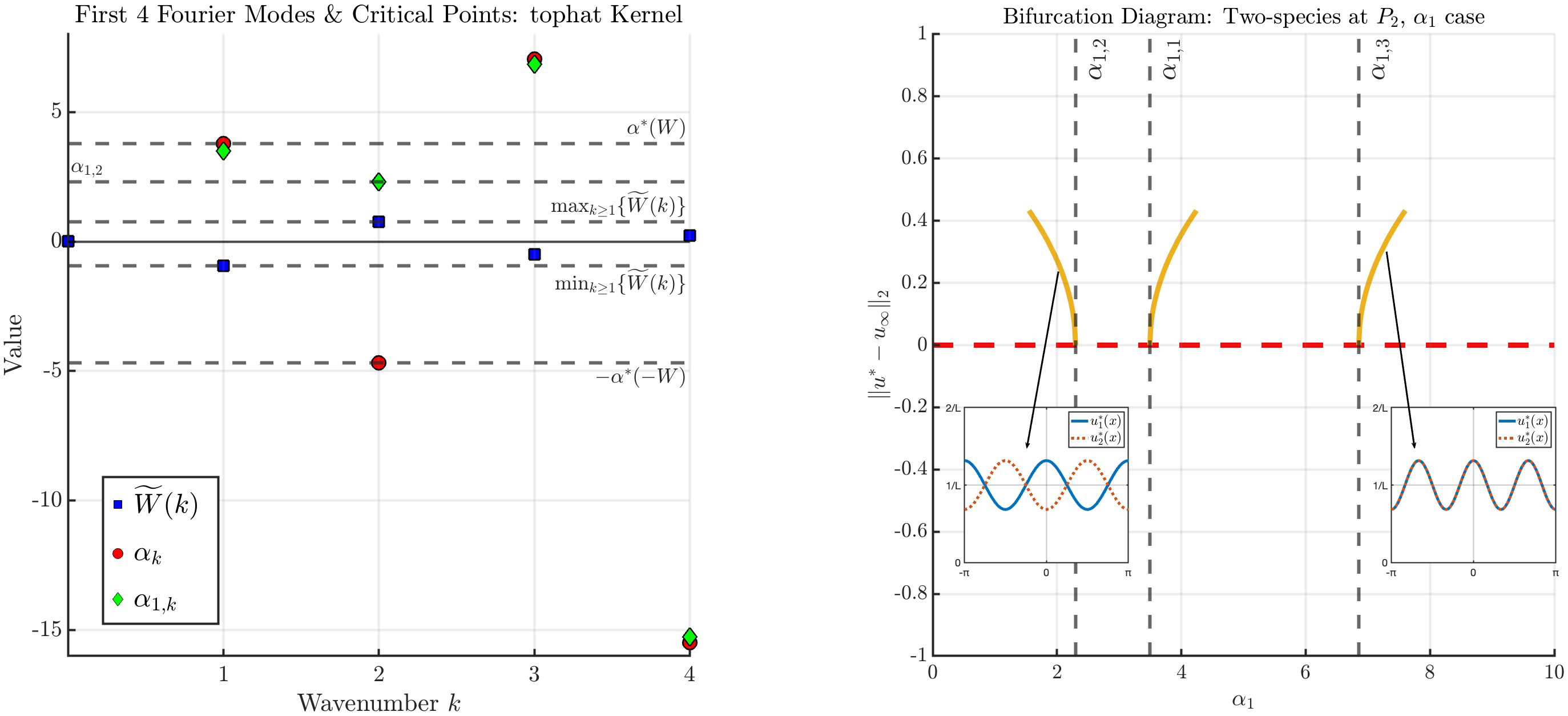}     
    \caption{A visualisation of the results of Theorems \ref{thm:bifurcations_alpha1_1} and \ref{thm:main_results_local_stability_alpha_system} for the two-species system. The top row corresponds with the point $P_1$, while the bottom row corresponds with the point $P_2$, where $P_i$ are marked as blue stars in Figure \ref{fig:stability_region_1}. Green lines denote a stable branch, dashed red lines denote an unstable branch, and the stability of the yellow branches is unknown. Further details are provided in Remark \ref{remark:alpha1_bifurcations_remark}.}
    \label{fig:system_bif_alpha1}
\end{figure}

We now describe bifurcations with respect to $\gamma$. 

\begin{theorem}[Description of local bifurcations w.r.t. $\gamma \geq 0$]\label{thm:bifurcations_gamma_1} Assume Hypothesis \textbf{\textup{\ref{hyp:kernel_shape}}} holds. Fix $\sigma, L >0$, and $\alpha_i \geq 0$, $i=1,2$. Denote by $(\mathbf{u}_\infty, \alpha_1) = (L^{-1}, L^{-1},\gamma)$ the homogeneous solution branch. We distinguish two cases.
\begin{enumerate}
    \item Every wavenumber $k \geq 1$ such that $\sign(h_{k} + \chi_1 \alpha_1) \neq \sign(h_{k} + \chi_2 \alpha_2)$ does not lead to a bifurcation point with respect to $\gamma$.
    \item  Suppose $k \geq 1$ such that $\sign(h_{k} + \chi_1 \alpha_1) = \sign(h_{k} + \chi_2 \alpha_2)$ so that we may define $\gamma_k(\chi_1, \chi_2)$ by
\begin{align}\label{eq:critical_gamma}
    \gamma_k := \sqrt{\left( h_k + \chi_1 \alpha_1 \right) \left( h_k + \chi_2 \alpha_2 \right)} \in \mathbb{R}^+.
\end{align}
Then, every $k^* \geq 1$ such that 
\begin{enumerate}
    \item $\card \{ k \in \mathbb{N}: \gamma_k = \gamma_{k^*} \} = 1$,
    \item $h_{k^*} + \chi_i \alpha_i \neq 0$, $i=1,2$,
\end{enumerate}
leads to a bifurcation point $(\mathbf{u}_\infty, \gamma_{k^*})$ of system \eqref{eq:general_system_SS}. The emergent branch at $\gamma = \gamma_{k^*}$ is of the form $(u_1(s), u_2(s), \gamma(s))$ with $s \in (-\delta, \delta)$ for some $\delta>0$, where $\gamma (0) = \gamma_{k^*}$, $\gamma^\prime  (0) = 0$, and $\gamma ^{\prime \prime}(0) \neq 0$. In particular, the bifurcation is of pitchfork type, and the emergent branch takes the form
\begin{align}
    (u_1, u_2) = \begin{pmatrix}
        L^{-1} \\ L^{-1}
    \end{pmatrix} + s \begin{pmatrix}
    1 \\ c_{\gamma_{k^*}}
\end{pmatrix} \, w_{k^*} (x) + o(s) \begin{pmatrix}
    1 \\ 1
\end{pmatrix}, \quad \gamma (s) = \gamma_{k^*} + \gamma_{k^*}^{\prime \prime} (0) \frac{s^2}{2} + o(s^2),
\end{align}
where $w_{k^*}$ is as defined in \eqref{eq:cos_transform_basis_function}, the coefficient $c_{\gamma_{k^*}}$ is given by
\begin{align}
    c_{\gamma_{k^*}} = - \sign(h_{k^*} + \chi_1 \alpha_1) \sqrt{\frac{(h_{k^*} + \chi_1 \alpha_1)}{(h_{k^*} + \chi_2 \alpha_2)}},
\end{align}
and $\gamma^{\prime \prime}(0)$ is given by
\begin{align}
    \gamma_{k^*}^{\prime \prime}(0) = - \frac{L c_0^2 h_{k^*}}{4 c_{\gamma_{k^*}}} \left[ 1 + c_{\gamma_{k^*}}^4 + \frac{( \delta_1 \chi_1 \alpha_{1} + \delta_2 \gamma_{k^*} + c_{\gamma_{k^*}}^2 (\delta_1 \gamma_{k^*} + \delta_2 \chi_2 \alpha_2)  )}{\det (M)} \right]
\end{align}
where $\det (M) = (1 + \chi_1 \alpha_{1} / h_{2k^*})(1 + \chi_2 \alpha_{2} / h_{2k^*}) - \gamma_{k^*}^2 / h_{2k^*}^2 \neq 0$, and
\begin{align}
    \begin{pmatrix}
    \delta_1 \\ \delta_2
\end{pmatrix}
:=
\begin{pmatrix}
    1 + ( \chi_2 \alpha_2 - c_{\gamma_{k^*}}^2 \gamma_{k^*} ) / h_{2k^*} \\
    c_{\gamma_{k^*}}^2 [ 1 + ( \chi_1 \alpha_1 - c_{\gamma_{k^*}}^{-2}\gamma_{k^*} ) / h_{2k^*} ] 
\end{pmatrix} .
\end{align} 
When $\gamma^{\prime \prime}_{k^*} (0) > 0$ $(< 0)$, the bifurcation is supercritical (subcritical).
In particular, when $\as{\widetilde W(2k^*)} \ll 1$, the bifurcation at $k^*$ is a supercritical pitchfork bifurcation (subcritical pitchfork bifurcation) if $\sign (h_{k^*}) = \sign(h_{k^*} + \chi_i \alpha_i)$ ($\sign (h_{k^*}) \neq \sign(h_{k^*} + \chi_i \alpha_i)$).
\end{enumerate}
\end{theorem}

\begin{theorem}[Point of critical stability \& stability exchange, $\gamma$ case]\label{thm:main_results_local_stability_gamma_system}
       Fix $\sigma, L > 0$ and assume Hypothesis \textbf{\ref{hyp:kernel_shape}} holds. Fix $\chi_1 \alpha_1, \chi_2 \alpha_2$ such that the following holds:
       \begin{align}\label{eq:gamma_bif_thm_2_cond_1}
     -\alpha^*(-W) < \chi_i \alpha_i < \alpha^*(W), \quad i=1,2.
       \end{align}
       Then, for any $(\chi_1 \alpha_1, \chi_2 \alpha_2)$ satisfying \eqref{eq:gamma_bif_thm_2_cond_1}, there exists a point of critical stability $\gamma^* > 0$ so that $\mathbf{u}_\infty$ is locally asymptotically stable for all $\gamma \in [0, \gamma^*)$ and is unstable for all $\gamma > \gamma^*$. Moreover, $\gamma^* = \gamma^*(\chi_1 \alpha_1, \chi_2 \alpha_2)$ is given by
       $$
\gamma^* = \min_{k} \{ \gamma_k : \gamma_k \in \mathbb{R} \} =  \begin{cases}
    \gamma_{k_W} \quad \text{ whenever } S^* < 0, \cr 
    \gamma_{k_{-W}} \quad \text{ whenever } S^* > 0.
\end{cases},
       $$
       where $\gamma_k$ is as defined in Theorem \ref{thm:bifurcations_gamma_1}, and $S^*$ is as defined in Proposition \ref{prop:local_stability}.
       
    If the bifurcation at $\gamma = \gamma^*$ is supercritical, the Principle of Exchange of Stability holds: $\mathbf{u}_\infty$ loses stability at $\gamma = \gamma^*$, and the emergent branch is locally asymptotically stable; otherwise, the bifurcation is subcritical and no exchange of stability occurs. When $\gamma^*$ occurs at $k=k_{ W}$, the solution components are in phase; when $\gamma^*$ occurs at $k=k_{-W}$, the solution components are out of phase. 
    
    In particular, there necessarily holds $\sign(h_{k_W}) = \sign (h_{k_W} + \chi_i \alpha_i)$ and $\sign(h_{k_{-W}}) = \sign (h_{k_{-W}} + \chi_i \alpha_i)$ for $i=1,2$, and so under the assumptions of Theorem \ref{thm:bifurcations_gamma_1}, the bifurcation is supercritical when $\as{\widetilde W(2 k_{\pm W})} \ll 1$.

       Finally, if \eqref{eq:gamma_bif_thm_2_cond_1} is violated, no point of critical stability exists and $\mathbf{u}_\infty$ is unstable for all $\gamma > 0$.
\end{theorem}

\begin{figure}[ht]
    \centering
    \includegraphics[width=0.95\linewidth]{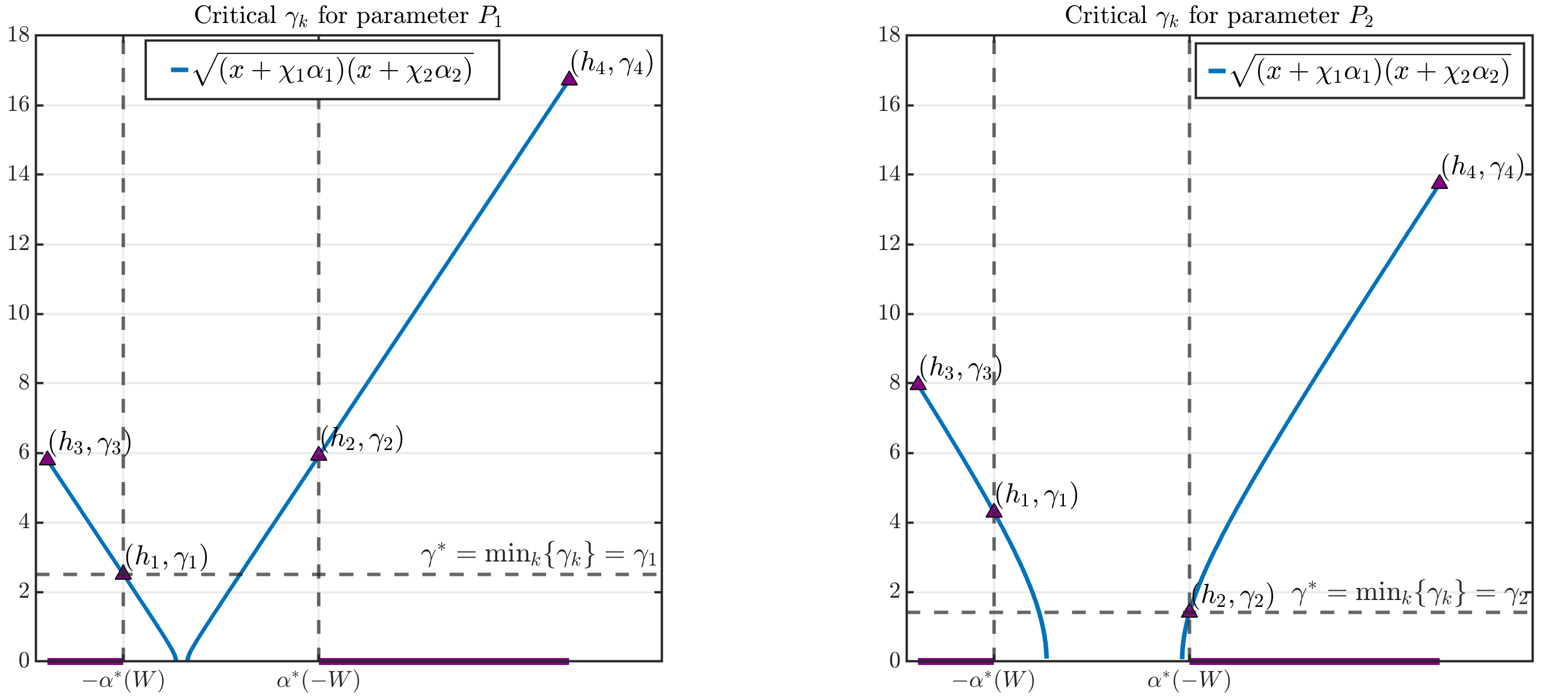}
    \caption{A visualisation of the interplay between the negative and positive Fourier modes as found in Theorem \ref{thm:main_results_local_stability_gamma_system} for the two-species system when bifurcating with respect to $\gamma$. The blue line depicts the continuous version of the bifurcation points $\gamma_k$ defined in \eqref{eq:critical_gamma}; the dark purple regions on the $x$-axis denote the possible range of the inputs, namely, $\ran (h_k) = (-\infty, -\alpha^*(W)] \cup [\alpha^*(-W), +\infty)$. The left panel uses the parameter values $P_1$, while the right panel uses the parameter values $P_2$, both of which are depicted in Figure \ref{fig:stability_region_1}. See Remark \ref{remark:gamma_bifurcation_remark} for further discussion.}
    \label{fig:system_stability_2}
\end{figure}

\begin{figure}[ht]
    \centering
    \includegraphics[width=0.9\linewidth]{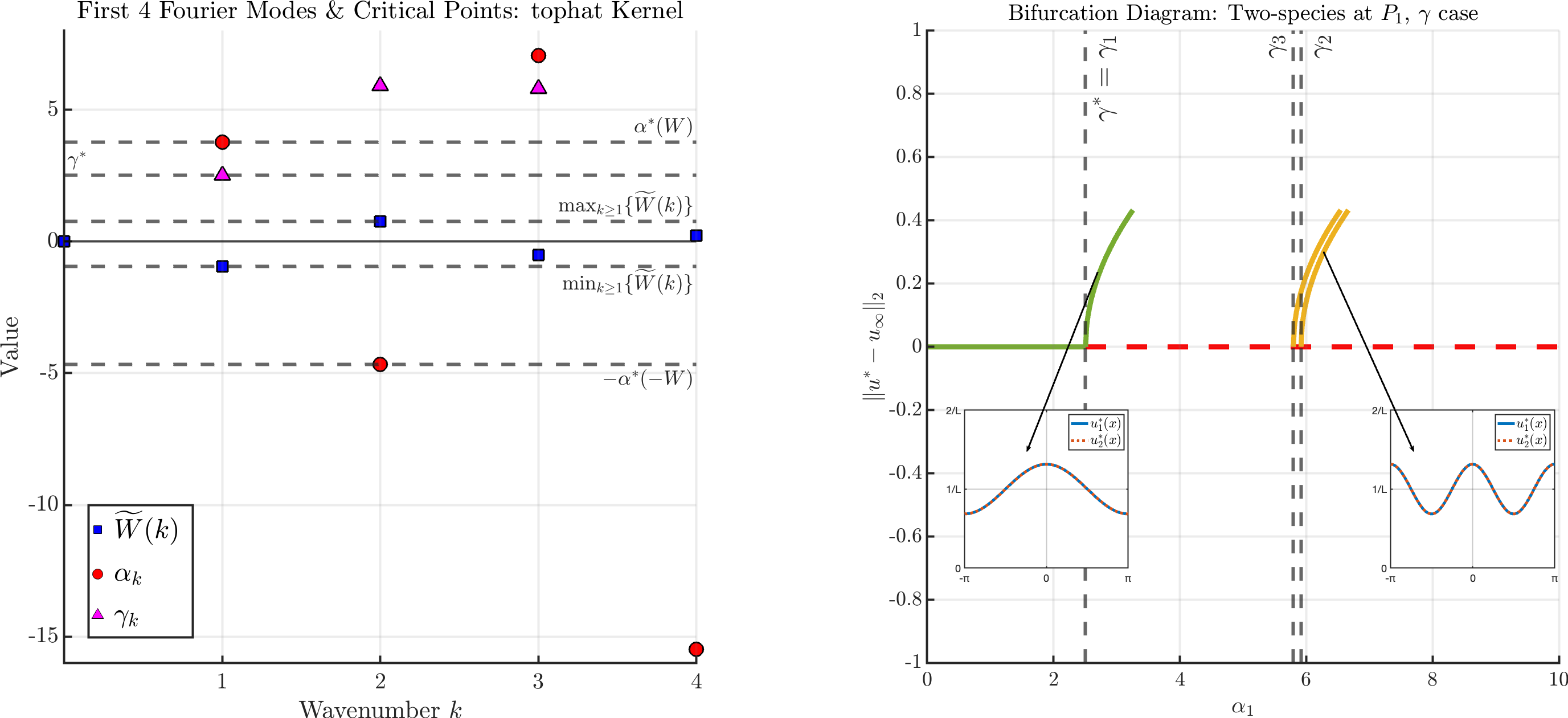} \\
    \includegraphics[width=0.9\linewidth]{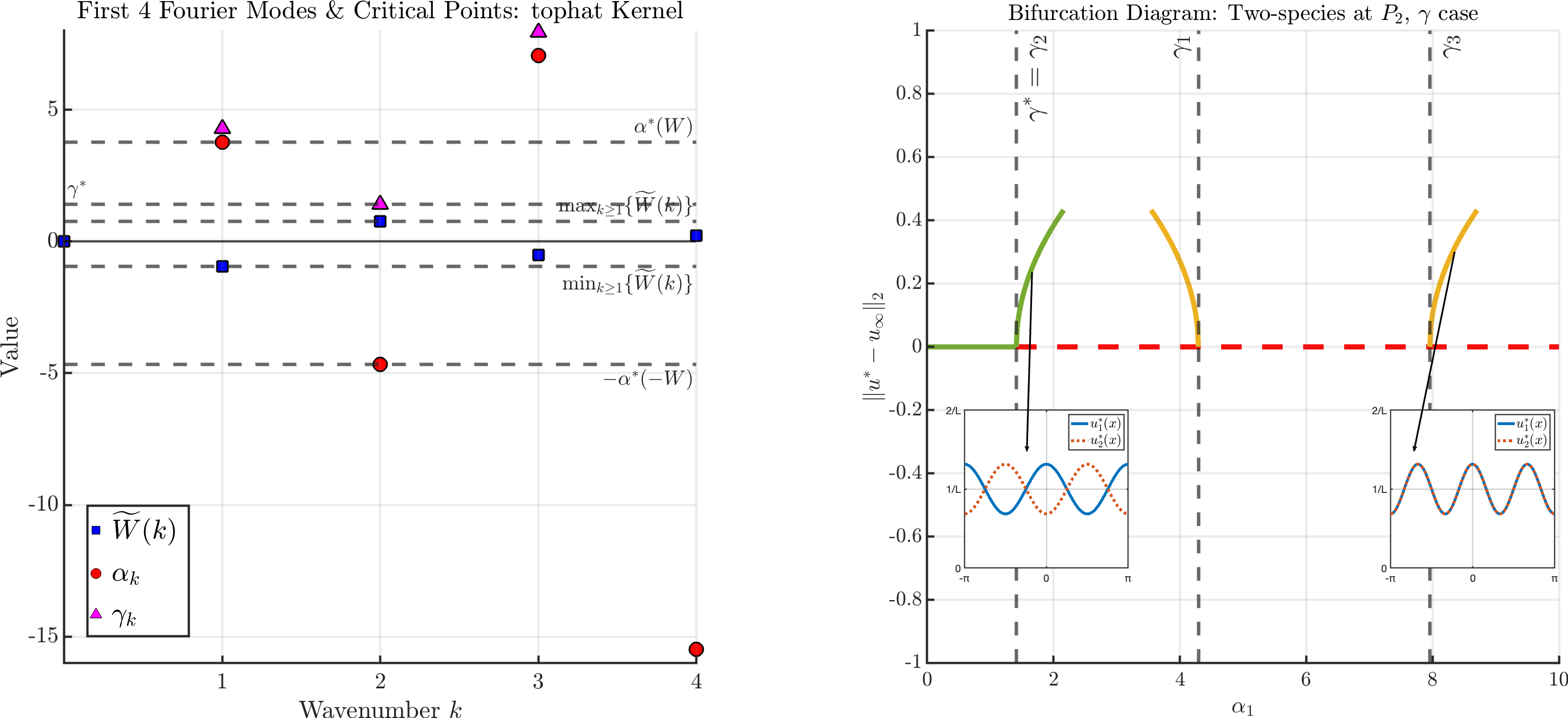}     
    \caption{A visualisation of the results of Theorems \ref{thm:bifurcations_gamma_1} and \ref{thm:main_results_local_stability_gamma_system} for the two-species system. The top and bottom row again correspond with points $P_1$ and $P_2$, respectively. Green lines denote a stable branch, dashed red lines denote an unstable branch, and the stability of the yellow branches is unknown. See Remark \ref{remark:gamma_bifurcation_remark} for further discussion.}
    \label{fig:system_bif_gamma}
\end{figure}

\begin{remark}\label{remark:gamma_bifurcation_remark}\
    \begin{itemize}
        \item In the edge case $S^* = 0$, the kernel of the linearised operator is no longer one-dimensional, and we cannot apply the theory of bifurcation from a simple eigenvalue.

        \item Assuming the stability criteria of Proposition \ref{prop:local_stability} holds, the point of critical stability $\gamma^*$ as described in Theorem \ref{thm:main_results_local_stability_gamma_system} is obtained by taking the minimum across all valid $\gamma_k$. Theorem \ref{thm:main_results_local_stability_gamma_system} tells us that in such a case, the first bifurcation must occur at wavenumber $k= k_{W}$ or $k= k_{-W}$, depending on the sign of $S^*$. This can be understood through Figure \ref{fig:stability_region_1}: if we fix $-\alpha^*(-W) < \chi_1 \alpha_1, \ \chi_2 \alpha_2 < \alpha^*(W)$, the homogeneous state is always stable for $\gamma \ll 1$. Since the stability region collapses around the line $S^*=0$ for increasing $\gamma$, as $\gamma$ increases there will be a value for which $(\chi_1 \alpha_2, \chi_2 \alpha_2)$ lies on one of the upper or lower curves that define the stability region. Then, it is easy to see that when $(\chi_1 \alpha_2, \chi_2 \alpha_2)$ lies above $S^*=0$, it much touch on the upper branch, whereas for $(\chi_1 \alpha_2, \chi_2 \alpha_2)$ lying below $S^* = 0$, it must touch on the lower branch.

        \item In Figure \ref{fig:system_stability_2}, we display the continuous version of $\gamma_k$ as defined in Theorem \ref{thm:bifurcations_gamma_1}, namely, the blue curve $f(x) = \sqrt{(x + \chi_1 \alpha_1)(x + \chi_2 \alpha_2)}$. The purple triangles correspond to the bifurcation points $\gamma_k$, obtained directly via evaluation $\gamma_k = f( h_k)$. The solid purple lines denote the valid domain of $\gamma_k$ according to the range $\ran ( h_k ) \subset (-\infty, - \alpha^*(W)] \cup [ \alpha^*(-W), \infty)$; otherwise, $f$ is complex valued. 

        \item The left panel of Figure \ref{fig:system_stability_2} shows the situation for the point $P_1$: the minimum occurs at $\gamma_1$, and so the point of critical stability occurs at wavenumber $k = k_W = 1$. The right panel of Figure \ref{fig:system_stability_2} shows the situation for the point $P_2$: the minimum now occurs at $\gamma_2$, and so the point of critical stability occurs at wavenumber $k = k_{-W} = 2$.

        \item The different behaviour between $P_1$ and $P_2$ can be understood through the quantity $S^*$ defined in Proposition \ref{prop:local_stability}: for $P_1$, there holds $S^*<0$, while for $P_2$ there holds $S^* > 0$ (see $P_1$ and $P_2$ plotted in Figure \ref{fig:stability_region_1}).

        \item The phase relationship in the $\alpha_1$ case found in Theorem \ref{thm:main_results_local_stability_alpha_system} also holds in the $\gamma$ case of Theorem \ref{thm:main_results_local_stability_gamma_system}: bifurcation at $k=k_W$ leads to in-phase solution components, while bifurcation at $k=k_{-W}$ leads to out-of-phase solution components.

        \item For generic bifurcation points identified in Theorem \ref{thm:bifurcations_gamma_1}, the behaviour is similar to the $\alpha_1$ case: bifurcations can be either sub- or supercritical, and the emergent solution components may be in or out of phase with each other. We display such cases in Figure \ref{fig:system_bif_gamma}: the top panels correspond to the point $P_1$, while the bottom panels correspond to the point $P_2$. 

        \item The top-right panel of Figure \ref{fig:system_bif_gamma} shows an exchange of stability for the $\gamma$ case from the point $P_1$, where the emergent branch occurs at wavenumber $k=1$ and the solution components are in phase (i.e., $c_{\gamma_1} > 0$).

        \item The top-left panel of Figure \ref{fig:system_bif_gamma} again shows an exchange of stability for the $\gamma$ case, now from the point $P_2$, so that the emergent branch occurs at wavenumber $k=2$, and the solution components are out of phase (i.e., $c_{\gamma_2} < 0$).
    \end{itemize}
\end{remark}

\subsection{The Differential Adhesion Hypothesis \& an application to cell-cell adhesion}\label{subsec:application}

The points $P_1$ and $P_2$ found in Figure \ref{fig:stability_region_1}, then used as test cases in Figures \ref{fig:system_bif_alpha1} and \ref{fig:system_bif_gamma}, were chosen arbitrarily to display some of the possible behaviour predicted by our bifurcation analysis. In this section, we consider a more carefully constructed example with a direct connection with cell-cell adhesion and the so-called Differential Adhesion Hypothesis \cite{foty2005differential, carrillo2019adhesion}. The key aspect the we wish to highlight is the following unintuitive experimental result: in a purely adhesive system, where two populations of cells adhere within and across populations, the two cell groups may spontaneously arrange themselves into a patterned state, and the cell densities are ``out of phase'' in the sense that one cell population will aggregate in the centre, while the other cell population will surround, or ``engulf'', the first. Whether they exhibit this behaviour or not depends on the relative differences in adhesion strengths between the populations. We describe this phenomenon briefly now.

When the two cell populations have identical adhesion strengths (e.g., the pink line in the top panel of Figure \ref{fig:system_bif_alpha1_adhesion_case}), the populations will not sort themselves in the manner described above. This can be understood clearly through the results of our bifurcation analysis: when the populations are identical (so that $\gamma = \alpha_i$ as well), we can reduce the problem to the scalar case (by adding the two equations and defining a new solution variable), and bifurcation can only occur at $k_W$. Therefore, the two solution components are necessarily in phase, and no cell sorting occurs. Only when the relative difference in adhesive strengths is sufficiently large do we observe the cell sorting behaviour described earlier.

Interestingly, with the relatively simple model \eqref{eq:general_system}, it is possible to set up a purely adhesive setting while producing a segregated pattern between the two cell populations. For this example, we keep all other parameters fixed as before, but we now fix the interaction potential to be the attractive top-hat kernel, which is to say, we choose the interaction kernel $-W$, where $W$ is as defined in Remark \ref{remark:scalar_bifurcation_remark}. We then choose the radius $R = 1.25$ to exaggerate $\alpha^*(W) \gg \alpha^*(-W)$ so that the line $S^* = 0$ has positive intercepts (compare Figure \ref{fig:stability_region_1} and the top panel of Figure \ref{fig:system_bif_alpha1_adhesion_case}), noting that our intention is to demonstrate the relevant qualitative behaviour, rather than to be quantitatively accurate with our parameter choices. We then fix $\chi_1 = \chi_2 = +1$, so that we are describing an attractive-attractive-attractive regime for adhesion strengths $\alpha_i, \gamma > 0$. Note carefully that, with the kernel and parameter values chosen, the contribution from a resonant mode $2k$ is negligible and we fall into the simpler case where we may assume $\as{\widetilde W(2k)} \ll 1$.

The top panel of Figure \ref{fig:system_bif_alpha1_adhesion_case} displays the stability region for several values of $\gamma$ in the $(\alpha_1, \alpha_2)$-plane, similar to the stability region depicted in Figure \ref{fig:stability_region_1}. We then introduce a third parameter value, $P_3 = (3.5, 6.0, 8.8)$, which is shown as a blue star in the top panel of Figure \ref{fig:system_bif_alpha1_adhesion_case}. In the case where $\gamma$ is large, we then observe an example of Case 2. of Theorem \ref{thm:main_results_local_stability_alpha_system}: there are now two points of critical stability with respect to $\alpha_1$, displayed as black triangles in the top panel of Figure \ref{fig:system_bif_alpha1_adhesion_case}.

If we then increase the adhesive strength of population $u_1$, i.e., we increase $\alpha_1$, we will bifurcate at $\alpha_1 = \alpha_{1,2}$, and the solution components are in phase with each other at frequency $k = 2$. On the other hand, if we instead decrease the adhesive strength of population $u_1$, i.e,. we decrease $\alpha_1$, we bifurcate at $\alpha_1 =\alpha_{1,1}$, and the solution components are out of phase. More importantly, perhaps, this pattern is a stable one, at least near the bifurcation point; we therefore describe this as the \textit{onset of engulfment}, as the result is necessarily a local one. This is emblematic of the behaviour observed in experiments describing the segregation patterns of adhesive cell populations. We note carefully, however, that we have chosen the cross-adhesion strength $\gamma$ larger than either of the self-adhesion strengths; this is not entirely consistent with experimental designs focusing on the role of cadherin in cell-cell adhesion, for example, where it is typically expected that $\alpha_1 < \gamma < \alpha_2$ (or $\alpha_2 < \gamma < \alpha_1$). Further investigation would be required to choose parameter values that are more characteristic of what we observe experimentally. Even with such caveats, it is quite interesting to observe in a rigorous setting that a fully adhesive setting can produce stable segregation patterns.

\begin{figure}[ht]
    \centering
    \includegraphics[width=0.5\linewidth]{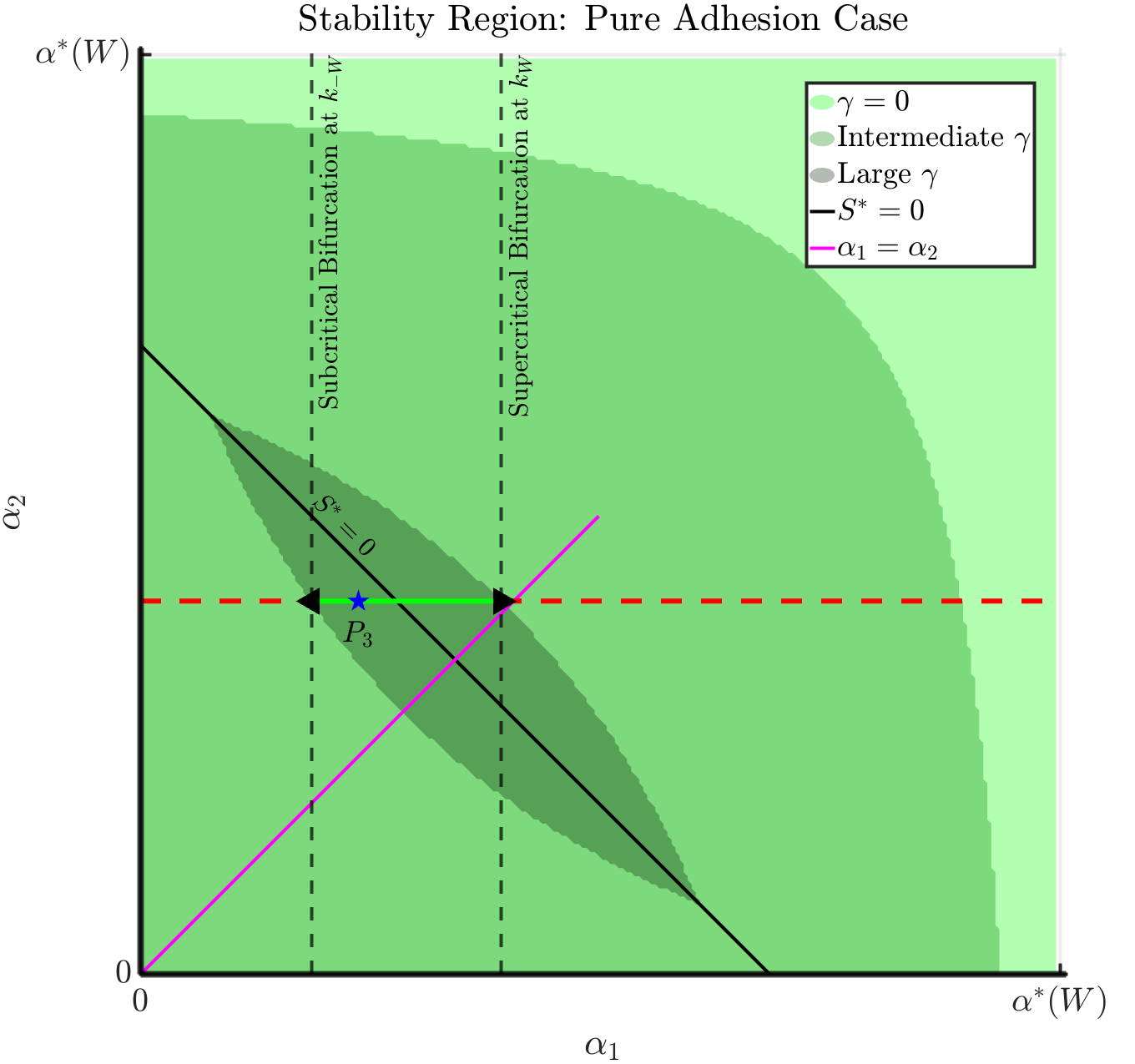} \\
    \includegraphics[width=0.9\linewidth]{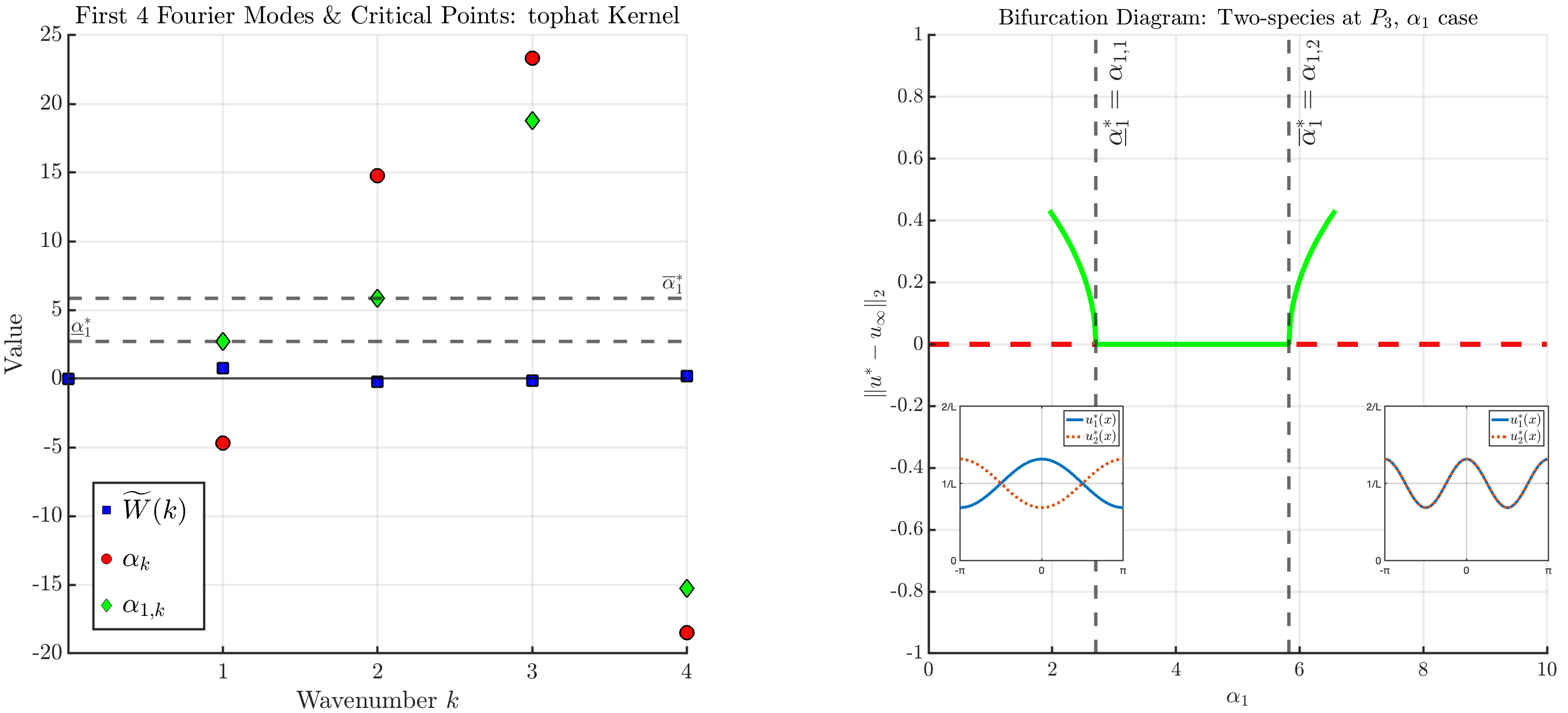}
    \caption{Visualisation of Theorem \ref{thm:main_results_local_stability_alpha_system} for Case 2. For large values of $\gamma$, there is now an island of local stability away from $(\alpha_1, \alpha_2) = (0,0)$. Decreasing the bifurcation parameter $\alpha_1$ from the point $P_3$ leads to a subcritical bifurcation at $\alpha_1 = \underline{\alpha}_1^*$, while increasing $\alpha_1$ from $P_3$ leads to a supercritical bifurcation at $\alpha_1 = \overline{\alpha}_1^*$. Both branches are stable near the bifurcation point. The lower branch has components out of phase with frequency $k_{-W} = 1$; the upper branch has components in phase with frequency $k_W = 2$. Further discussion is found in Section \ref{subsec:application}.}
    \label{fig:system_bif_alpha1_adhesion_case}
\end{figure}

\section{Preliminaries}\label{sec:preliminaries}


Let $\mathbb{T} := \mathbb{R} / L \mathbb{Z} := ( -\tfrac{L}{2}, \tfrac{L}{2})$, the torus of size $L > 0$. Denote $Q_T = \mathbb{T} \times (0,T)$ for $T>0$ fixed. We denote by $\mathcal{P}(\mathbb{T})$ the class of Borel probability measures on $\mathbb{T}$; by $\mathcal{P}_{\text{ac}}(\mathbb{T}) \subset \mathcal{P}(\mathbb{T})$ those absolutely continuous with respect to the Lebesgue measure; and by $\mathcal{P}_{\text{ac}}^+(\mathbb{T}) \subset \mathcal{P}_{\text{ac}}(\mathbb{T})$ those having strictly positive densities almost everywhere in $\mathbb{T}$. We also denote by $C^m (\mathbb{T})$ the restriction of all $L$-periodic and $m$-times continuously differentiable functions. We denote by $W * f$ the spatial convolution of a function $f \in L^2(\mathbb{T})$ with an even kernel $W$:
\begin{align}
    W*f(x) := \int_\mathbb{T} W(x-y) f(y) \dy.
\end{align}

For $1 \leq p \leq \infty$, we denote by $L^p(\mathbb{T})$ the Lebesgue space of $L^p$-integrable functions; $W^{m,p} (\mathbb{T})$ denotes the usual periodic Sobolev space of $m$-times weakly differentiable functions with derivatives belonging to $L^p(\mathbb{T})$; and $H^m (\mathbb{T}) = W^{m,2} (\mathbb{T})$. We denote by $\textup{BV}(\mathbb{T})$ the space of functions of Bounded Variation, i.e., those functions belonging to $L^1(\mathbb{T})$ with a well-defined distributional derivative $DW$ (see Appendix \ref{sec:BV_functions} and the references therein).

We then denote by $L^2_s(\mathbb{T}) \subset L^2(\mathbb{T})$ the space of even, square-integrable functions, a closed subspace of $L^2(\mathbb{T})$. We have an orthonormal basis of $L^2_s(\mathbb{T})$ given by $\{ w_k (x) \}_{k=1}^\infty$ where
\begin{align}\label{eq:orthogonal_basis_wk}
    w_k(x) = \begin{cases}
    \sqrt{\tfrac{2}{L}} \cos \left( \tfrac{2 \pi k x}{L} \right), \quad k \geq 1, \cr 
    \tfrac{1}{\sqrt{L}}, \quad\quad\quad\quad\quad\quad k = 0.
\end{cases}
\end{align}
In the scalar case, we use the standard inner product $(u,v) = ( u,v )_{L^2} := \int_\mathbb{T} u v \dx $. Given a function $f \in L^2_s(\mathbb{T})$ we then denote by $\widetilde f(k) := ( f, w_k )$ the cosine transform of $f$.

For the two-species system, we then consider elements belonging to the product space $\mathbf{f}:=(f_1,f_2) \in L^2_s (\mathbb{T}) \times L^2_s (\mathbb{T})$. We equip $L^2_s(\mathbb{T}) \times L_s^2(\mathbb{T})$ with the usual inner product
\begin{align}\label{euclidian_operator_norm_def}
    \left< f , g \right>_{\mathbb{H}} = \left< (f_1, f_2) , (g_1, g_2) \right>_{\mathbb{H}} := (f_1, g_1) + (f_2, g_2) 
\end{align}
and denote the Hilbert space $\mathbb{H} := L^2_s (\mathbb{T}) \oplus L^2_s (\mathbb{T})$. Given the orthonormal basis $\{ w_k \}_{k=1}^\infty$ of $L^2_s (\mathbb{T})$ given by \eqref{eq:orthogonal_basis_wk}, we have a natural orthonormal basis for $\mathbb{H}$ given by
\begin{align}\label{ortho_basis_twospecies}
\{ \bm{w}_{1,k}, \bm{w}_{2,k} \}_{k=1}^\infty =  \{ ( w_k , 0) , ( 0 , w_k ) \} _{k=1}^\infty .
\end{align}

We then adapt the following well-posedness result of \cite{carrillo2024wellposedness} for lower-regularity kernels.

\begin{theorem}\label{thm:wellposed_time}
    Assume \textbf{\ref{hyp:kernel_shape}} holds. Then, given $0 < u_{i0} \in H^{4} (\mathbb{T}) \cap \mathcal{P}_{\textup{ac}} (\mathbb{T})$ for each $i=1,2$, there exists a unique classical solution $\mathbf{u} (\cdot,t) = (u_1,u_2)$ solving problem \eqref{eq:general_system} subject to \eqref{interaction_assumptions} such that $\mathbf{u} (\cdot,t) \in \mathcal{P}_{\textup{ac}} (\mathbb{T}) \cap C^2 (\mathbb{T}) \times \mathcal{P}_{\text{ac}} (\mathbb{T}) \cap C^2 (\mathbb{T})$ for all $t>0$. Moreover, $\mathbf{u}$ is strictly positive in $\mathbb{T}$, i.e., $u_i (\cdot,t)>0$ for each $i=1,2$ for all $t > 0$, and has finite entropy, i.e., $\mathcal{S} (u_1(\cdot,t),u_2(\cdot,t)) < \infty$ for all $t>0$.
\end{theorem}

\begin{proof}[Theorem \ref{thm:wellposed_time}]
    For the well-posedness and strict positivity, we use results obtained in \cite{carrillo2024wellposedness}. First, our Hypotheses \textbf{\ref{hyp:kernel_shape}} encodes hypotheses \textbf{(H1)}-\textbf{(H4)} of \cite{carrillo2024wellposedness}. Since $u_{i0} \in H^4(\mathbb{T})$, the Sobolev embedding ensures that in fact $u_{i0} \in C^{3, 1/2} (\mathbb{T})$ and so $\chi_i W * u_{i0} \in W^{2,\infty} (\mathbb{T})$ and $\grad u_{i0} \in L^\infty(\mathbb{T})$. Finally, we note that results established in \cite{carrillo2024wellposedness} are done on the whole space; the case on the $1$-dimensional torus is more elementary, and in particular does not necessarily require that the kernels have compact support (as assumed in \textbf{(H6)} of \cite{carrillo2024wellposedness}). Hence, the existence of a unique, strictly positive classical solution follows from \cite[Theorem 4.3]{carrillo2024wellposedness}.

    Finally, the boundedness of the entropy $\mathcal{S}(\mathbf{u})$ follows from the strict positivity and boundedness of $u_1, u_2$ in $H^1(\mathbb{T})$ for all $t>0$.
\end{proof}

\section{Global Asymptotic Stability}\label{sec:GAS}

We prove the following exponential convergence to the stationary state in relative entropy, which is essentially a more general result than Theorem \ref{thm:global_stability_system}.

\begin{proposition}\label{prop:trend_to_homogeneous}
    Let $\mathbf{u_0} \in \mathcal{P}_{\textup{ac}}(\mathbb{T}) \cap H^4 (\mathbb{T})$ with $\mathcal{S} (\mathbf{u_0}) < \infty$ and assume that $W_i$, $W \in W^{2,\infty}(\mathbb{T})$, $i=1,2$ are even. Suppose that $\alpha_i, \gamma \geq 0$ are such that there holds 
    $$
     \overline{\gamma} + \overline{\alpha}  < \frac{2 \pi^2}{L^2}
    $$
    where 
$$
\overline{\gamma} := \gamma \norm{W_{xx}}_{L^\infty} , \quad\quad \overline{\alpha} := \max_{i=1,2} \left\{\alpha_i \norm{(W_{iu})_{xx}}_{L^\infty} \right\} ,
$$
and $W_{iu}$ denotes the unstable part of the kernel $W_i$. Then the classical solution $\mathbf{u}$ solving \eqref{eq:general_system} subject to \eqref{interaction_assumptions} is exponentially stable in relative entropy and for all $t\geq0$ there holds
    $$
\mathcal{H} (\mathbf{u} \, \vert \, \mathbf{u_\infty}) \leq \exp \left[ - \left( \frac{4 \pi^2}{L^2} - 2(\overline{\gamma} + \overline{\alpha} ) \right) t  \right] \mathcal{H}(\mathbf{u_0} \, \vert \, \mathbf{u_\infty}) .
$$

\end{proposition}

\begin{proof}[Proof of Proposition \ref{prop:trend_to_homogeneous} and Theorem \ref{thm:global_stability_system}]
    As in \cite{carrillo2019aggregation}, we can introduce the relative entropy of the system:
\begin{align}
    \mathcal{H}(\mathbf{u}\, \vert \, \mathbf{u_\infty}) = \sum_{i=1}^2 \int u_i \log \left( \frac{u_i}{u_\infty} \right)\, {\rm d}x,
\end{align}
where $\mathbf{u_\infty} = ( u_\infty, u_\infty)$ has identical components from our normalisation choice earlier. We then obtain a similar identity between the free energy and the relative entropy:
\begin{align}
    \mathcal{F} (\mathbf{u}) - \mathcal{F} (\mathbf{u_\infty}) = \sigma \mathcal{H}( \mathbf{u} \, \vert \, \mathbf{u_\infty} ) + \frac{1}{2} \mathcal{E} (\mathbf{u}- \mathbf{u_\infty} , \mathbf{u}- \mathbf{u_\infty} ).
\end{align}
Note that, with modification of the free energy, we cannot dispense with the normalisation chosen.

We can then follow \cite[Proof of Theorem 1.1(b)]{Carrillo2020} to get estimates of the form
\begin{align}
    \frac{{\rm d}}{{\rm d}t}\mathcal{H} (\mathbf{u} \, \vert \, \mathbf{u_\infty})  &\leq  - \frac{4 \pi^2 \sigma}{L^2} \mathcal{H} (\mathbf{u} \, \vert \, \mathbf{u_\infty})  + \sum_{i=1}^2 \int u_i \left( \alpha_i (W_i)_{xx} * u_i + \gamma W_{xx}* u_j   \right) {\rm d}x .
\end{align}

As in \cite{Carrillo2020}, we can control the self-interaction terms as follows:
$$
\int u_i (W_i)_{xx} * u_i {\rm d}x \leq \int u_i (W_{iu})_{xx} * u_i { \rm d}x ,
$$
where $(W_{iu})$ denotes the \textit{unstable} components of the kernels $W_i$, i.e., the portion of $W_i$ that has negative Fourier modes. Unlike \cite{Carrillo2020}, however, we cannot use the same procedure for the cross-interaction terms since the term $ \iint u_i W * u_j {\rm d}x$ is no longer sign-definite with respect to the stable or unstable components of the kernel $W$. Hence, we must retain all modes. As in \cite{Carrillo2020}, we use that $(W_{iu})_{xx}$ has mean zero to replace $u_i$ by $u_i - u_\infty$ so that for the self-interaction terms
$$
\int u_i  (W_{iu})_{xx} * u_i {\rm d}x \leq \norm{(W_{iu})_{xx} * (u_i - u_\infty)}_{L^\infty} \norm{u_i - u_\infty}_{L^1} \leq \norm{(W_{iu})_{xx}}_{L^\infty} \norm{u_i - u_\infty}_{L^1}^2 .
$$
For the cross-interaction terms, we first note that from the evenness of the kernel and symmetry of the system, we may combine them as $2 \int u_i W_{xx} * u_j {\rm d}x$. Then, estimating as above, we obtain
$$
\int u_i  W_{xx} * u_j {\rm d}x \leq \norm{W_{xx} * (u_j - u_\infty)}_{L^\infty} \norm{u_i - u_\infty}_{L^1} \leq \norm{W_{xx}}_{L^\infty}\norm{u_j - u_\infty}_{L^1} \norm{u_i - u_\infty}_{L^1} .
$$
Hence, putting these estimates together with Cauchy's inequality yields
\begin{align}
    \sum_{i=1}^2 \int u_i \left( \alpha_i (W_i)_{xx} * u_i + \gamma W_{xx}* u_j   \right) {\rm d}x \leq& \left( \alpha_1 \norm{(W_{1u})_{xx}}_{L^\infty} + \gamma \norm{W_{xx}}_{L^\infty}  \right) \norm{u_1 - u_\infty}_{L^1}^2  \nonumber \\
    & + \left( \alpha_2 \norm{(W_{2u})_{xx}}_{L^\infty} + \gamma \norm{W_{xx}}_{L^\infty}  \right) \norm{u_2 - u_\infty}_{L^1} ^2.
\end{align}
Then, using the Csiszár-Kullback-Pinsker inequality, we can write
$$
\norm{u_i - u_\infty}_{L^1}^2 \leq 2 \int u_i \log \left( \frac{u_i}{u_\infty} \right) {\rm d}x
$$
so that
\begin{align}
        \frac{{\rm d}}{{\rm d}t}\mathcal{H} (\mathbf{u} \, \vert \, \mathbf{u_\infty}) \leq&\ 2 \left( \alpha_1 \norm{(W_{1u})_{xx}}_{L^\infty} + \gamma \norm{W_{xx}}_{L^\infty}  \right) \int u_1 \log \left( \frac{u_1}{u_\infty} \right) {\rm d}x  \nonumber \\
    & + 2 \left( \alpha_2 \norm{(W_{2u})_{xx}}_{L^\infty} + \gamma \norm{W_{xx}}_{L^\infty}  \right) \int u_2 \log \left( \frac{u_2}{u_\infty} \right) {\rm d}x \nonumber \\
    & - \frac{4 \pi^2 \sigma}{L^2} \mathcal{H} (\mathbf{u} \, \vert \, \mathbf{u_\infty}) \nonumber \\
    =&\ 2 \alpha_1 \norm{(W_{1u})_{xx}}_{L^\infty} \int u_1 \log \left( \frac{u_1}{u_\infty} \right) {\rm d}x + 2 \alpha_2 \norm{(W_{2u})_{xx}}_{L^\infty} \int u_2 \log \left( \frac{u_2}{u_\infty} \right) {\rm d}x \nonumber \\
    & - \left( \frac{4 \pi^2 \sigma}{L^2} - 2 \gamma \norm{W_{xx}}_{L^\infty} \right) \mathcal{H} (\mathbf{u} \, \vert \, \mathbf{u_\infty}) .
\end{align}
Unfortunately, we cannot combine the remaining quantities in a more optimal way, so we simply define
$$
\overline{\gamma} := \gamma \norm{W_{xx}}_{L^\infty} , \quad\quad \overline{\alpha} := \max_{i=1,2} \left\{\alpha_i \norm{(W_{iu})_{xx}}_{L^\infty} \right\}
$$
and conclude by Gr{\" o}nwall's inequality that
$$
\mathcal{H} (\mathbf{u} \, \vert \, \mathbf{u_\infty}) \leq \exp \left[ - 2\left( \frac{2 \pi^2 \sigma}{L^2} - (\overline{\gamma} + \overline{\alpha} ) \right) t  \right] \mathcal{H}(\mathbf{u_0} \, \vert \, \mathbf{u_\infty}) .
$$
Notice that if $\gamma = 0$, the equations become decoupled, and we recover precisely the estimate of \cite{Carrillo2020}. This completes the proof of Proposition \ref{prop:trend_to_homogeneous}.

The conclusion of Theorem \ref{thm:global_stability_system} follows by ignoring the \textup{unstable modes} step executed above.
\end{proof}

\section{Characterisation of Stationary States}\label{sec:stationary_states_charaterisation}

We now seek to study the stationary states of problem \eqref{eq:general_system} subject to \eqref{interaction_assumptions}, that is, classical solutions $\mathbf{u} \in C^2 (\mathbb{T}) \times C^2 (\mathbb{T})$ solving problem \eqref{eq:general_system_SS} subject to \eqref{interaction_assumptions}. We establish analogues to Theorem 2.3, Proposition 2.4, and Theorem 2.7 as found in \cite{Carrillo2020}.

We first state an existence and regularity result for the stationary problem. As in \cite{Carrillo2020}, we formulate a relationship between solutions of the stationary problem and fixed points of a nonlinear map; since the proof is similar to that presented in \cite{Carrillo2020}, we move the proof to the Appendix \ref{sec:appendix}. We again note that our result uses the weaker bounded variation condition of Hypothesis \textbf{\ref{hyp:kernel_shape}} as opposed to the usual $H^1$-regularity assumption, and is therefore a moderate improvement of \cite[Theorem 2.3]{Carrillo2020}.

\begin{theorem}[Existence, regularity, and strict positivity]\label{thm:existenceregularitySS}
Consider the stationary problem \eqref{eq:general_system_SS} under Hypothesis \textbf{\ref{hyp:kernel_shape}}. Then we have that
\begin{enumerate}[label=\alph*.)]
    \item  There exists a weak solution $\mathbf{u} = (u_1, u_2) \in \left[ H^1 (\mathbb{T}) \cap \mathcal{P}_{\textup{ac}} (\mathbb{T}) \right]^2$ solving \eqref{eq:general_system_SS}, and any weak solution is a fixed point of the nonlinear map $ \mathcal{T}: \left[ \mathcal{P}_\textup{ac} (\mathbb{T}) \right]^2 \mapsto \left[ \mathcal{P} _\textup{ac} (\mathbb{T})\right]^2$ given by
    \begin{align}\label{eq:nonlinear_map_1}
        \mathcal{T} \mathbf{u} & = \left( T_1 \mathbf{u} , T_2 \mathbf{u} \right) \nonumber \\
 &:= \left( \frac{1}{Z_1 (\mathbf{u}, \alpha_1, \gamma)} e^{ - (\alpha_1 W_1 * u_1 + \gamma W* u_2)}  , \frac{1}{Z_2 (\mathbf{u}, \alpha_2, \gamma)} e^{ - ( \alpha_2 W_2 * u_2 + \gamma W * u_1)
} \right) ,
    \end{align}
    where
    \begin{align}\label{eq:nonlinear_map_2}
        Z_i (\mathbf{u}, \alpha_i, \gamma) := \int_\mathbb{T} e^{ -(\alpha_i W_i * u_i + \gamma W * u_j) } \dx , \quad i \neq j.
    \end{align}
    \item Any weak solution $\mathbf{u}$ of \eqref{eq:general_system_SS} is smooth and strictly positive, i.e., $\mathbf{u} \in \left[  C^\infty (\mathbb{T}) \cap \mathcal{P}_\textup{ac}^+ (\mathbb{T})\right]^2$.
\end{enumerate}
\end{theorem}

We already know that the PDE has a free energy functional $\mathcal{F} : [\mathcal{P}_{\textup{ac}}^+ (\mathbb{T})]^2 \mapsto \mathbb{R}$ given by \eqref{free_energy_functional}. The first variation of $\mathcal{F}$ with respect to $u_i$ is
\begin{align}
    \frac{\delta \mathcal{F}}{\delta u_i} =  ( \log (u_i) + 1 ) + \alpha_i W_i * u_i + \gamma W* u_j.
\end{align}
Therefore, from \eqref{eq:energy_dissipation} there holds
\begin{align}
    \frac{{\rm d}}{{\rm d }t} \mathcal{F} (\mathbf{u}) = - \mathcal{J} (\mathbf{u}) \leq 0,
\end{align}
where the entropy dissipation functional $\mathcal{J}: \left[ \mathcal{P}_{\textup{ac}}^+ (\mathbb{T}) \right]^2 \mapsto \mathbb{R}^+ \cup \{ + \infty \}$ is defined by
\begin{align}
\mathcal{J} (\mathbf{u}) := \begin{cases}
        \sum_{i=1}^2 \int_\mathbb{T} u_i \magg{\frac{\p}{\p x} \left( \log (u_i) + \alpha_i W_i * u_i + \gamma W* u_j \right)}, \quad \mathbf{u} \in \left[\mathcal{P}_{\textup{ac}}^+ \cap H^1 (\mathbb{T}) \right]^2, \\
        +\infty, \hspace{7.2cm} \text{otherwise}.
    \end{cases}
\end{align}
Finally, for our subsequent bifurcation analysis, it is useful to define map $\widehat{G} : \left[ \mathcal{P}_{\textup{ac}} (\mathbb{T})\right]^2 \mapsto \left[ \mathcal{P}_{\textup{ac}} (\mathbb{T})\right]^2$ as
\begin{align}
    \widehat{G} (\mathbf{u}) &:= \mathbf{u}- \mathcal{T} \mathbf{u} = \left( u_1 - T_1 \mathbf{u} , u_2 - T_2 \mathbf{u} \right) ,
\end{align}
which encodes stationary states with fixed points of the nonlinear map $\mathcal{T}$. We establish the following equivalences for our system, which is analogous to \cite[Proposition 2.4]{Carrillo2020} for the multi-species case. Again, as the proofs are similar, we put them in the Appendix \ref{sec:appendix}.

\begin{proposition}[Some equivalencies]\label{prop:equivalencies1}
 Fix $\alpha_i,\, \gamma \geq 0$ for $i = 1,2$, and assume Hypothesis \textbf{\ref{hyp:kernel_shape}} holds. Let $\mathbf{u} \in \left[ \mathcal{P}_\textup{ac}^+ (\mathbb{T}) \right]^2$. Then, the following are equivalent.
 \begin{enumerate}[(1)]
     \item $\mathbf{u}$ is a classical solution of the stationary problem \eqref{eq:general_system_SS};
     \item $\mathbf{u}$ is a zero of the map $\widehat{G}( \mathbf{u})$;
     \item $\mathbf{u}$ is a critical point of the free energy $\mathcal{F} (\mathbf{u})$;
     \item $\mathcal{J}(\mathbf{u}) = 0$.
 \end{enumerate}
\end{proposition}
\begin{remark}
    Proposition \ref{prop:equivalencies1} extends the result of \cite{Carrillo2020} (in one spatial dimension) in two ways. First, we weaken the regularity requirements to Hypothesis \textbf{\ref{hyp:kernel_shape}}; second, we extend to the case of several interacting populations. In fact, Proposition \ref{prop:equivalencies1} is true for any number of interacting populations so long as the detailed-balance condition \eqref{condition:detailed_balance} holds.
\end{remark}

\section{Spectral and linear stability analysis}\label{sec:linear_analysis}

In this section, we establish some results for the linearised problems. The goal of this section is to perform a spectral analysis to identify bifurcation points, and then to describe in detail regions of linear stability for the linearised equation. We begin with a brief spectral analysis for the scalar equation as found in \cite[Section 3.2]{Carrillo2020}, and then provide a proof of Theorem \ref{thm:local_bifurcations_scalar_2}. We then obtain the analogous results for the two-species system \eqref{eq:general_system_SS}, identifying the bifurcation points as described in Theorems \ref{thm:bifurcations_alpha1_1} and \ref{thm:bifurcations_gamma_1}, followed by a proof of the stability result Proposition \ref{prop:local_stability}.

\subsection{The scalar equation}\label{subsec:scalar_stability}
We briefly review the scalar equation case, as it will be relevant to our subsequent analysis. The stationary problem reads
$$
0 = \sigma \rho_{xx} + \alpha ( \rho (W*\rho)_x ) _x,
$$
with $\rho = \rho(x)$ and we treat $\alpha \geq 0$ as the bifurcation parameter. Recall that the point of critical stability $\alpha^* = \alpha^*(W)>0$ is given by \eqref{critical_alpha_positive} for a given interaction kernel $W$, and is finite whenever $\mathcal{K}^-$ is non-empty. Linearising \eqref{eq:general_system_SS} about the homogeneous state $\rho_\infty = L^{-1}$, one obtains
\begin{align*}
    \mathcal{L}_0 w := \sigma w_{xx} + \alpha \rho_\infty ( W * w)_{xx},
\end{align*}
a symmetric integrodifferential operator whose eigenfunctions form an orthonormal basis of $L^2_s(\mathbb{T})$. The eigenfunctions are given precisely by the orthonormal basis introduced in \eqref{eq:orthogonal_basis_wk}. Using the evenness of the kernel $W$ and properties of the orthonormal basis, one identifies the eigenvalues of $\mathcal{L}_0$ to be (see \cite[Section 3.2]{Carrillo2020}):
\begin{align}\label{eq:scalar_spectrum}
    \lambda(k) = - \left( \frac{2 \pi k}{L} \right)^2 \left( \sigma + \alpha \frac{\widetilde W (k)}{\sqrt{2L}} \right).
\end{align}
For fixed $\sigma, L > 0$, we are interested in values of $\alpha > 0$ for which $\lambda(k) = 0$. From relation \eqref{eq:scalar_spectrum}, for any wavenumber $k \geq 1$ such that $\widetilde W(k) \neq 0$ we identify the values
\begin{align}\label{eq:critical_alpha_scalar}
    \alpha_k = - \frac{\sigma \sqrt{2L}}{\widetilde W(k)} = - h_k.
\end{align}
As we seek bifurcation points $\alpha > 0$, we need only consider those wavenumbers for which $\widetilde W(k) < 0$. This is why we require condition $ii.)$ to hold in Theorem \ref{thm:local_bifurcations_scalar}. We now prove Theorem \ref{thm:local_bifurcations_scalar}.

\begin{proof}[Proof of Theorem \ref{thm:local_bifurcations_scalar}]
    Using the formulas in the Appendix \ref{sec:appendix_frechet_derivatives_n_species}, one may follow the proof of \cite[Theorem 4.2]{Carrillo2020} to recover Theorem \ref{thm:local_bifurcations_scalar}, except for the branch direction. However, we prefer to give a full proof of the theorem for completeness. To this end, set $G(u, \alpha) := (I - T)u$ and assume $\alpha_k>0$ is a valid bifurcation point as defined in \eqref{eq:critical_alpha_scalar}. We first identify the first Fr{\'e}chet derivative of $G$ evaluated at $(u,\alpha) = (u_\infty, \alpha_k)$. To this end, we introduce the following bounded linear functional $F :  L^2(\mathbb{T}) \mapsto L^2 ( \mathbb{T})$ for a fixed kernel $W$:
\begin{align}
    F(\eta; W) := W * \eta - \frac{1}{L} \int_\mathbb{T} W * \eta \, \dy,
\end{align}
where $\eta \in L^2_s (\mathbb{T})$ is a mean-zero variation. We make note of two key identities that will simplify future computations. Let $c \in \mathbb{R}$ be a constant, $w_k(x)$ a basis element for $k \geq 1$ as defined in \eqref{eq:orthogonal_basis_wk}, and $W \in L_s^2(\mathbb{T})$. Then there holds
\begin{align}\label{eq:F_map_identities}
    F(c; W) = 0, \quad \text{ and } \quad F( c w_k (x); W ) = c F(w_k(x); W) =  c \sqrt{\tfrac{L}{2}} \widetilde W(k) w_k(x).
\end{align}
From \eqref{app:first_frechet_scalar_preeval}, we find that $D_u G(u_\infty, \alpha_k) [ \eta ] = \eta + \frac{\alpha}{\sigma L} F(\eta, W)$, and so from relations \eqref{eq:F_map_identities} we immediately identify the kernel to be $\ker ( D_u G(u_\infty, \alpha_k) ) = \Span \{ w_k \}$. As we assume $\alpha_k$ is achieved at a unique wavenumber $k \geq 1$, there holds $\Dim ( \ker ( D_u G(u_\infty, \alpha_k) ) = 1$. We then decompose the space $L^2_s(\mathbb{T})$ as
    $$
L^2_s(\mathbb{T}) = \ker ( D_u G(u_\infty, \alpha_k) ) \oplus \ran ( D_u G(u_\infty, \alpha_k) ),
    $$
    and we denote by $P: L_s^2(\mathbb{T}) \mapsto \Span \{ w_k \}$ the projection along $\ran ( D_u G(u_\infty, \alpha_k) )$. Note carefully that \cite{kielh2004bifurcation} treats the general case of a mapping $G : X \mapsto Z$, and therefore introduces independent projections along the range in $X$ and along the kernel in $Z$. Since we have $X = Z = L^2_s(\mathbb{T})$, and since our operator self-adjoint, we need not distinguish between these two notions of projection. See \cite[(I.3.7)]{kielh2004bifurcation} and the surrounding discussion for further details.
    
    We now verify the transversality condition. From equation \eqref{app:mixed_derivative_scalar}, it is direct to find 
\begin{align}\label{eq:second_mixed_scalar_proof_1}
    (D^2_{u \alpha} G (u_\infty, \alpha_k) [w_k] , w_k) = h_k^{-1} \neq 0,
\end{align}
    since $h_k$ is nonzero by assumption.
    
    We now confirm that the bifurcation is \textit{not} a transcritical one. By formula \eqref{app:second_derivative_scalar}, we find $D^2_{uu} G(u_\infty, \alpha_k) [ w_k, w_k] \sim w_k^2 + 1$, which is orthogonal to $\ker ( D_u G(u_\infty, \alpha_k) )$. Therefore, by formula \cite[(I.6.3)]{kielh2004bifurcation} we have $\alpha^\prime _k(0) \sim (D^2_{uu} G(u_\infty, \alpha_k) [ w_k, w_k], w_k) = 0$, and the bifurcation is not transcritical.
    
    The bifurcation direction is thus determined by $\alpha_k ^{\prime \prime } (0)$, which requires derivatives of our nonlinear map up to and including third order. The formula \cite[(I.6.11)]{kielh2004bifurcation} is given as
\begin{align}\label{eq:third_derivative_formula_general}
        \alpha_k ^{\prime \prime} (0) = - \frac{1}{3} \ \frac{ ( D^3_{uuu} \Phi(u_\infty, \alpha_k)[w_k, w_k, w_k] , w_k) }{( D^2_{u \alpha} G(u_\infty, \alpha_k) [w_k]  ,w_k)} ,
    \end{align}
    where $D^3_{uuu} \Phi(u_\infty, \alpha_k)[w_k, w_k, w_k]$ is obtained via
    \begin{align}\label{eq:Phi_third_derivative_definition}
        D^3_{uuu} \Phi(u_\infty, \alpha_k)[w_k, w_k, w_k] = &\ P D^3_{uuu} G(u_\infty, \alpha_k) [ w_k, w_k, w_k] \nonumber \\
        & - 3 P D^2_{uu} G(u_\infty, \alpha_k) [ w_k, (I - P) (D_u G(u_\infty, \alpha_k) )^{-1} (I - P) D^2_{uu} G(u_\infty,\alpha_k) [ w_k, w_k ] ) .
    \end{align}

    \noindent \textbf{Step 1.} We first identify the first component of $D^3_{uuu} \Phi(u_\infty, \alpha_k)[w_k, w_k, w_k]$ defined in \eqref{eq:Phi_third_derivative_definition}. From formula \eqref{app:third_derivative_scalar}, we have that
    \begin{align}\label{eq:third_derivative_inter_1}
        D^3_{uuu} G(u_\infty, \alpha_k) [ w_k, w_k, w_k] = \frac{1}{L} \left( \tfrac{\alpha_k }{\sigma} \sqrt{\tfrac{L}{2}} \widetilde W(k) \right)^3  \left[ w_k^3(x)  - \frac{3}{L} w_k(x) \right] .
    \end{align}
    Under the projection $P$ we need to compute
\begin{align}\label{eq:integral_third_piece}
    \int_\mathbb{T} w_k^4 \dx - \frac{3}{L} \int w_k^2 \dx &= \frac{12}{8L} - \frac{3}{L} =  - \frac{3}{2L}.
\end{align}
Substituting \eqref{eq:integral_third_piece} and $\alpha_k = - \sigma \sqrt{2L} / \widetilde W(k)$ intro \eqref{eq:third_derivative_inter_1}, we conclude that
\begin{align}\label{eq:scalar_projection_proof_1}
    P D^3_{uuu} G(u_\infty, \alpha_k) [ w_k, w_k, w_k] = \frac{3 L}{2} w_k .
\end{align}

    \noindent \textbf{Step 2.} We now identify the second component of $D^3_{uuu} \Phi(u_\infty, \alpha_k)[w_k, w_k, w_k]$ defined in \eqref{eq:Phi_third_derivative_definition}, working from the inside out. First, from formula \eqref{app:second_derivative_scalar} and the identity $w_k^2 - 1/L = w_{2k} / \sqrt{2L}$ we have that
    \begin{align}\label{eq:second_der_scalar_11}
        D^2_{uu} G(u_\infty,\alpha_k) [ w_k, w_k ] = &\ - \frac{1}{2} \left( \frac{\alpha_k \widetilde W(k)}{\sigma} \right)^2 ( w_k^2 (x) - \frac{1}{L} ) = - \sqrt{\tfrac{L}{2}} w_{2k},
    \end{align}
    where we have substituted $\alpha_k$ and used the identity $\cos^2 (\theta) = (1 + \cos (2 \theta) ) / 2$. Clearly, $D^2_{uu} G(u_\infty,\alpha_k) [ w_k, w_k ]$ is orthogonal to $w_k \in \ker ( D_u G(u_\infty, \alpha_k) )$, and so $(I - P) D^2_{uu} G(u_\infty,\alpha_k) [ w_k, w_k ] = D^2_{uu} G(u_\infty,\alpha_k) [ w_k, w_k ]$. 

    Next, due to our decomposition of $L^2_s(\mathbb{T})$, we have that $D_u G(u_\infty, \alpha_k) : \ran (D_u G(u_\infty, \alpha_k)) \mapsto (I - P) D_u G(u_\infty, \alpha_k)$ is an isomorphism. Hence, we seek $\eta \in L_s^2(\mathbb{T})$ such that
    $$
D_u G(u_\infty, \alpha_k) [ \eta ] = (I - P) D^2_{uu} G(u_\infty,\alpha_k) [ w_k, w_k ] .
    $$
From \eqref{eq:second_der_scalar_11} and formula \eqref{app:first_frechet_scalar_preeval}, we thus seek $\eta$ such that
    $$
D_u G(u_\infty, \alpha_k) [ \eta ] = \eta + \frac{\alpha_k}{\sigma L} F(\eta ; W) = - \sqrt{\tfrac{L}{2}} w_{2k}  ,
    $$
    Since $F(\cdot; W)$ is linear, the only possibility is that $\eta = \beta w_{2k}$ for some constant $\beta$ to be determined. Substitution yields
    \begin{align}
        \beta := \frac{- \sqrt{\tfrac{L}{2}}}{\left( 1 - \tfrac{\widetilde W(2k)}{ \widetilde W(k) } \right)} = - \sqrt{\frac{L}{2}} \ \left( \frac{\widetilde W(k)}{\widetilde W(k) - \widetilde W(2k)} \right)
    \end{align}
Notice that $\beta$ is well-defined since $\widetilde W(k) \neq \widetilde W(2k)$ by assumption. As before, $(I-P) \beta w_{2k} = \beta w_{2k}$ since $w_{2k}$ is orthogonal to $w_k$. 

Finally, we compute $D^2_{uu} G(u_\infty, \alpha_k) [ w_k, \beta w_{2k} ]$. From formula \eqref{eq:second_der_scalar_aux}, we find
\begin{align}
    D^2_{uu} G(u_\infty, \alpha_k) [ w_k, \beta w_{2k} ] = - \frac{1}{L} \left( \frac{\alpha_k}{\sigma} \right)^2 \left[ F( \beta w_{2k}; W) F( w_k; W) - \frac{1}{L} \int_\mathbb{T} F(\beta w_{2k} ; W) W * w_k \dx \right].
\end{align}
By orthogonality, the integral term vanishes, and using \eqref{eq:F_map_identities} we are left with
\begin{align}
    D^2_{uu} G(u_\infty, \alpha_k) [ w_k, \beta w_{2k} ] =&\ - \frac{1}{L} \left( \frac{\alpha_k}{\sigma} \right)^2 \left[ \sqrt{\tfrac{L}{2}} \ \beta \ \widetilde W(2k) \ w_{2k} \ \sqrt{\tfrac{L}{2}}\  \widetilde W(k) \ w_{k} \right] \nonumber \\
    =&\  L \sqrt{\frac{L}{2}} \left( \frac{\widetilde W(2k)}{ \widetilde W(k) - \widetilde W(2k)} \right) w_{2k} w_k \nonumber \\
    =&\  L^2 \left( \frac{\widetilde W(2k)}{ \widetilde W(k) - \widetilde W(2k)} \right) \left( w_k^3 - \frac{w_k}{L} \right).
\end{align}
To obtain the projection under $P$, we use again computation \eqref{eq:integral_third_piece} to obtain
\begin{align}\label{eq:auxiliary_derivative_prof_2}
    P D^2_{uu} G(u_\infty, \alpha_k) [ w_k, \beta w_{2k} ] =  L^2 \left( \frac{\widetilde W(2k)}{ \widetilde W(k) - \widetilde W(2k)} \right) \left( \frac{12}{8L} - \frac{1}{L} \right) w_k = \frac{L}{2} \left( \frac{\widetilde W(2k)}{ \widetilde W(k) - \widetilde W(2k)} \right) w_k
\end{align}

\noindent\textbf{Step 3.} We now piece all elements together. Inserting \eqref{eq:second_mixed_scalar_proof_1}, \eqref{eq:scalar_projection_proof_1}, and \eqref{eq:auxiliary_derivative_prof_2} into \eqref{eq:third_derivative_formula_general}, we obtain the final result after simplification:
\begin{align}
    \alpha^{\prime \prime}_k (0) =& - \frac{1}{3} h_k \left( \frac{3 L}{2} - 3 \left[ \frac{L}{2} \left( \frac{\widetilde W(2k)}{ \widetilde W(k) - \widetilde W(2k)} \right)\right] \right) \nonumber \\
    =&\ - \frac{L}{2} h_k \left( 1 - \left( \frac{\widetilde W(2k)}{ \widetilde W(k) - \widetilde W(2k)} \right) \right) ,
\end{align}
which is well-defined as we assume $\widetilde W(k) \neq \widetilde W(2k)$. Since $h_k < 0$ by assumption, the sign of $\alpha_k^{\prime \prime}(0)$ is determined by the sign of
$$
\frac{\widetilde W(k) - 2 \widetilde W(2k)}{\widetilde W(k) - \widetilde W(2k)} = \frac{1 -2 r}{1 - r},
$$
where $r:= \widetilde W(2k) / \widetilde W(k) \neq 1$ by assumption. When $r \in (-\infty, 1/2) \cup (1, \infty)$, the second derivative is positive; when $r \in (1/2, 1)$, the second derivative is negative. Thus, when $\widetilde W(k) < 2 \widetilde W(2k)$ or $\widetilde W(2k) < \widetilde W(k)$, the bifurcation is supercritical; when $\widetilde W(k)< \widetilde W(2k)<\dfrac{\widetilde W(k)}{2}$, the bifurcation is subcritical. This completes the proof.
\end{proof}

We now prove Theorem \ref{thm:local_bifurcations_scalar_2}.
\begin{proof}[Proof of Theorem \ref{thm:local_bifurcations_scalar_2}]
Our goal is to write the time-dependent problem \eqref{eq:scalar_system_1} as an abstract semiflow of the form $u^\prime(t) = Au(t) + \mathcal{G}(u(t))$, where $A: D(A) \mapsto L^2(\mathbb{T})$ is sectorial, and the graph norm of $A$ is equivalent to the norm of $D(A)$, and where $\mathcal{G}$ is a continuously differentiable function with locally Lipschitz derivative satisfying $\mathcal{G}(u_\infty) =\mathcal{G}^\prime (u_\infty) = 0$ (see \eqref{eq:abstract_semiflow_scalar} below). This will allow us to apply the theory of \cite[Ch. 9.1]{lunardi1995analytic} so that the Principle of Linearised Stability holds, and subsequently so does the \textit{Principle of Exchange of Stability} \cite[Ch. I]{kielh2004bifurcation}. More precisely, we verify that the operator $A:=\mathcal{L}_0$ is sectorial so that the associated evolution equation $w_t - \mathcal{L}_0 w$ generates an analytic semigroup on the underlying space. Then, we study the spectrum of the operator $\mathcal{L}_0$, especially the smallest possible positive value of $\alpha_k$ as a function of $k$, so that the spectrum consists only of negative (real) values until $\alpha = \alpha^*(W)$, where $\alpha^*(W)$ is the point of critical stability defined in \eqref{critical_alpha_positive}. We may then apply \cite[Theorem I.7.4]{kielh2004bifurcation}.

To this end, we consider $\mathcal{L}_0 : D(\mathcal{L}_0) \mapsto L^2(\mathbb{T})$ with $D(\mathcal{L}_0) = H^2 (\mathbb{T})$, the space of $L$-periodic functions belonging to $W^{2,2}(\mathbb{T})$. It is not difficult to see that $D(\mathcal{L}_0) \subset L^2(\mathbb{T})$ is a closed subspace, and is in fact dense in $L^2(\mathbb{T})$. It is well known that the operator $\partial^2 / \partial x^2 : D(\mathcal{L}_0) \mapsto L^2(\mathbb{T})$ is sectorial, see \cite[Ch. 5]{lunardi1995analytic}.

Let $\{ z_k \}_{k=1}^\infty \subset D(\mathcal{L}_0)$ be a bounded sequence. By the Rellich-Kondrachov compactness theorem, there exists a subsequence (still labelled by $k$) such that $z_k \to z$ strongly in $H^1(\mathbb{T})$; in particular, $(z_k)_x, z_x \in L^2(\mathbb{T})$, and so by Lemma \ref{lemma:BV_TV_embeddings}, $W * (z - z_k)_x$ is weakly differentiable, and there holds
\begin{align*}
    \norm{( W * ( z - z_k) )_{xx}}_{L^2(\mathbb{T})} \leq  \norm{W}_{\textup{BV}} \norm{(z - z_k)_x}_{L^2(\mathbb{T})} \to 0 \text{ as } k \to \infty.
\end{align*}
Consequently, $(W * \cdot)_{xx} : D(\mathcal{L}_0) \mapsto L^2(\mathbb{T})$ is a compact linear operator whenever $W \in \textup{BV}(\mathbb{T})$. This is obviously weaker than as assumed in the statement of the theorem, since $H^1 (\mathbb{T}) \subset W^{1,1} (\mathbb{T}) \subset \textup{BV}(\mathbb{T})$. By \cite[Proposition 2.4.3]{lunardi1995analytic}, we conclude that $\mathcal{L}_0$ is sectorial.

Finally, by the Gagliardo-Nirenberg interpolation inequality, it is immediate that $\norm{u}_{D(\mathcal{L}_0)}$ is equivalent to $\norm{u}_{L^2} + \norm{\mathcal{L}_0u}_{L^2}$.

Hence, the nonlinear time-dependent problem can then be written as
\begin{align}\label{eq:abstract_semiflow_scalar}
    \rho_t = \mathcal{L}_0 \rho + \alpha \left( [\rho - \rho_\infty] (W * \rho)_x  \right)_x =: \mathcal{L}_0 \rho + \mathcal{G}(\rho),
\end{align}
so that $\mathcal{G}(\rho_\infty) = \mathcal{G}^\prime (\rho_\infty) = 0$. By Theorems \cite[Theorem 9.1.2]{lunardi1995analytic} and \cite[Theorem 9.1.3]{lunardi1995analytic}, the Principle of Linearised Stability holds for our original nonlinear problem, and we may proceed as follows.

As found in \eqref{eq:scalar_spectrum}, the spectrum of $\mathcal{L}_0$ is real, and so whenever there holds
$$
0 \leq \alpha < \alpha^*(W),
$$
we have that $\lambda(k) \leq - \varepsilon_0 < 0$, for all $k \geq 1$, for some $\varepsilon_0>0$. In such a case, $\rho = \rho_\infty$ is locally asymptotically stable. On the other hand, when $\alpha > \alpha^*(W)$, there exists $k \geq 1$ such that $\lambda(k) > 0$, and $\rho_\infty$ is unstable. The conclusion of the theorem follows from \cite[Theorem I.7.4]{kielh2004bifurcation}, where the stability exchange is determined by the branch direction given by the sign of $\alpha^{\prime \prime}_k(0)$ found in Theorem \ref{thm:local_bifurcations_scalar}. More precisely, when $\widetilde W(k) < 2 \widetilde W(2k)$ or $\widetilde W(2k) < \widetilde W(k)$, the bifurcation is supercritical, and an exchange of stability occurs. When $\widetilde W(k)< \widetilde W(2k)<\tfrac{\widetilde W(k)}{2}$, the bifurcation is subcritical, and the emergent branch is unstable. This completes the proof.
\end{proof}

\subsection{The two-species system: bifurcation points \& linear stability}\label{sec:two_species_linear_stability}

We now develop some analogous results for the two-species system. First, we identify the bifurcation points (in terms of $\alpha_1$ and $\gamma$) for the two-species system \eqref{eq:general_system_SS}, similar to the identification of the points \eqref{eq:critical_alpha_scalar} from the spectrum \eqref{eq:scalar_spectrum} in the scalar case. We then prove Proposition \ref{prop:local_stability}, providing a precise description of the region of linear stability.

\subsubsection{Identification of bifurcation points}

If we linearise the stationary problem \eqref{eq:general_system_SS} about some stationary state $(u_1^*, u_2^*)$, we obtain the following linear integrodifferential operator
$$
\mathcal{L} \mathbf{w} = \begin{pmatrix}
    \sigma (\mathrm{w}_1)_{xx} + \left[ \alpha_1 ( ( W_1 * u_1 ^*)_x \mathrm{w}_1 + u_1^* (W_1 * \mathrm{w}_1 )_x ) + \gamma ( (W*u_2^*)_x \mathrm{w}_1 + u_1^* (W * \mathrm{w}_2)_x ) \right]_x \\
    \sigma (\mathrm{w}_2)_{xx} + \left[ \alpha_2 ( ( W_2 * u_2 ^*)_x \mathrm{w}_2 + u_2^* (W_2 * \mathrm{w}_2 )_x ) + \gamma ( (W*u_1^*)_x \mathrm{w}_2 + u_2^* (W * \mathrm{w}_1)_x ) \right]_x
\end{pmatrix},
$$
where $\mathbf{w} = (\mathrm{w}_1, \mathrm{w}_2) \in [D( \mathcal{L}_0)]^2$. Here, we intentionally separate the self-interaction and cross-interaction terms to highlight that when $\gamma \equiv 0$, the two equations decouple, and we recover the linearised equation for the scalar problem for each component $u_i$ as obtained at the beginning of \cite[Section 3.2]{Carrillo2020}. If we choose the homogeneous state $(u_1^*, u_2^*) = \mathbf{u_\infty}$, this reduces to
\begin{align}
\mathcal{L} \mathbf{w} =& \begin{pmatrix}
    \sigma (\mathrm{w}_1)_{xx} + \alpha_1 L^{-1} ( W_1 * \mathrm{w}_1 )_{xx} + \gamma L^{-1} (W * \mathrm{w}_2)_{xx} \\
    \sigma (\mathrm{w}_2)_{xx} + \alpha_2 L^{-1} ( W_2 * \mathrm{w}_2)_{xx} + \gamma L^{-1} (W * \mathrm{w}_1)_{xx}
\end{pmatrix} \nonumber \\
=& \begin{pmatrix} 
    \frac{\p^2}{\p x^2} \left( \sigma \cdot + \alpha_1 L^{-1} W_1 * \cdot  \right) & \gamma L^{-1} \frac{\p^2}{\p x^2} \left( W * \cdot \right) \\
    \gamma L^{-1} \frac{\p^2}{\p x^2} \left( W * \cdot \right) & \frac{\p^2}{\p x^2} \left( \sigma \cdot + \alpha_2 L^{-1} W_2 * \cdot  \right) 
\end{pmatrix} \begin{pmatrix}
        \mathrm{w}_1 \\ \mathrm{w}_2
    \end{pmatrix} .
\end{align}
In this form, the diagonal elements clearly highlight the influence of diffusion/self-interaction and the connection to the scalar problem, while the off-diagonal elements describe the (symmetric) cross-interactions.

Notice that $\mathcal{L}$ at $\mathbf{u}=\mathbf{u}_\infty$ is of the form $\sigma \Delta I + A(x)$, where $A(x)$ is a symmetric $2\times2$ matrix. Therefore, $\left< \mathcal{L} \mathbf{f}, \mathbf{g}  \right>_{\mathbb{H}} = \left< \mathbf{f}, \mathcal{L} \mathbf{g}\right>_{\mathbb{H}}$, and so $\mathcal{L}$ is a symmetric elliptic integrodifferential operator. As in the scalar case, from the spectral theory for symmetric elliptic operators, eigenfunctions of $\mathcal{L}$ form an orthonormal basis of $\mathbb{H}$ given by \eqref{ortho_basis_twospecies}. The eigenvalues are real, given by the relation
\begin{align}\label{eq:same_wavenumber_disp_relation}
    2 \lambda^{\pm} (k) &= - \left( \frac{2 \pi k}{L}   \right)^2 \left( \Gamma_1(k) + \Gamma_2(k) \right) \pm \left( \frac{2 \pi k}{L}   \right)^2 \sqrt{\left( \Gamma_1(k) - \Gamma_2(k) \right)^2 + \left( 2 \gamma  (2 L)^{-1/2} \widetilde W(k) \right)^2}.
\end{align}
where, for a fixed wavenumber $k' \geq 1$, we define
\begin{align}\label{eq:critical_scalar_1}
    \Gamma_i (k') := \sigma  + \alpha_i (2L)^{-1/2} \widetilde W_i (k'),
\end{align}
and $\widetilde W_i (k)$ is the $k^{\textup{th}}$ Fourier mode of $W_i$ given by the cosine transform of $W_i$ as defined in \eqref{eq:cos_transform_basis_function}. Relation \eqref{eq:same_wavenumber_disp_relation} is as found in, e.g., \cite[Section 2.2]{giunta2024weakly} or \cite[Section 3]{pottslewis2019}.

For fixed parameters $\sigma, L > 0$, we seek to identify those parameter values $\gamma, \alpha_i \geq 0$, $i=1,2$, such that $\lambda^{\pm}(k) = 0$. Whenever $\lambda^{\pm}(k) = 0$, we obtain a hypersurface depending on $(\alpha_1, \alpha_2, \gamma)$ from which we isolate the positive parameters of interest. To this end, as the discriminant in \eqref{eq:same_wavenumber_disp_relation} is nonnegative, we set $\lambda^{\pm}(k) = 0$ to obtain the relation
\begin{align}\label{eq:critical_value_intermediate_step}
    \lambda^{\pm}_k = 0 \iff \Gamma_1(k) \Gamma_2(k) = \left(  \frac{\gamma}{\sqrt{2L}} \widetilde W(k) \right)^2,
\end{align}

We consider two cases separately, depending on whether we isolate $\alpha_1$ or $\gamma$.

\noindent\textbf{Case I: $\alpha_1 \geq 0$.} 
Isolating for $\alpha_1$ in \eqref{eq:critical_value_intermediate_step}, we obtain the following relation for each $k \geq 1$ such that $\widetilde W_i(k) \neq 0$:
\begin{align}\label{eq:initial_critical_alpha}
    \alpha_{1,k} := - \frac{\sigma \sqrt{2L}}{\widetilde W_1 (k)} + \frac{\gamma^2 \widetilde W(k)^2}{\widetilde W_1 (k) \widetilde W_2 (k)} \ \frac{1}{\left( \tfrac{\sigma \sqrt{2L}}{\widetilde W_2(k)} + \alpha_2 \right)}.
\end{align}
Under the structural criteria of Hypothesis \textbf{\ref{hyp:kernel_shape}}, i.e. since $W_i = \chi_i W$ for $\chi_i \in \{ \pm 1 \}$, we obtain the values $\alpha_{1,k}$ as defined in \eqref{eq:thm_alpha_crit_1} in the statement of Theorem \ref{thm:bifurcations_alpha1_1}:
\begin{align}\label{eq:alpha_1_k}
    \alpha_{1,k} = -\chi_1 h_k + \chi_1 \frac{\gamma^2}{(h_k + \chi_2 \alpha_2)},
\end{align}
where $h_k = \sigma \sqrt{2L} / \widetilde W(k)$ is as defined in \eqref{eq:h_k_relation}. Notice that $\alpha_{1,k}$ is only well defined for those wavenumbers such that $( h_k +\chi_2 \alpha_2) \neq 0$, and we only consider wavenumbers such that $\alpha_{1,k} > 0$, which are precisely conditions $ii.)$ and $iii.)$ of Theorem \ref{thm:bifurcations_alpha1_1}, respectively.

\noindent\textbf{Case II: $\gamma\geq0$.} From \eqref{eq:critical_value_intermediate_step} we find for each $k \geq 1$ such that $\widetilde W(k) \neq 0$ that
\begin{align}\label{eq:initial_critical_gamma}
     \gamma_k ^2 := \frac{ 2 L \Gamma_1(k) \Gamma_2(k) }{\magg{\widetilde W(k)}} = \frac{2L\left( \sigma + \alpha_1 \widetilde W_1 (k) / \sqrt{2L} \right)\left(\sigma + \alpha_2 \widetilde W_2 (k) / \sqrt{2L}  \right)}{\widetilde W(k) \widetilde W(k)}.
\end{align}
As in the previous case, under Hypothesis \textbf{\ref{hyp:kernel_shape}} we obtain the final values:
\begin{align}\label{eq:gamma_k}
     \gamma_k ^2  = \left( \tfrac{\sigma \sqrt{2L}}{\widetilde W(k)} + \chi_1 \alpha_1 \right) \left( \tfrac{\sigma \sqrt{2L}}{\widetilde W(k)} +\chi_2 \alpha_2 \right) = (h_k +\chi_1 \alpha_1 ) ( h_k +\chi_2 \alpha_2).
\end{align}
Notice that we need not consider wavenumbers such that $\sign (h_k +\chi_1 \alpha_1 ) \neq \sign ( h_k +\chi_2 \alpha_2)$, as such a case would yield $\gamma_k \in \mathbb{C} \setminus \mathbb{R}$ and is thus not a bifurcation point. This yields $1.$ of Theorem \ref{thm:bifurcations_gamma_1}. Assuming these signs agree, we may take the square root and consider only the positive branch to obtain $\gamma_k$ as defined in \eqref{eq:initial_critical_gamma} in the statement of Theorem \ref{thm:bifurcations_gamma_1}. This is condition $2.(b)$ of Theorem \ref{thm:bifurcations_gamma_1}.

From relations \eqref{eq:alpha_1_k} and \eqref{eq:gamma_k}, we have identified the bifurcation points of system \eqref{eq:general_system_SS} with respect to $\alpha_1 \geq 0$ and $\gamma \geq 0$ under the criteria provided in Theorems \ref{thm:bifurcations_alpha1_1} and \ref{thm:bifurcations_gamma_1}, respectively.

\subsubsection{Linear stability analysis }\label{sec:linear_stability_analysis}

The goal of this subsection is to prove Proposition \ref{prop:local_stability}. In the scalar case, as described in Section \ref{subsec:scalar_stability}, the region of linear stability is always an interval of the form $[0, \alpha^*(W))$ for some $\alpha^*(W) \in (0, +\infty]$. 

The situation for the two-species system is more complicated due to the influence of other (fixed) parameters once the bifurcation parameter of interest has been chosen. For the two-species system, we now have a region of linear stability in $3$-dimensional space, and so we seek to describe it in some detail. This is what is presented in Proposition \ref{prop:local_stability}. 

Moving forward, we fix $\sigma, L > 0$, and $\chi_i \in \{ +1, -1 \}$, and we will assume that $(\chi_1 \alpha_1, \chi_2 \alpha_2, \gamma) \in \mathbb{R} \times \mathbb{R} \times \mathbb{R}^+ \setminus \{ (z,z,z) \}$ for all $z \in \mathbb{R}$ so that $\chi_1 \alpha_1 = \chi_2 \alpha_2 = \gamma$ does not hold. Indeed, should $\chi_1 \alpha_1 = \chi_2 \alpha_2 = \gamma$ hold, one may add the two equations and define a new variable $w := ( u_1 + u_2 ) / 2$ to deduce that $w$ solves the scalar equation $0 = \sigma w_{xx} + 2 \gamma ( w ( W* w)_x )_x$. The bifurcation structure then follows from Theorems \ref{thm:local_bifurcations_scalar}-\ref{thm:local_bifurcations_scalar_2}. 

\textbf{General Criteria for Asymptotic Stability.} When identifying the critical values of $\alpha_1$ (with $(\chi_2 \alpha_2, \gamma)$ fixed) or $\gamma$ (with $(\chi_1 \alpha_1 , \chi_2 \alpha_2)$ fixed), as given in \eqref{eq:alpha_1_k} and \eqref{eq:gamma_k}, respectively, we sought parameter values for which $\lambda^{\pm} (k) = 0$ for $k \geq 1$. To describe a (possible) exchange of stability, we must first identify necessary and sufficient conditions under which $\lambda^{\pm} (k) < 0$ for all $k \geq 1$. As all eigenvalues are real, there always holds $\lambda^- < \lambda^+$, and so we ignore the smaller of the two roots, focusing only on $\lambda^+(k)$. We now prove Proposition \ref{prop:local_stability}.

\begin{proof}[Proof of Proposition \ref{prop:local_stability}]
We first note that $\Gamma_i (k) = \sigma (1 + \chi_i \alpha_i / h_k)$, from which we write \eqref{eq:same_wavenumber_disp_relation} in a more suggestive form:
\begin{align}\label{eq:system_eigenvalues_general}
    2 \lambda^{+} (k) =  - \sigma \left( \tfrac{2 \pi k}{L} \right)^2 \left( 2 + \frac{\chi_1 \alpha_1 + \chi_2 \alpha_2}{h_k} - \frac{1}{\as{h_k}} \sqrt{(\chi_1 \alpha_1 - \chi_2 \alpha_2)^2 + (2 \gamma)^2} \right).
\end{align}
The sign of $\lambda^{+}(k)$ is therefore determined by the sign of the quantity
\begin{align}\label{eq:sign_lambda_condition_1}
    2 + \frac{\chi_1 \alpha_1 + \chi_2 \alpha_2}{h_k} - \frac{1}{\as{h_k}} \sqrt{(\chi_1 \alpha_1 - \chi_2 \alpha_2)^2 + (2 \gamma)^2}.
\end{align}

A necessary condition for $\lambda^{+}(k) < 0$ for all $k \geq 1$ is
\begin{align}\label{eq:stability_aux_1}
        \sum_{i=1}^2 (1 + \chi_i \alpha_i / h_k) > 0 \quad \forall k \geq 1.
\end{align}
Clearly, if \eqref{eq:stability_aux_1} is violated for some $k^\prime$, then $\lambda^+(k^\prime) > 0$. On the other hand, \eqref{eq:stability_aux_1} is guaranteed to hold when, for example, the homogeneous state is locally stable in the associated scalar equation for each $i$ independently, which is exactly when $1 + \chi_i \alpha_i / h_k > 0$, for all $k \geq 1$, for $i = 1,2$. In general, since $h_k$ as a function of $k$ has $\ran (h_k) \subset ( - \infty, - \alpha^*(W)] \cup [ \alpha^*(-W), \infty)$, the requirement\eqref{eq:stability_aux_1} is equivalent to
$$
    0< \sum_{i=1}^2 (1 - \chi_i \alpha_i / \alpha^*(W) ) \quad \text{ and }\quad 0< \sum_{i=1}^2 (1 + \chi_i \alpha_i / \alpha^*(-W) ),
$$
or more simply
\begin{align}\label{hyp_stability_necessary}
   - \alpha^*(-W) < \frac{\chi_1 \alpha_1 + \chi_2 \alpha_2}{2} < \alpha^*(W).
\end{align}
Hence, we always assume that \eqref{hyp_stability_necessary} holds. 

From \eqref{eq:sign_lambda_condition_1}, the necessary and sufficient condition for $\lambda^+ (k) < 0$ for all $k \geq 1$ is
$$
2 + \frac{\chi_1 \alpha_1 + \chi_2 \alpha_2}{h_k} - \frac{1}{\as{h_k}} \sqrt{(\chi_1\alpha_1 - \chi_2\alpha_2)^2 + (2 \gamma)^2} > 0, \quad \quad \forall k \geq 1.
$$
Upon expansion and further simplification, we identify the necessary and sufficient condition
\begin{align}
    \lambda^\pm (k) < 0, \quad \forall k \geq 1 \iff \ \eqref{hyp_stability_necessary} \text{ and } 0 < h_k^2 + 2 h_k \left( \frac{\chi_1 \alpha_1 + \chi_2 \alpha_2}{2} \right) + (\chi_1 \chi_2 \alpha_1 \alpha_2 - \gamma^2 ), \quad \forall k \geq 1,
\end{align}
noticing the sign is opposite to what one may expect due to the common factor $- \sigma \left( \tfrac{2 \pi k}{L} \right)^2$ extracted in \eqref{eq:system_eigenvalues_general}. This yields a quadratic in the variable $h_k$, opening upwards, whose (real) roots are given by
\begin{align}\label{eq:critical_xi_parameters}
    \xi^{\pm} = \xi^{\pm}(\chi_1 \alpha_1, \chi_2 \alpha_2, \gamma) := - \left( \frac{\chi_1 \alpha_1 + \chi_2 \alpha_2}{2} \right) \pm \sqrt{ \left( \frac{\chi_1 \alpha_1 - \chi_2 \alpha_2}{2} \right)^2 + \gamma^2}.
\end{align}
Noting again that $h_k$ has $\ran (h_k) \subset ( - \infty, - \alpha^*(W)] \cup [ \alpha^*(-W), \infty)$, the largest negative value and smallest positive value attainable by $h_k$ are $- \alpha^*(W)$ and $\alpha^*(-W)$, respectively, from which we conclude that
\begin{align}\label{eq:stability_necessary_and_sufficient}
    \lambda^{+}(k) < 0 \quad \forall k \geq 1 \iff \eqref{hyp_stability_necessary} \quad \text{ and } \quad - \alpha^*(W) < \xi^- \quad \text{ and } \quad \xi^+ < \alpha^*(- W).
\end{align}

Notice that, consistent with the global asymptotic stability result of Theorem \ref{thm:global_stability_system}, since $\alpha^*(\pm W ) > 0$ and $\xi^\pm (0,0,0) = 0$, there exists $0 < r \ll 1$ so that $\lambda^+ (k) < 0$ for all $k \geq 1$ whenever $(\alpha_1, \alpha_2, \gamma) \in B_r(0) \subset \mathbb{R}^3$. 

We now simplify the last two relations of \eqref{eq:stability_necessary_and_sufficient}. The relation $- \alpha^*(W) < \xi^-$ is equivalent to
$$
\sqrt{ \left( \frac{\chi_1 \alpha_1 - \chi_2 \alpha_2}{2} \right)^2 + \gamma^2} < \alpha^*(W) - \left( \frac{\chi_1 \alpha_1 + \chi_2 \alpha_2}{2} \right),
$$
where the right-hand side is positive due to \eqref{hyp_stability_necessary}. Squaring both sides and simplifying yields
$$
0 < \alpha^*(W) [ \alpha^*(W) - \chi_1 \alpha_1 - \chi_2 \alpha_2 ] + \chi_1 \chi_2 \alpha_1 \alpha_2 - \gamma^2 = [\alpha^*(W) - \chi_1 \alpha_1][\alpha^*(W) - \chi_2 \alpha_2] - \gamma^2.
$$
Similarly, the relation $\xi^+ < \alpha^*(-W)$ is equivalent to
$$
0 < \alpha^*(-W) [ \alpha^*(-W) - \chi_1 \alpha_1 - \chi_2 \alpha_2 ] + \chi_1 \chi_2 \alpha_1 \alpha_2 - \gamma^2 = [\alpha^*(-W) + \chi_1 \alpha_1][\alpha^*(-W) + \chi_2 \alpha_2] - \gamma^2.
$$
Since both of these quantities must remain positive to maintain linear stability, i.e., since all conditions of \eqref{eq:stability_necessary_and_sufficient} must hold simultaneously, we obtain the first stability criterion of Proposition \ref{prop:local_stability} by taking the minimum of these two quantities.

To obtain the last statement of the Proposition, we argue as follows. Taking the difference between these two quantities, we find
$$
[\alpha^*(W) - \chi_1 \alpha_1][\alpha^*(W) - \chi_2 \alpha_2] - [\alpha^*(-W) + \chi_1 \alpha_1][\alpha^*(-W) + \chi_2 \alpha_2] = 0 \iff S^* = 0,
$$
where $S^*$ is as defined in \eqref{eq:S_star} in the statement of the Proposition. The minimum is then given by $[\alpha^*(W) - \chi_1 \alpha_1][\alpha^*(W) - \chi_2 \alpha_2]$ when $S^* < 0$, and is instead given by $[\alpha^*(-W) + \chi_1 \alpha_1][\alpha^*(-W) + \chi_2 \alpha_2]$ when $S^*>0$.

Finally, since $\gamma^2 \geq 0$, linear stability requires that $ \alpha^*(W) + \chi_i \alpha_i$ are either both positive or both negative. If both were negative, then $2 \alpha^*(W) < \chi_1 \alpha_1 + \chi_2 \alpha_2$, in violation of \eqref{hyp_stability_necessary}. Hence, both must be positive, and we conclude that there necessarily holds $\chi_i \alpha_i < \alpha^*(W)$ for each $i=1,2$. A similar argument yields the requirement that $- \alpha^*(-W) < \chi_i \alpha_i$ for each $i=1,2$, and the Proposition is proven.
\end{proof}

\section{Bifurcation Analysis}

In this section, we do some preliminary analysis, followed by the proofs of the main bifurcation theorems. 

\subsection{Fr{\'e}chet derivatives \& preliminary computations}

Due to the translation invariance of stationary states, we restrict our analysis to the space $L^2_s(\mathbb{T})$, the (closed) subspace of $L^2(\mathbb{T})$ comprised of even functions, whose orthonormal basis is given by \eqref{eq:orthogonal_basis_wk}. We then study the nonlinear map $\widehat G : [ L_s^2 (\mathbb{T})]^2  \times \mathbb{R} \mapsto [ L_s^2 (\mathbb{T})  ]^2$ given by $\widehat G (\mathbf{u}, \nu) := ( I - \mathcal{T})\mathbf{u} $,
and we immediately recentre $\widehat G$ via
$$
G (\mathbf{u}, \nu) := \widehat G (\mathbf{u} + \mathbf{u}_\infty, \nu)
$$
so that $G(\mathbf{0}, \nu) = 0$ for all $\nu \geq 0$, and $\nu$ is the bifurcation parameter, either $\alpha_1$ or $\gamma$. Notice that $\mathcal{T}$ is translation invariant in the sense that $\mathcal{T}(\mathbf{u} + \mathbf{u_\infty}) = \mathcal{T}\mathbf{u}$. In \eqref{eq:alpha_1_k} and \eqref{eq:gamma_k}, we have identified the valid bifurcation points for our subsequent analysis, just as we identified the valid bifurcation points for the scalar problem in \eqref{eq:critical_alpha_scalar}.

The first objects of interest will be the relevant Fr{\'e}chet derivatives of the entire nonlinear map $G$, which comes down to computing several Fr{\'e}chet derivatives of $\mathcal{T}$. Much of the explicit computations are found in Appendix \ref{sec:appendix_frechet_derivatives_n_species}, and so we compile only the necessary details here.

To this end, denote by $D_{\mathbf{u}} \mathcal{T}\mathbf{u} [\cdot,\cdot] : [ L_s^2(\mathbb{T})]^2 \mapsto [L_s^2(\mathbb{T})]^2$ the Fr{\'e}chet derivative of $\mathcal{T}$ with respect to $\mathbf{u}$ (evaluated at $\mathbf{u}$), in the direction of the variation $\boldsymbol{\eta} := (\eta_1,\eta_2) \in [L^2 _s (\mathbb{T}) ]^2$:
\begin{align}
    D_{\mathbf{u}} \mathcal{T} \mathbf{u} [\boldsymbol{\eta}] := \begin{pmatrix}
        D_{u_1} T_1 \mathbf{u} [\eta_1] + D_{u_2} T_1 \mathbf{u} [\eta_2]  \\
        D_{u_1} T_2\mathbf{u} [\eta_1] + D_{u_2} T_2\mathbf{u} [\eta_2] .
    \end{pmatrix}
\end{align}
Heuristically, one can interpret the above object as a Jacobian matrix with entries $D_{u_i} T_j \mathbf{u}$, $i,j=1,2$, acting on the variation $\bm{\eta}$. Higher order derivatives (e.g., $D^2 _{\mathbf{u} \mathbf{u}}\mathcal{T}$, $D^3 _{\mathbf{u} \mathbf{u} \mathbf{u}}\mathcal{T}$) are obtained in a similar way, by treating all combinations of higher-order Fr{\'e}chet derivatives of the subcomponents $T_i$ of $\mathcal{T}$; further details in this regard are found in Appendix \ref{sec:appendix_frechet_derivatives_n_species}.

We then identify the four elements of $D_{\mathbf{u}} \mathcal{T} \mathbf{u} [\eta_1, \eta_2]$ in terms of the map $F(\cdot \ ; W)$ introduced in the proof of Theorem \ref{thm:local_bifurcations_scalar_2} (see also Appendix \ref{frechet:special_F} and subsequent discussion) as follows. The diagonal terms are given by
$$
D_{u_i} T_i \mathbf{u}_\infty [\eta_i] = - \frac{\alpha_i}{\sigma L} F( \eta_i; \chi_i W) = - \frac{\chi_i \alpha_i}{\sigma L} F(\eta_i; W), \quad i=1,2,
$$
while the off-diagonal terms are given by
$$
D_{u_i} T_j \mathbf{u}_\infty [\eta_i] = - \frac{\gamma}{\sigma L} F( \eta_i; W ), \quad i \neq j.
$$
We can now write the full Fr{\'e}chet derivative $D_{\mathbf{u}} G (\mathbf{0}, \nu) [ \bm{\eta} ]$ as
\begin{align}\label{eq:1st_frechet_map}
    D_{\mathbf{u}} G (\mathbf{0}, \nu) [ \bm{\eta} ] &= \begin{pmatrix}
        \eta_1 + \tfrac{\chi_1 \alpha_1}{\sigma L} F(\eta_1; W) + \tfrac{\gamma}{\sigma L} F(\eta_2 ; W) \\
        \eta_2 + \tfrac{\chi_2 \alpha_2}{\sigma L} F(\eta_2; W) + \tfrac{\gamma}{\sigma L} F(\eta_1; W)
    \end{pmatrix} \nonumber \\
    &= \left[ I + \begin{pmatrix}
        \tfrac{\chi_1 \alpha_1}{\sigma L} F( \cdot \ ; W) & 0 \\
        0 & \tfrac{\chi_2 \alpha_2}{\sigma L} F( \cdot \ ; W)
    \end{pmatrix}
    + \tfrac{\gamma}{\sigma L} \begin{pmatrix}
        0 &  F(\cdot \ ; W) \\
         F(\cdot \ ; W) & 0
    \end{pmatrix} \right] \begin{pmatrix}
        \eta_1 \\ \eta_2
    \end{pmatrix}.
\end{align}
We immediately obtain the second-order mixed derivatives with respect to either $\alpha_1$ or $\gamma$:
\begin{align}\label{eq:2nd_mixed_frechet_alpha1}
    D^2_{\mathbf{u} \alpha_1} G (\mathbf{0}, \nu) [ \bm{\eta} ]
    &=  \frac{\chi_1 }{\sigma L}\begin{pmatrix}
         F(\cdot \ ; W) & 0 \\
        0 & 0
    \end{pmatrix}
     \begin{pmatrix}
        \eta_1 \\ \eta_2
    \end{pmatrix},
\end{align}
and
\begin{align}\label{eq:2nd_mixed_frechet_gamma}
    D^2_{\mathbf{u} \gamma} G (\mathbf{0}, \nu) [ \bm{\eta} ]
    &=
     \frac{1}{\sigma L} \begin{pmatrix}
        0 &  F(\cdot \ ; W) \\
         F(\cdot \ ; W) & 0
    \end{pmatrix}  \begin{pmatrix}
        \eta_1 \\ \eta_2
    \end{pmatrix}.
\end{align}
Ultimately, we only evaluate these objects in the direction of a kernel element, i.e., $\boldsymbol{\eta} = (w_k(x), cw_k(x))$ for a fixed wavenumber $k \geq 1$ and some constant $c \in \mathbb{R}$ to be determined. From the identities in \eqref{eq:F_map_identities}, this will simplify things significantly; consequently, \eqref{eq:1st_frechet_map} evaluated in the direction $\boldsymbol{\eta} = (w_k, cw_k)$ yields
\begin{align}\label{eq:1st_frechet_map_kernel_element}
    D_{\mathbf{u}} G (\mathbf{0}, \nu) [ \bm{\eta} ]
    &= \begin{pmatrix}
        1 + \tfrac{\chi_1 \alpha_1}{\sigma \sqrt{2L}}  \widetilde W(k) & \tfrac{\gamma}{\sigma \sqrt{2L}} \widetilde W (k) \\
        \tfrac{\gamma}{\sigma \sqrt{2L}} \widetilde W (k) & 1 + \tfrac{\chi_2 \alpha_2}{\sigma \sqrt{2L}} \widetilde W(k) 
    \end{pmatrix} \begin{pmatrix}
        w_k \\ c w_k
    \end{pmatrix}= \begin{pmatrix}
        1 + \tfrac{\chi_1 \alpha_1}{h_k}  & \tfrac{\gamma}{h_k} \\
        \tfrac{\gamma}{h_k} & 1 + \tfrac{\chi_2 \alpha_2}{h_k}
    \end{pmatrix} \begin{pmatrix}
        w_k \\ c w_k
    \end{pmatrix}.
\end{align}
Similarly, \eqref{eq:2nd_mixed_frechet_alpha1}-\eqref{eq:2nd_mixed_frechet_gamma} become
\begin{align}\label{eq:2nd_mixed_frechet_alpha1_kernel_element}
    D^2_{\mathbf{u} \alpha_1} G (\mathbf{0}, \nu) [ \bm{\eta} ]
    &=  \frac{\chi_1 \widetilde W(k)}{\sigma \sqrt{2L}}\begin{pmatrix}
          1 & 0 \\
        0 & 0
    \end{pmatrix}
     \begin{pmatrix}
        w_k \\ c w_k
    \end{pmatrix} = \frac{\chi_1 }{h_k}\begin{pmatrix}
          1 & 0 \\
        0 & 0
    \end{pmatrix}
     \begin{pmatrix}
        w_k \\ c w_k
    \end{pmatrix},
\end{align}
and
\begin{align}\label{eq:2nd_mixed_frechet_gamma_kernel_element}
    D^2_{\mathbf{u} \gamma} G (\mathbf{0}, \nu) [ \bm{\eta} ]
    &=
     \frac{\widetilde W (k)}{\sigma \sqrt{2L}} \begin{pmatrix}
        0 &   1 \\
         1 & 0
    \end{pmatrix}  \begin{pmatrix}
        w_k \\ c w_k
    \end{pmatrix} = \frac{1}{h_k} \begin{pmatrix}
        0 &   1 \\
         1 & 0
    \end{pmatrix}  \begin{pmatrix}
        w_k \\ c w_k
    \end{pmatrix},
\end{align}
respectively.

\subsection{Proof of bifurcation results I: existence of bifurcation points \& branch direction}\label{sec:bif_proofs_1}

In this section, we prove four key claims, from which we complete the proof of Theorems \ref{thm:bifurcations_alpha1_1} and \ref{thm:bifurcations_gamma_1}. First, we establish the two key hypotheses required to apply the bifurcation theory of Crandall-Rabinowitz. More precisely, we first seek to show the following two claims.

\textbf{Claim 1.} \textit{The Fr{\'e}chet derivative $D_{\mathbf{u}} G(\mathbf{0} , \nu)[\cdot]$ of $G(\mathbf{u}, \nu)$ is a Fredholm operator with index zero. Moreover, the kernel of $D_{\mathbf{u}} G(\mathbf{0}, \nu_k)$ is one-dimensional for any $k^* \in \mathbb{N}$ such that $\card \{ k \in \mathbb{N}: \nu_k = \nu_{k^*} \} = 1$, where $\nu_k$ is one of the bifurcation points $\alpha_{1,k}$ or $\gamma_k$ as defined in \eqref{eq:alpha_1_k} and \eqref{eq:gamma_k}, respectively.}

\textbf{Claim 2.} \textit{$D^2_{\mathbf{u}\nu}G(\mathbf{0}, \nu)[\mathbf{v}] \not \in \ran (D_{\mathbf{u}} G(\mathbf{0}, \nu) [\cdot])$, where $\nu$ is either $\alpha_{1,k}$ or $\gamma_k$, and $\mathbf{v} \in \ker (D_{\mathbf{u}} G(\mathbf{0},\nu))$ such that $\left< \mathbf{v}, \mathbf{v}\right>_{\mathbb{H}}=1$, so long as we assume that
\begin{itemize}
    \item $h_{k} + \chi_2 \alpha_2 \neq 0$ when $\nu = \alpha_{1,k}$, or else
    \item $h_{k^*} + \chi_i \alpha_i \neq 0$, $i=1,2$, and $h_{k^*} + \chi_1 \alpha_1  \neq h_{k^*} + \chi_2 \alpha_2$ when $\nu = \gamma_k$.
\end{itemize}
}

\underline{Proof of \textbf{Claim 1.}} As in the proof of \cite[Theorem 1.2]{Carrillo2020}, we first notice that $F(\cdot\ ; W) : L^2_s (\mathbb{T}) \mapsto L_s^2(\mathbb{T})$ is a Hilbert-Schmidt operator for any $W \in L^2(\mathbb{T})$ fixed. Indeed, the Hilbert-Schmidt norm of $F(\cdot \ ;W)$ is
$$
\norm{F( \cdot \ ; W)}_{\textup{HS}}^2 = \sum_{k\geq 1} \norm{F(w_k; W)}_{L^2}^2 = \frac{L}{2} \sum_{k \geq 1} \as{\widetilde W(k)}^2 < \infty,
$$
which follows from identity \eqref{eq:F_map_identities} and that $w_k$ are orthonormal. As each component $D_{u_i} T_j \mathbf{u_\infty} [ \cdot ] $ is comprised of linear combinations of $F$, they are also Hilbert-Schmidt operators on $L^2_s(\mathbb{T})$ in their own right. Hence, it is not difficult to verify that $D_{\mathbf{u}} \mathcal{T}$ is also a Hilbert-Schmidt operator on $[L_s^2(\mathbb{T})]^2$ with norm
$$
\norm{D_{\mathbf{u}} \mathcal{T}}_{\textup{HS}}^2 = \sum_{k \geq 1} \sum_{i,j=1}^2 \left( \norm{D_{u_i} T_j \mathbf{u}_\infty [ w_k ]  }_{L^2}^2 \right) < \infty.
$$
Consequently, we conclude that $D_{\mathbf{u}} \mathcal{T}$ is compact, and hence $D_{\mathbf{u}} G(\mathbf{0}, \nu) = I - D_{\mathbf{u}} \mathcal{T}$ is a Fredholm operator, since $I$ is Fredholm, and Fredholm operators are invariant with respect to compact perturbations (see, e.g., \cite[Corollary 4.3.8]{Davies2007}). Furthermore, from \eqref{eq:1st_frechet_map} we find that the mapping $\nu \mapsto D_{\mathbf{u}} G(\mathbf{0}, \nu)$ is norm-continuous:
$$
\norm{D_{\mathbf{u}} G(\mathbf{0}, \nu_1) - D_{\mathbf{u}} G(\mathbf{0}, \nu_2)} = \frac{\as{\nu_1 - \nu_2}}{\sigma L} \norm{F(\cdot\ ; W)}_{L^2(\mathbb{T})},
$$
for $\nu = \alpha_1$ or $\nu = \gamma$. Consequently, the index of $D_{\mathbf{u}} G(\mathbf{0}, \nu)$ satisfies $\ind ( D_{\mathbf{u}} G(\mathbf{0}, \nu)) = \ind (I) = 0$ by \cite[Theorem 4.3.11]{Davies2007}. Hence, $D_{\mathbf{u}} G(\mathbf{0}, \nu)$ is a Fredholm operator with index one, proving the first part of \textbf{Claim 1.}.

Next, we show that the kernel is one-dimensional. To this end, we diagonalize $D_{\mathbf{u}} G (\mathbf{0}, \nu) [ \bm{\eta} ]$ with respect to the orthonormal basis $\{ \bm{w}_{1,k}, \bm{w}_{2,k} \}_{k=1}^\infty$ introduced in \eqref{ortho_basis_twospecies} to obtain
\begin{align}\label{eq:diagaonalise_G}
    D_{\mathbf{u}} G (\mathbf{0}, \nu) [ (w_k, w_k) ] = 
    \begin{cases}
         \begin{pmatrix}
            L^{-1/2} & 0\\ 0 & L^{-1/2}
        \end{pmatrix}, \quad\quad\quad\quad k = 0, \\
    \begin{pmatrix}
        1 + \frac{\chi_1 \alpha_1 }{h_k}  & \frac{\gamma}{h_k}   \\
        \frac{\gamma }{h_k}  & 1 + \frac{\chi_2 \alpha_2}{h_k}  
    \end{pmatrix} \begin{pmatrix}
        w_k(x) \\ w_k(x)
    \end{pmatrix} , \quad \text{otherwise}
    \end{cases}
\end{align}
where $w_k(x)$ is as defined in \eqref{eq:orthogonal_basis_wk}. Subsequently, we observe that for some constant $c \neq 0$ there holds
\begin{align}
    (w_k, c w_k) \in \ker ( D_{\mathbf{u}} G (\mathbf{0}, \nu) [\cdot] ) \iff & D_{\mathbf{u}} G (\mathbf{0}, \nu) [(w_k, c w_k)] = 0 \nonumber \\
    \iff & \det (A_k) = 0 \quad \text{ and } \quad c = c_k = - (h_k + \chi_1 \alpha_1 )/\gamma,
\end{align}
where $A_k$ is the coefficient matrix obtained in \eqref{eq:diagaonalise_G}, defined for each wavenumber $k \geq 1$ such that $h_k \neq 0$.

For $\alpha_2$ and $\gamma$ fixed, we then obtain the generic condition for a bifurcation point with respect to $\alpha_1$: $\alpha_1 = \alpha_{1,k}$ if and only if $\det (A_k) = 0$.

Thus, when $\alpha_1 = \alpha_{1,k}$ as defined in \eqref{eq:alpha_1_k}, we find that 
$$
c_{\alpha_{1,k}} = - \frac{h_k + \chi_1 \alpha_{1,k} }{\gamma} = - \frac{\gamma}{(h_k + \chi_2 \alpha_2)}.
$$
Similarly, for $\alpha_i$ fixed, $i=1,2$, $\gamma = \gamma_k$ if and only if $\det (A_k) = 0$, and so we find that when $\gamma = \gamma_k$ as defined in \eqref{eq:gamma_k} there holds
$$
c_{\gamma_k} = - \sign(h_k + \chi_1 \alpha_1 ) \sqrt{\frac{h_k + \chi_1 \alpha_1}{h_k + \chi_2 \alpha_2}}
$$
\textbf{Note:} in computing $\gamma_k$, we implicitly assume that $k$ is valid in the sense that $\sign(h_k + \chi_1 \alpha_1) = \sign(h_k + \chi_2 \alpha_2)$, which is consistent with the computations to obtain the bifurcation points \eqref{eq:gamma_k}. The sign of $c_{\gamma_k}$ then depends intimately on the sign of either of these quantities, which will be relevant when discussing the phase relationship between the components of the solution at the emergent branch. 

Hence, any $\alpha_{1,k}$ or $\gamma_k$ satisfying the cardinality condition produces a unique kernel element, from which we conclude the kernel is one-dimensional, completing the proof of \textbf{Claim 1.}.

As in the scalar case, we may now decompose the space $[L^2_s(\mathbb{T}) ]^2$ as
$$
[ L_s^2(\mathbb{T}) ]^2 = \ker ( D_{\mathbf{u}} G(\mathbf{0}, \nu) ) \oplus \ran( D_{\mathbf{u}} G(\mathbf{0}, \nu) ),
$$
and we denote by $P : [ L_s^2 (\mathbb{T}) ]^2 \mapsto \Span \{ (w_k, c_\nu w_k) \}$ the orthonormal projection , where $\nu = \alpha_{1,k}$ or $\nu = \gamma_k$, and the constant $c_\nu \neq 0$ corresponding to the bifurcation point $\nu_k$ is as defined above. Different from the scalar case, however, we first normalise the kernel element as $\left< (w_k, cw_k), (w_k, c w_k) \right>_{\mathbb{H}} = 1$ does not necessarily hold. We define
$$
\mathbf{v} := c_0 ( w_k, c_\nu w_k)
$$
with $c_0= c_0(\nu) := ( 1 + c_\nu^2 )^{-1/2}$, so that $\left< \mathbf{v}, \mathbf{v} \right>_{\mathbb{H}} = 1$. The projection $P$ along $\ran( D_{\mathbf{u}} G(\mathbf{0}, \nu) )$ is then obtained via
$$
P \mathbf{z} = \left< \mathbf{z}, \mathbf{v} \right>_{\mathbb{H}} \mathbf{v} .
$$

\underline{Proof of \textbf{Claim 2.}} To prove the second claim, we require the Fr{\'e}chet derivatives with respect to $\alpha_1$ and $\gamma$. From the calculations of Appendix \ref{sec:appendix_frechet_derivatives_n_species}, we have
$$
D_\nu G(\mathbf{0}, \nu) [\boldsymbol{\eta}] = 0,
$$
for either $\nu = \alpha_1$ and $\nu = \gamma$. Then, having already computed the second-order mixed derivatives in \eqref{eq:2nd_mixed_frechet_alpha1_kernel_element}-\eqref{eq:2nd_mixed_frechet_gamma_kernel_element}, we proceed as in \cite{Carrillo2020}. Let $\nu$ be either $\alpha_1$ or $\gamma$. Since $D_{\mathbf{u} } G(\mathbf{0}, \nu)$ is a Fredholm operator, it has closed range. By the Closed Range Theorem, we have that $\ran(D_{\mathbf{u} } G(\mathbf{0}, \nu)) = \ran (I - D_{\mathbf{u}} \mathcal{T}) = \ker (I - (D_{\mathbf{u}} \mathcal{T})^*)^{\perp}$, where $(D_{\mathbf{u}} \mathcal{T})^*$ denotes the adjoint of $D_{\mathbf{u}} \mathcal{T}$. 

We first consider the $\alpha_1$ case. Let $\alpha_1 = \alpha_{1,k}$ and note that $\alpha_{1,k}$ is well-defined whenever $h_k + \chi_2 \alpha_2 \neq 0$. From the proof of \textbf{Claim 1.}, we have identified the kernel to be the linear span of $(w_k, c_\nu w_k)$ for (valid) $k \geq 1$ fixed. Similar to how we obtained \eqref{eq:2nd_mixed_frechet_alpha1_kernel_element}, plugging $\mathbf{v}$ into \eqref{eq:2nd_mixed_frechet_alpha1} we find
\begin{align}
    D_{\mathbf{u} \alpha_1}^2 G(\mathbf{0}, \alpha_{1,k})[\mathbf{v}] = c_0 \frac{ \chi_1 }{h_k} \left(  w_k (x), 0 \right),
\end{align}
and consequently, there holds
\begin{align}\label{eq:2nd_mixed_derivative_final_alpha1}
    \left< D_{\mathbf{u} \alpha_1}^2 G(\mathbf{0}, \alpha_{1,k})[\mathbf{v}], \mathbf{v} \right>_{\mathbb{H}} = \frac{c_0^2 \chi_1 }{h_k} \neq 0,
\end{align}
since $(w_k, w_k) = 1$, and it is assumed that $h_k \neq 0$.

Similarly, in the case $\nu = \gamma$, $\gamma_k$ is well-defined whenever $h_{k^*} + \chi_i \alpha_i \neq 0$, $i=1,2$, and $\chi_1 \alpha_1  \neq \chi_2 \alpha_2$. Thus, we obtain from \eqref{eq:2nd_mixed_frechet_gamma} that
\begin{align}
    D_{\mathbf{u} \gamma}^2 G(\mathbf{0}, \gamma_k)[\mathbf{v}] = c_0 \left( \frac{ c_{\gamma_k}}{h_k}  w_k(x) , \frac{1 }{h_k} w_k(x) \right),
\end{align}
and hence
\begin{align}\label{eq:2nd_mixed_derivative_final_gamma}
    \left< D_{\mathbf{u} \gamma}^2 G(\mathbf{0}, \gamma_k)[\mathbf{v}], \mathbf{v} \right>_{\mathbb{H}} = \frac{2 c_0^2 c_{\gamma_k}}{h_k} \neq 0,
\end{align}
so long as $c_{\gamma_k} \neq 0$. In either case, from \eqref{eq:2nd_mixed_derivative_final_alpha1} and \eqref{eq:2nd_mixed_derivative_final_gamma} we conclude that
$$
D_{\mathbf{u} \nu}^2 G(\mathbf{0}, \nu)[\mathbf{v}] \not \in \ker (I - (D_{\mathbf{u}} \mathcal{T})^*)^{\perp} = \ran (I - D_{\mathbf{u}} \mathcal{T}) , \quad \nu = \alpha_{1,k} \ \text{ or } \ \nu = \gamma_{k},
$$
and \textbf{Claim 2.} is proven.

We now further our results by analysing higher-order derivatives.  We show the following two additional claims.

\textbf{Claim 3.} There holds $\left< D^2 _{\mathbf{u}\mathbf{u}} G(\mathbf{0}, \nu)[ \mathbf{v}, \mathbf{v}], \mathbf{v} \right>_{\mathbb{H}} = 0$.

\underline{Proof of \textbf{Claim 3.}} In fact, using the homogeneity of the functional $F(\cdot\ ; W)$ and the formula given in \eqref{appendix:second_frechet_system_basis_form} yields the following form of the Fr{\'e}chet derivative in the direction of $[ \mathbf{v}, \mathbf{v} ]$ for any constant $c_\nu$:
\begin{align}\label{eq:second_frechet_kernel_element_alpha1}
    D^2_{\mathbf{u} \mathbf{u}} G (\mathbf{0}, \nu) [ \mathbf{v}, \mathbf{v} ] = - L \left( \frac{c_0}{h_k}\right)^2 \begin{pmatrix}
        (\chi_1 \alpha_1 + c_\nu \gamma)^2 \\ (\gamma + c_\nu \chi_2 \alpha_2)^2 
    \end{pmatrix}\left( w_k^2 - \frac{1}{L} \right).
\end{align}
Consequently, as $\int_{\mathbb{T}} w_k \dx = \int_{\mathbb{T}} w_k^3 \dx = 0$, we have that
\begin{align}
    \left<D^2_{\mathbf{u} \mathbf{u}} G (\mathbf{0}, \nu) [ \mathbf{v}, \mathbf{v} ], \mathbf{v} \right>_{\mathbb{H}} = 0,
\end{align}
and \textbf{Claim 3.} is proven. Note also that when $\gamma = 0$ so that the system is decoupled, $c_{\alpha_{1,k}} = 0$ and formula \eqref{eq:second_frechet_kernel_element_alpha1} recovers formula \eqref{eq:second_der_scalar_11} in the scalar case.

\textbf{Claim 4.} There holds $\left< D^3 _{\mathbf{u}\mathbf{u}\mathbf{u}} G(\mathbf{0}, \nu)[ \mathbf{v}, \mathbf{v}, \mathbf{v} ], \mathbf{v} \right>_{\mathbb{H}} \neq 0$. 

\underline{Proof of \textbf{Claim 4.}} Again using the homogeneity of $F(\cdot\ ; W)$ and the formula given in \eqref{appendix:third_frechet_system_basis_form}, the third Fr{\'e}chet derivative in the direction of $[\mathbf{v}, \mathbf{v}, \mathbf{v}]$ is given by
\begin{align}\label{eq:third_frechet_kernel_element_alpha1}
        D^3_{\mathbf{u} \mathbf{u} \mathbf{u}} G (\mathbf{0}, \nu) [ \mathbf{v}, \mathbf{v}, \mathbf{v} ] = L^2 \left(\frac{c_0}{h_k}\right)^3 \begin{pmatrix}
        (\chi_1 \alpha_1 + c_\nu \gamma)^3 \\ (\gamma + c_\nu \chi_2 \alpha_2)^3 
    \end{pmatrix}\left( w_k^3 - \frac{3}{L} w_k \right).
\end{align}
Unlike the second derivative, the resulting inner product will now feature even powers of $w_k$ and will be non-trivial. Using again the computation \eqref{eq:integral_third_piece}, we conclude that
\begin{align}\label{eq:final_third_frechet_generic}
    \left< D^3_{\mathbf{u} \mathbf{u} \mathbf{u}} G (\mathbf{0}, \nu) [ \mathbf{v} , \mathbf{v}, \mathbf{v}], \mathbf{v} \right>_{\mathbb{H}} &= c_0^4 L^2 h_k^{-3} \left[ (\chi_1 \alpha_1 + c_\nu \gamma)^3 + c_\nu (\gamma + c_\nu \chi_2 \alpha_2) ^3\right] \left(\int_\mathbb{T} w_k^4 \dx - \frac{3}{L} \int w_k^2 \dx \right) \nonumber \\
    &= -\frac{3}{2} c_0^4 L h_k^{-3} \left[ (\chi_1 \alpha_1 + c_\nu \gamma)^3 + c_\nu (\gamma + c_\nu \chi_2 \alpha_2) ^3\right]
\end{align}

Note again that when $\gamma=0$, formula \eqref{eq:final_third_frechet_generic} recovers formula \eqref{eq:scalar_projection_proof_1} for the scalar case. To complete the proof of the claim, we treat the cases of $\alpha_1$ and $\gamma$ separately. 

\textbf{Case I: $\alpha_1 \geq 0$.} When $\nu=\alpha_1$ so that $\alpha_1 = \alpha_{1,k}$ and $c = c_{\alpha_{1,k}}$, we find that
$$
(\chi_1  \alpha_{1,k} + \gamma c_{\alpha_{1,k}})^3 = - h_k^3, \quad \text{ and } \quad c_{\alpha_{1,k}} (\gamma + c_{\alpha_{1,k}} \chi_2 \alpha_2 )^3 = c_{\alpha_{1,k}} (- c_{\alpha_{1,k}} h_k)^3 = - h_k^3 \left(\frac{\gamma}{h_k + \chi_2 \alpha_2}  \right)^4,
$$
so that \eqref{eq:final_third_frechet_generic} at $\nu = \alpha_{1,k}$ simplifies to
\begin{align}\label{eq:final_third_frechet_alpha1}
    \left< D^3_{\mathbf{u} \mathbf{u} \mathbf{u}} G (\mathbf{0}, \alpha_{1,k}) [ \mathbf{v},\mathbf{v},\mathbf{v} ], \mathbf{v} \right>_{\mathbb{H}} &= \frac{3}{2} c_0^4 L \left[ 1 + c_{\alpha_{1,k}}^4  \right] = \frac{3}{2} c_0^4 L \left[ 1 + \left(\frac{\gamma}{h_k + \chi_2 \alpha_2}  \right)^4  \right] > 0.
\end{align}

\textbf{Case II: $\gamma \geq 0$.} Similarly, when $\nu = \gamma$ so that $\gamma = \gamma_k$ and $c = c_{\gamma_k}$, we find that
$$
(\chi_1  \alpha_1 + c_{\gamma_k} \gamma_k )^3 = - h_k^3, \quad \text{ and } \quad c_{\gamma_k} (\gamma_k + c_{\gamma_k} \chi_2 \alpha_2 )^3 = c_{\gamma_k} (- c_{\gamma_k} h_k)^3 = - h_k^3 \left(\frac{h_k + \alpha_1 \chi_1 }{h_k + \alpha_2 \chi_2 }\right)^2,
$$
so that \eqref{eq:final_third_frechet_generic} at $\nu = \gamma_k$ simplifies to
\begin{align}\label{eq:final_third_frechet_gamma}
    \left< D^3_{\mathbf{u} \mathbf{u} \mathbf{u}} G (\mathbf{0}, \gamma_k) [ \mathbf{v}, \mathbf{v}, \mathbf{v} ], \mathbf{v} \right>_{\mathbb{H}} &= \frac{3}{2} c_0^4 L \left[ 1 + c_{\gamma_k}^4  \right] = \frac{3}{2} c_0^4 L \left[ 1 + \left(\frac{h_k + \alpha_1 \chi_1 }{h_k + \alpha_2 \chi_2 }\right)^2  \right] > 0,
\end{align}
and \textbf{Claim 4.} is proven.

Amazingly, both \eqref{eq:final_third_frechet_alpha1} and \eqref{eq:final_third_frechet_gamma} are sign definite, independent of all other available parameters, which guarantees that the bifurcation is of pitchfork type. 

We now conclude with the computation of the second derivative of the bifurcating branch using equation \cite[(I.6.11)]{kielh2004bifurcation}, which reduces to computing \eqref{eq:Phi_third_derivative_definition} for the two-species system.

To this end, from \eqref{eq:final_third_frechet_generic} we immediately obtain the projection of $D^3_{\mathbf{u} \mathbf{u} \mathbf{u}} G (\mathbf{0}, \nu) [ \mathbf{v} , \mathbf{v}, \mathbf{v}]$ along the range:
\begin{align}
    P D^3_{\mathbf{u} \mathbf{u} \mathbf{u}} G (\mathbf{0}, \nu) [ \mathbf{v} , \mathbf{v}, \mathbf{v}] =\ -\frac{3}{2} c_0^4 L h_k^{-3} \left[ (\chi_1 \alpha_1 + c_\nu \gamma)^3 + c_\nu (\gamma + c_\nu \chi_2 \alpha_2) ^3\right] \mathbf{v} . 
\end{align}
Evaluation at $\nu = \alpha_{1,k}$ or $\nu = \gamma_k$ yields
\begin{align}
    P D^3_{\mathbf{u} \mathbf{u} \mathbf{u}} G (\mathbf{0}, \nu) [ \mathbf{v} , \mathbf{v}, \mathbf{v}] =\ \frac{3}{2} c_0^4 L \left[ 1 + c_\nu^4 \right] \mathbf{v} . 
\end{align}

We now proceed with the correction term. We have already identified $D^2_{\mathbf{u} \mathbf{u}} G (\mathbf{0}, \nu) [ \mathbf{v}, \mathbf{v} ]$ in \eqref{eq:second_frechet_kernel_element_alpha1}; as noted earlier, $w_k^2 - L^{-1} \sim w_{2k}$, $D^2_{\mathbf{u} \mathbf{u}} G (\mathbf{0}, \nu) [ \mathbf{v}, \mathbf{v} ]$ is orthogonal to $w_k$, and the projection contributes nothing. Therefore, $(I-P) D^2_{\mathbf{u} \mathbf{u}} G (\mathbf{0}, \nu) [  \mathbf{v}, \mathbf{v} ] = D^2_{\mathbf{u} \mathbf{u}} G (\mathbf{0}, \nu) [ \mathbf{v}, \mathbf{v} ]$.

Next, we identify $(D_{\mathbf{u}} G(\mathbf{0}, \nu))^{-1}D^2_{\mathbf{u} \mathbf{u}} G (\mathbf{0}, \nu) [ \mathbf{v}, \mathbf{v} ]$. As in the scalar case, this is guaranteed to exist since $D_{\mathbf{u}} G(\mathbf{0}, \nu)$ is an isomorphism along its range (i.e., we have removed the kernel). Hence, we seek the element $\bm{\eta} = (\eta_1, \eta_2) \in [ L_s^2(\mathbb{T}) ] ^2$ satisfying
$$
D_{\mathbf{u}} G (\mathbf{0}, \nu) [\bm{\eta}] = D^2_{\mathbf{u} \mathbf{u}} G (\mathbf{0}, \nu) [  \mathbf{v}, \mathbf{v} ].
$$
The left-hand side can be identified via \eqref{eq:1st_frechet_map_kernel_element}, while the right-hand side is found in \eqref{eq:second_frechet_kernel_element_alpha1}; together, we obtain the system
\begin{align}\label{eq:auxiliary_system_2species_case1}
\begin{cases}
    \eta_1 + \frac{\chi_1 \alpha_1}{\sigma L} F(\eta_1; W) + \frac{\gamma}{\sigma L} F(\eta_2; W) =& - \sqrt{\tfrac{L}{2}} \left( \frac{c_0}{h_k} \right)^2 ( \chi_1 \alpha_1 + c_\nu \gamma)^2 \ w_{2k} = - \sqrt{\tfrac{L}{2}} c_0^2 w_{2k}, \\
    \eta_2 + \frac{\chi_2 \alpha_2}{\sigma L} F(\eta_2; W) + \frac{\gamma}{\sigma L} F(\eta_1; W) =& - \sqrt{\tfrac{L}{2}} \left( \frac{c_0}{h_k} \right)^2 ( \gamma + c_\nu \chi_2 \alpha_2 )^2 \ w_{2k} = - \sqrt{\tfrac{L}{2}} c_0^2 c_\nu^2 w_{2k},
\end{cases}
\end{align}
where we have simplified the right-hand side using $w_k^2 - L^{-1} = w_{2k} / \sqrt{2L}$, and the identities $( \chi_1 \alpha_{1} + c_\nu \gamma) = - h_k$ and $(\gamma + c_\nu \chi_2 \alpha_2) = - c_\nu h_k$ at a bifurcation point, as used in simplifying \eqref{eq:final_third_frechet_alpha1} and \eqref{eq:final_third_frechet_gamma}. Given the form of the right-hand side of\eqref{eq:auxiliary_system_2species_case1}, we see that the only possibility is to choose $\bm{\eta} = (\beta_1 w_{2k}, \beta_2 w_{2k})$ for some $(\beta_1, \beta_2)$ to be determined. Upon substitution and simplification using the properties of the map $F$, we obtain the algebraic system
\begin{align}\label{eq:algebraic_system_1}
   M \bm{\beta} := \begin{pmatrix}
        1 + \tfrac{\chi_1 \alpha_1}{h_{2k}} & \tfrac{\gamma}{h_{2k}} \\
        \tfrac{\gamma}{h_{2k}} & 1 + \tfrac{\chi_2 \alpha_2}{h_{2k}}
    \end{pmatrix} \begin{pmatrix}
        \beta_1 \\ \beta_2
    \end{pmatrix}
    = - \sqrt{\tfrac{L}{2}} c_0^2  \begin{pmatrix}
        1 \\
       c_\nu ^2
    \end{pmatrix}.
\end{align}
We immediately observe that $\det (M) \neq 0$ since $h_k \neq h_{2k}$ by assumption, and so for all other parameters held fixed, there exists a unique vector $\bm{\beta} = (\beta_1, \beta_2)$ solving system \eqref{eq:algebraic_system_1}. Since $M$ is symmetric, its inverse is given by
$$
M^{-1} = \frac{1}{\det (M)} \begin{pmatrix}
        1 + \tfrac{\chi_2 \alpha_2}{h_{2k}} & -\tfrac{\gamma}{h_{2k}} \\
        -\tfrac{\gamma}{h_{2k}} & 1 + \tfrac{\chi_1 \alpha_1}{h_{2k}}
    \end{pmatrix}.
$$
We then define
\begin{align}\label{eq_beta_inversion}
\begin{pmatrix}
    \delta_1 \\ \delta_2
\end{pmatrix}
:=
\begin{pmatrix}
    1 + ( \chi_2 \alpha_2 - c_\nu^2 \gamma ) / h_{2k} \\
    c_\nu^2 [ 1 + ( \chi_1 \alpha_1 - c_\nu^{-2}\gamma ) / h_{2k} ] 
\end{pmatrix}
\end{align}
so that the components $(\beta_1, \beta_2)$ are given by $(\beta_1, \beta_2) = - \sqrt{\frac{L}{2}} \frac{c_0^2}{\det (M)} (\delta_1, \delta_2)$. To connect the final bifurcation formulae with the scalar case, it is fruitful to understand the limit of $M$ as a function of $h_{2k}$, particularly when $\as{\widetilde W(2k)} \to 0$ so that $\as{h_{2k}} \to +\infty$. In this case, $M = I$ is the identity matrix, and $\beta_i$ are given by the right-hand side of \eqref{eq:algebraic_system_1}.

Now that we have identified the inverse to be $\bm{\eta} = (\beta_1 w_{2k}, \beta_2 w_{2k})$, we need $(I-P) \bm{\eta}$; as before, the projection component contributes nothing by orthogonality, and so $(I-P) \bm{\eta} = \bm{\eta}$.

We now evaluate $D^2_{\mathbf{u} \mathbf{u}} G (\mathbf{0}, \nu) [ \mathbf{v}, \bm{\eta} ]$. Inserting formula \eqref{eq:second_der_hom_general} into \eqref{eq:full_second_derivative_system}, we obtain
\begin{align}
    D^2_{\mathbf{u} \mathbf{u}} G (\mathbf{0}, \nu) [  \mathbf{v}, \bm{\eta} ] =&\ - \frac{\widetilde W(2k) \widetilde W(k) c_0}{2 \sigma^2} \begin{pmatrix}
        (\beta_1 \chi_1 \alpha_1 + \beta_2 \gamma)(\chi_1 \alpha_1 + c \gamma) \\ (\beta_1 \gamma + \beta_2 \chi_2 \alpha_2)(\gamma + c \chi_2 \alpha_2)
    \end{pmatrix} w_{2k} w_k  \nonumber \\
    = &\ - \frac{c_0 L}{h_k h_{2k}} \begin{pmatrix}
        (\beta_1 \chi_1 \alpha_1 + \beta_2 \gamma)(\chi_1 \alpha_1 + c \gamma) \\ (\beta_1 \gamma + \beta_2 \chi_2 \alpha_2)(\gamma + c \chi_2 \alpha_2)
    \end{pmatrix} w_{2k} w_k ,
\end{align}
and at a bifurcation point in particular there holds
\begin{align}
    D^2_{\mathbf{u} \mathbf{u}} G (\mathbf{0}, \nu) [  \mathbf{v}, \bm{\eta} ] = \frac{c_0 L}{h_{2k} } \begin{pmatrix}
        \beta_1 \chi_1 \alpha_1 + \beta_2 \gamma \\ 
        c_\nu (\beta_1 \gamma + \beta_2 \chi_2 \alpha_2 )
    \end{pmatrix} w_{2k} w_k.
\end{align}

Finally, we compute the projection of $D^2_{\mathbf{u} \mathbf{u}} G (\mathbf{0}, \nu) [ \mathbf{v}, \bm{\eta}  ]$ along the range. This amounts to computing $(w_{2k} w_k, w_k) = 1/\sqrt{2L}$ so that upon substitution of \eqref{eq_beta_inversion} there holds
\begin{align}
    \left< D^2_{\mathbf{u} \mathbf{u}} G (\mathbf{0}, \nu) [ \mathbf{v}, \bm{\eta} ], \mathbf{v} \right>_{\mathbb{H}} =&\ \frac{c_0^2}{ h_{2k}} \ \sqrt{\frac{L}{2}} \ \left[ \beta_1 \chi_1 \alpha_1 + \beta_2 \gamma + c_\nu^2 (\beta_1 \gamma + \beta_2 \chi_2 \alpha_2)  \right]  \nonumber \\
    =&\ - \frac{L c_0^4}{2 h_{2k} \det (M)} \left[ \delta_1 \chi_1 \alpha_1 + \delta_2 \gamma + c_\nu^2 (\delta_1 \gamma + \delta_2 \chi_2 \alpha_2)  \right]
\end{align}
For simplicity, let us denote $s_0 := \left< D^2_{\mathbf{u} \mathbf{u}} G (\mathbf{0}, \nu) [ \mathbf{v}, \bm{\eta} ], \mathbf{v} \right>_{\mathbb{H}}$ as computed above, noting that this is precisely the numerator of the correction term in formula \cite[(I.6.11)]{kielh2004bifurcation} (i.e., formula \eqref{eq:Phi_third_derivative_definition}, but for the two-species case). Notice that, since $\beta_i \sim o(1)$ as $\as{h_{2k}} \to +\infty$, $\lim_{\as{h_{2k}} \to +\infty} s_0 = 0$.

We may now write the final formulae for the second derivatives $\alpha^{\prime \prime}_{1,k} (0)$ and $\gamma^{\prime \prime}_k (0)$. We first write the following generic formula for the numerator of the second derivative in terms of $\nu$, where $\nu$ is either $\alpha_{1,k}$ or $\gamma_k$:
\begin{align}
   - \frac{1}{3} \left< D^3_{\mathbf{u}\mathbf{u}\mathbf{u}} \Phi(\bm{0}, \nu)[\mathbf{v}, \mathbf{v}, \mathbf{v}], \mathbf{v}\right>_{\mathbb{H}} =&\ - \frac{1}{3} \left( \frac{3}{2} c_0^4 L \left[ 1 + c_\nu^4 \right] - 3 s_0 \right) =  -\frac{1}{2} c_0^4 L \left[ 1 + c_\nu^4 \right] + s_0 .
\end{align}
When $\alpha_1 = \alpha_{1,k}$ so that $c = c_{\alpha_{1,k}}$, the denominator of $\alpha_{1,k}^{\prime \prime} (0)$ is given by \eqref{eq:2nd_mixed_derivative_final_alpha1}; hence, the second derivative is given by
\begin{align}\label{eq:alpha_second_derivative_generalcase1}
    \alpha_{1,k}^{\prime \prime} (0) =&\ - \frac{\chi_1 c_0^2 L h_k}{2} \left[ 1 + c_{\alpha_{1,k}}^4 + \frac{( \delta_1 \chi_1 \alpha_{1,k} + \delta_2 \gamma + c_{\alpha_{1,k}}^2 (\delta_1 \gamma + \delta_2 \chi_2 \alpha_2)  )}{\det (M)} \right] .
\end{align}

Similarly, when $\gamma = \gamma_k$ so that $c = c_{\gamma_k}$, the denominator is given by \eqref{eq:2nd_mixed_derivative_final_gamma} and we obtain
\begin{align}\label{eq:gamma_second_derivative_generalcase1}
    \gamma_k^{\prime \prime}(0) =&\ - \frac{L c_0^2 h_k}{4 c_{\gamma_k}} \left[ 1 + c_{\gamma_k}^4 + \frac{( \delta_1 \chi_1 \alpha_{1} + \delta_2 \gamma_k + c_{\gamma_{k}}^2 (\delta_1 \gamma_k + \delta_2 \chi_2 \alpha_2)  )}{\det (M)} \right] 
\end{align}

Different from the scalar case, there is no immediate simplification that allows one to identify the sign of the second derivative in terms of the Fourier coefficients alone; this is due to the combined interactions between the bifurcation parameter and all other (fixed) parameters. However, we may use the intuition gained from analysis of the scalar case: if we assume that $\widetilde W(2k) = 0$ so that the resonant mode (relative to $k$) does not contribute, all contributions from the correction term $s_0$ vanish. Therefore, for the case of $\alpha_1$, we find using \eqref{eq:2nd_mixed_derivative_final_alpha1} and \eqref{eq:final_third_frechet_alpha1} the simplified formula
\begin{align}\label{eq:final_bifurcation_direction_alpha1}
    \alpha_{1,k}^{\prime \prime} (0) = - \frac{1}{3} \frac{\left< D^3_{\mathbf{u} \mathbf{u} \mathbf{u}} G (\mathbf{0}, \alpha_{1,k}) [ \mathbf{v},\mathbf{v},\mathbf{v} ], \mathbf{v} \right>_{\mathbb{H}}}{\left< D_{\mathbf{u} \alpha_1}^2 G(\mathbf{0}, \alpha_{1,k})[\mathbf{v}], \mathbf{v} \right>_{\mathbb{H}}} = - \frac{ \chi_1 c_0^2 L h_k}{2} \left[ 1 + \left( \frac{\gamma}{h_k + \chi_2 \alpha_2}\right)^4  \right].
\end{align}
Similarly, using \eqref{eq:2nd_mixed_derivative_final_gamma} and \eqref{eq:final_third_frechet_gamma}, the case of $\gamma$ yields the formula
\begin{align}\label{eq:final_bifurcation_direction_gamma}
    \gamma_k^{\prime \prime} (0) = - \frac{1}{3} \frac{\left< D^3_{\mathbf{u} \mathbf{u} \mathbf{u}} G (\mathbf{0}, \gamma_k) [ \mathbf{v} ,\mathbf{v},\mathbf{v}], \mathbf{v} \right>_{\mathbb{H}}}{\left< D_{\mathbf{u} \gamma}^2 G(\mathbf{0}, \gamma_k)[\mathbf{v},], \mathbf{v} \right>_{\mathbb{H}}} = - \frac{L c_0^2 h_k}{4 c_{\gamma_k}} \left[ 1 + \left(\frac{h_k + \alpha_1 \chi_1 }{h_k + \alpha_2 \chi_2 }\right)^2   \right].
\end{align}

Finally, we conclude with the proof of Theorems \ref{thm:bifurcations_alpha1_1} and \ref{thm:bifurcations_gamma_1}.

\begin{proof}[Proof of Theorems \ref{thm:bifurcations_alpha1_1} and \ref{thm:bifurcations_gamma_1}.]
    We provide the details for Theorem \ref{thm:bifurcations_alpha1_1}, with Theorem \ref{thm:bifurcations_gamma_1} following in an identical fashion.

    First, by \textbf{Claim 1.} and \textbf{Claim 2.} proven above, we have shown the existence of bifurcation points $\alpha_{1,k}$ by direct application of \cite[Theorem I.5.1]{kielh2004bifurcation}. Using the (stationary) bifurcation formulas of \cite[Ch. I.6]{kielh2004bifurcation} (i.e., the higher-order derivatives computed in \textbf{Claim 3.} and \textbf{Claim 4.}), we conclude that the bifurcation is of pitchfork type, with the direction of the bifurcating branch determined by \eqref{eq:alpha_second_derivative_generalcase1} as it appears in the statement of the theorem. The \textit{in particular} part follows from the simplified formula \eqref{eq:final_bifurcation_direction_alpha1} when $\widetilde W(2k) = 0$.

    The proof of Theorem \ref{thm:bifurcations_gamma_1} part 2. follows identically, using the $\gamma_k$ and $c_{\gamma_k}$ quantities computed in \textbf{Claim 1.-4.}. For \ref{thm:bifurcations_gamma_1} part 1., we simply note that when $\sign( h_k + \chi_1 \alpha_1) \neq \sign ( h_k + \chi_2 \alpha_2 )$, $\gamma_k$ is imaginary and does not lead to a bifurcation point in real space.
\end{proof}

\subsection{Points of critical stability for the two-species system}\label{sec:prep_bif_proofs_2}

The goal of this section is to extend the linear analysis of Section \ref{sec:linear_stability_analysis} so that we identify the point(s) of critical stability for the two-species system. We consider the two cases corresponding to $\alpha_1$ and $\gamma$. Once these points have been carefully described, we will conclude with the proof of Theorems \ref{thm:main_results_local_stability_alpha_system} and \ref{thm:main_results_local_stability_gamma_system} in Section \ref{sec:bif_proofs_2}. 

From the linear analysis of the scalar equation in Section \ref{subsec:scalar_stability}, the situation is relatively simple because the point of critical stability is always given by $\alpha = \alpha^*(W) > 0$, and the region of linear stability is always an interval that includes $\alpha = 0$. Therefore, we may choose $0 < \alpha \ll 1$ to guarantee linear stability of the homogeneous state and increase $\alpha$ until we hit the (necessarily unique) first bifurcation point, from which an exchange of stability occurs. We then refer to this point as the point of critical stability, which is guaranteed to exist whenever $\mathcal{K}^- \neq \emptyset$.

Since the two-species system has a $3$-dimensional region of linear stability as identified in Proposition \ref{prop:local_stability}, we must take more care in identifying a ``first'' bifurcation point. Bifurcation with respect to $\gamma$ will be emblematic of the scalar case, but bifurcation with respect to $\alpha_1$ may yield zero, one, or two points of critical stability, depending on the relative magnitudes of $\alpha^*(\pm W)$ and the other (fixed) parameters. Any exchange of stability will then depend on the number of points of critical stability, and whether we cross the boundary of the stability region by \textit{increasing} the bifurcation parameter, or by \textit{decreasing} the bifurcation parameter. 

\noindent\textbf{Guaranteeing a non-empty region of linear stability.} First, we identify the points of intersection between the upper and lower curves of Proposition \ref{prop:local_stability} occurring along the line $S^*=0$ (see the progressively darker regions of Figure \ref{fig:stability_region_1}). When two distinct points of intersection exist, the stability region is non-empty; otherwise, it is empty. To this end, notice that for $\gamma = 0$, the stability region in the $(\chi_1 \alpha_1, \chi_2 \alpha_2)$-plane is simply the rectangle $(-\alpha^*(-W), \alpha^*(W)) \times (-\alpha^*(-W), \alpha^*(W))$, which is precisely the region of stability for the scalar equation, namely, $\chi_i \alpha_i \in (- \alpha^*(-W), \alpha^*(W))$ for $\chi \in \{ \pm 1 \}$ and $\alpha_i > 0$ (the lightest green region of Figure \ref{fig:stability_region_1}). As $\gamma$ increases from $0$, the stability region in the $(\chi_1 \alpha_1, \chi_2 \alpha_2)$-plane shrinks, identified as the region enclosed by the two curves
$$
\chi_2 \alpha_2 = \alpha^*(W) - \frac{\gamma^2}{\alpha^*(W) - \chi_1 \alpha_1} \quad \text{and}\quad
\chi_2 \alpha_2 = -\alpha^*(-W) + \frac{\gamma^2}{\alpha^*(-W) + \chi_1 \alpha_1},
$$
which follows from the minimum \eqref{eq:gamma_minimum_1} obtained in Proposition \ref{prop:local_stability}. We then deduce that there exists $\overline{\gamma}>0$ so that two distinct points of intersection exist as follows. Since the stability region is bisected by the line $S^* = 0$, we substitute $\chi_1 \alpha_1 = - \chi_2 \alpha_2 + \alpha^*(W) - \alpha^*(-W)$ into the curves above and solve for $\alpha_2 = \alpha_2^{\pm} (\gamma)$ as a function of $\gamma$ to obtain the following pair of coordinates as functions of $\gamma$:
\begin{align}
    &\left( \alpha_1^\pm (\gamma), \alpha_2^\pm (\gamma) \right) = \nonumber \\
    &\left( \tfrac{(\alpha^*(W) - \alpha^*(-W))}{2}  \mp \sqrt{ \left( \tfrac{\alpha^*(W) + \alpha^*(-W)}{2} \right)^2 - \gamma^2} , \tfrac{(\alpha^*(W) - \alpha^*(-W))}{2} \pm \sqrt{ \left( \tfrac{\alpha^*(W) + \alpha^*(-W)}{2} \right)^2 - \gamma^2}  \right)
\end{align}
We again refer to Figure \ref{fig:stability_region_1} for a visual depiction of this computation. When $\gamma = 0$, we see that $(\alpha_1^\pm (0), \alpha_2^\pm(0)) = (\mp \alpha^*(\mp W), \pm \alpha^*(\pm W))$. Then, the point $\gamma>0$ at which $(\alpha_1^+(\gamma), \alpha_2^+(\gamma)) = (\alpha_1^-(\gamma), \alpha_2^-(\gamma))$ is precisely our maximal value $\overline{\gamma}$ so that the stability region is non-empty, and it is given by the mean of the two critical values for the kernel $W$: 
$$
\overline{\gamma} = \frac{\alpha^*(W) + \alpha^*(-W)}{2}.
$$
For any $\gamma > \overline{\gamma}$, the points are complex, and there are no longer any real-valued points of intersection. We therefore assume that 
\begin{align}\label{eq:stability_region_exists}
    (\chi_1 \alpha_1, \chi_2 \alpha_2, \gamma) \in (-\alpha^*(-W), \alpha^*(W)) \times (-\alpha^*(-W), \alpha^*(W)) \times (0, \overline{\gamma})
\end{align}
so that a non-empty region of linear stability exists.

Identifying a point of critical stability, either with respect to $\alpha_1$ or $\gamma$, can now be understood in terms of the behaviour of the roots $\xi^\pm$, as defined in \eqref{eq:critical_xi_parameters} as a function of $\alpha_1$ or $\gamma$, in relation to the relative magnitudes of $\alpha^*(\pm W)$. In fact, under the conditions of Proposition \ref{prop:local_stability}, the following monotonicity properties of the roots $\xi^{\pm}$ hold when all other parameters are held fixed:
\begin{align}
   & \frac{\partial \xi^\pm }{\partial \alpha_1} < 0 \quad \text{ whenever } \quad \chi_1 = +1; \quad \frac{\partial \xi^\pm }{\partial \alpha_1} > 0 \quad \text{ whenever } \quad \chi_1 = -1; \nonumber \\
   & \frac{\partial \xi^+ }{\partial \gamma} > 0 ; \quad\quad \frac{\partial \xi^- }{\partial \gamma} < 0 . \nonumber
\end{align}
The increasing/decreasing property of $\xi^{\pm}$ with respect to $\gamma$ is immediate from \eqref{eq:critical_xi_parameters}; with respect to $\alpha_1$, we see that
\begin{align}
    \frac{\partial \xi^{\pm}}{\partial \alpha_1 } = - \frac{\chi_1}{2} \left( 1 \mp \frac{\tfrac{\chi_1 \alpha_1 - \chi_2 \alpha_2}{2}}{\sqrt{\left( \tfrac{\chi_1 \alpha_1 - \chi_2 \alpha_2}{2} \right)^2 + \gamma^2}} \right),
\end{align}
and the result follows since for any $\gamma > 0$ and $\chi_1 \alpha_1 \neq \chi_2 \alpha_2$ there holds
$$
\left\lvert \frac{\tfrac{\chi_1 \alpha_1 - \chi_2 \alpha_2}{2}}{\sqrt{\left( \tfrac{\chi_1 \alpha_1 - \chi_2 \alpha_2}{2} \right)^2 + \gamma^2}} \right\rvert < 1.
$$

\noindent\textbf{Identifying points of critical stability.} We are now ready to identify points of critical stability with respect to $\alpha_1$ and $\gamma$. We begin with the second case because it is most similar to the scalar equation.

\textbf{Case II: $\gamma \geq 0$.} The case of $\gamma$ is easiest for two reasons. First, when $\gamma=0$ the problem is decoupled and the stability with respect to $\alpha_i$, $i=1,2$, is determined through the analysis of the scalar equation; this means that for $- \alpha^*(-W) < \chi_i \alpha_i < \alpha^*(W)$ fixed, $i=1,2$, the valid region of linear stability will necessarily include $\gamma = 0$. Second, the two roots $\xi^\pm$ as functions of $\gamma$ are such that as $\gamma$ increases, the smaller root $\xi^-$ is decreasing while the larger root $\xi^+$ is increasing. Consequently, there is now a race between these two roots as functions of $\gamma$, where $\xi^- (\gamma^-) = -\alpha^*(-W)$ for some $\gamma^->0$, and $\xi^+(\gamma^+) = \alpha^*(W)$ for some different $\gamma^+ >0$. The smaller of $\gamma^\pm$ determines the first point of linear instability, and so $\gamma^*$ is simply given by
$$
\gamma^* = \argmin_{\gamma^\pm} \{ \xi^-(\alpha_1, \alpha_2, \gamma^-) , \xi^+(\alpha_1, \alpha_2, \gamma^+)  \},
$$
which yields
\begin{align}\label{eq:gamma_star_special_case}
    \gamma^* = \gamma^*(\chi_1 \alpha_1, \chi_2 \alpha_2) = \begin{cases}
        \gamma_{k_W} \quad \text{ whenever } S^* < 0; \cr
        \gamma_{k_{-W}}  \quad \text{ whenever } S^* > 0,
    \end{cases}
\end{align}
where $\gamma_k$ is as defined in \eqref{eq:gamma_k}, and $S^*$ is as defined in \eqref{eq:S_star} in the statement of Proposition \ref{prop:local_stability}. This interplay between the critical wavenumber $k_{\pm W}$ and the sign of $S^*$ is what is depicted in Figure \ref{fig:system_stability_2}: the left panel corresponds to the point $P_1$, which has $S^* < 0$ and the first point of linear instability occurs at wavenumber $k_W$; the right panel corresponds to the point $P_2$, which has $S^* > 0$ and the first point of linear instability occurs at wavenumber $k_{-W}$. Both of these cases are consistent with the results of Theorem \ref{thm:bifurcations_gamma_1}; however, in the degenerate case $\alpha^*(W) - \alpha^*(-W) = \chi_1 \alpha_1 + \chi_2 \alpha_2$, i.e. when $S^*=0$, we cannot use the results of Theorem \ref{thm:bifurcations_gamma_1} as the kernel of the linearised operator is two-dimensional (given by the linear span of $w_{k_W}$ and $w_{k_{-W}}$) and the eigenvalue is no longer simple.

\textbf{Case I: $\alpha_1 \geq 0$.} Bifurcation with respect to $\alpha_1$ is more complicated than with respect to $\gamma$ because the region of linear stability may not contain $(0,0)$ in the $(\chi_1 \alpha_1, \chi_2 \alpha_2)$-plane, particularly when $\gamma$ is fixed too large (see the darkest green region in the top panel of Figure \ref{fig:system_bif_alpha1_adhesion_case}; this region is isolated away from $(0,0)$.). We therefore fix $\gamma \in (0, \overline{\gamma})$ according to \eqref{eq:stability_region_exists} so that a non-empty stability region exists, and then we fix $( \chi_1 \alpha_1, \chi_2 \alpha_2)$ within this non-empty region.

From our previous analysis of $\xi^\pm$ defined in \eqref{eq:critical_xi_parameters} and the associated monotonicity properties, we equivalently notice that $\partial \xi^\pm / \partial (\chi_1 \alpha_1) < 0$, and so $\xi^\pm$ are decreasing functions of the composite parameter $\chi_1 \alpha_1$ when $(\chi_2 \alpha_2, \gamma)$ are held fixed. As we fixed $\chi_1 \alpha_1$ in the region of (linear) stability, the conditions of \eqref{eq:stability_necessary_and_sufficient} are satisfied. Consequently, for such $(\chi_2 \alpha_2, \gamma)$ fixed, there exists two distinct values $\underline{\alpha}_1^* < \chi_1 \alpha_1 < \overline{\alpha}_1^*$ such that the homogeneous state is linearly stable for all $\chi_1 \alpha_1 \in (\underline{\alpha}_1^*, \overline{\alpha}_1^*)$ and is unstable for $\chi_1 \alpha_1 \in (-\infty, \underline{\alpha}_1^*) \cup (\overline{\alpha}_1^*, \infty)$. To see this, we argue as follows.

As $\chi_1 \alpha_1$ increases, both $\xi^\pm$ are decreasing; consequently, it is only possible to violate the second condition of \eqref{eq:stability_necessary_and_sufficient} as $\chi_1 \alpha_1$ increases. By continuity and monotonicity of $\xi^-$, we conclude that there exists a unique value $\overline{\alpha}^*_1$ such that
\begin{align}\label{alpha1_star_W_plus}
    - \alpha^*(W) = \xi^- (\overline{\alpha}_1^*, \alpha_2, \gamma),
\end{align}
so long as $\alpha^*(W) < \infty$; otherwise, we take $\overline{\alpha}_1^* = +\infty$. Notice that when $\alpha^*(W) < \infty$, \eqref{alpha1_star_W_plus} is necessarily achieved at the wavenumber associated with $\alpha^*(W)$, namely, $k=k_W$. Then, for all $\alpha_1 > \overline{\alpha}_1^*$, the second condition of \eqref{eq:stability_necessary_and_sufficient} is violated, and the linear instability follows.

Similarly, as $\chi_1 \alpha_1$ decreases, both $\xi^\pm$ are increasing functions of $\chi_1 \alpha_1$, and so it is only possible to violate the third condition of \eqref{eq:critical_xi_parameters}. Again, by the continuity and monotonicity of $\xi^+$ (now as a function of decreasing $\chi_1 \alpha_1$), there exists a unique value $\underline{\alpha}_1^*$ such that
\begin{align}\label{alpha1_star_W_minus}
    \xi^+ (\underline{\alpha}_1^*, \alpha_2, \gamma) = \alpha^*(-W),
\end{align}
so long as $\alpha^*(-W) < \infty$, otherwise we take $\underline{\alpha}_1^* = -\infty$. Again, \eqref{alpha1_star_W_minus} necessarily occurs at the wavenumber $k=k_{-W}$. As before, for all $\alpha_1 < \underline{\alpha}_1^*$, the third condition of \eqref{eq:stability_necessary_and_sufficient} is violated, and the linear instability follows.

We now see that, for $(\chi_2 \alpha_2, \gamma)$ fixed according to \eqref{eq:stability_region_exists}, the number of points of critical stability with respect to $\chi_1 \alpha_1$ depends intimately on the signs of the quantities $\overline{\alpha}_1^*$ and $\underline{\alpha}_1^*$; indeed, they could each be positive or negative. In fact, solving for $\underline{\alpha}_1^*$ and $\overline{\alpha}_1^*$ in \eqref{alpha1_star_W_plus} and \eqref{alpha1_star_W_minus}, respectively, we find that
$$
\underline{\alpha}_1^* = \alpha_{1,k_{-W}} \quad \text{ and } \quad \overline{\alpha}_1^* = \alpha_{1,k_{W}},
$$
and each could take positive or negative values. 
To make the connection with the sign of $\chi_1$ precise, we separate three cases.

\textbf{Case Ia: $\alpha_{1,k_{-W}} < 0 < \alpha_{1,k_W}$.} In this case, there exists a single point of critical stability for each $\chi_1 = -1$ and $\chi_1 = +1$, and these points respectively correspond with $\alpha_{1, k_{\pm W}}$. Precisely, we find that 
\begin{align}
    \alpha_1^* (\chi_1) = \begin{cases}
        \alpha_{1,k_W}, \quad \text{ if } \quad \chi_1 = +1; \cr
        \alpha_{1,k_{-W}} \quad \text{ if } \quad \chi_1 = -1,
    \end{cases}
\end{align}
and the homogeneous state is linearly stable for all $\alpha_1 \in [0, \alpha_1^*(\chi_1))$ and is unstable for $\alpha_1 > \alpha_1^* (\chi_1)$.

\textbf{Case Ib: $0<\alpha_{1,k_{-W}} < \alpha_{1,k_W}$.} When both points are positive, we have two points of critical stability for $\chi_1 = +1$ and no point of critical stability for $\chi_1 = -1$. That is, when $\chi_1 = +1$, the homogeneous state is linearly stable for all $\alpha_1 \in (\alpha_{1,k_{-W}}, \alpha_{1,k_W})$ and is unstable for all $\alpha_1 \in [0, \alpha_{1,k_{-W}}) \cup (\alpha_{1,k_W}, \infty)$. When $\chi_1 = -1$, the homogeneous state is unstable for all $\alpha_1 \geq 0$.

\textbf{Case Ic: $\alpha_{1,k_{-W}} < \alpha_{1,k_W} < 0$} When both points are negative, \textbf{Case Ib} is reversed: when $\chi_1 = +1$, there is no point of critical stability, while for $\chi_1 = -1$, there are two points of critical stability. That is, when $\chi_1 = -1$, the homogeneous state is linearly stable for all $\alpha_1 \in (-\alpha_{1,k_{W}}, -\alpha_{1,k_{-W}})$ and is unstable for all $\alpha_1 \in [0, -\alpha_{1,k_{W}}) \cup (-\alpha_{1,k_{-W}}, \infty)$. When $\chi_1 = +1$, the homogeneous state is unstable for all $\alpha_1 \geq 0$.

In Figure \ref{fig:stability_region_1}, we can observe \textbf{Case Ia} as follows. When $\chi_1=+1$, the darkest green region contains $\alpha_1 = 0$; by increasing $\alpha_1$ from within this region, we meet the boundary of the darkest green region and destabilise at $k_{W}$ as we are above the line $S^* = 0$.

We see another example of \textbf{Case Ia} for intermediate values of $\gamma$ in the top panel of Figure \ref{fig:system_bif_alpha1_adhesion_case}: the intermediate stability region always contains $\alpha_1 = 0$, and increasing $\alpha_1$ eventually intersects with the boundary above the line $S^* = 0$.

We then observe \textbf{Case Ib} for large values of $\gamma$ in the top panel of Figure \ref{fig:system_bif_alpha1_adhesion_case}: the darkest green stability region no longer includes $\alpha_1 = 0$, and the two points of critical stability are obtained by 1. increasing $\alpha_1$ from within the stability region until we destabilise at $k = k_{W}$ (i.e. above the line $S^* = 0$), and 2. decreasing $\alpha_1$ from within the stability region until we destabilise at $k = k_{-W}$ (i.e. below the line $S^* = 0$). 

\textbf{Case Ic} is then understood through \textbf{Case Ib}, as they are the same but reversing the sign of $\chi_1$.

\subsection{Proof of bifurcation results II: the first branch \& stability exchange}\label{sec:bif_proofs_2}

We are now ready to conclude with the proof of Theorems \ref{thm:main_results_local_stability_alpha_system} and \ref{thm:main_results_local_stability_gamma_system} using the linear analysis of Section \ref{sec:two_species_linear_stability}, paired with the bifurcation result of Theorems \ref{thm:bifurcations_alpha1_1} and \ref{thm:bifurcations_gamma_1}.

\begin{proof}[Proof of Theorems \ref{thm:main_results_local_stability_alpha_system} and \ref{thm:main_results_local_stability_gamma_system}]
We first establish that the Principle of Linearised Stability holds, as in the proof of Theorem \ref{thm:local_bifurcations_scalar_2}. In fact, this follows from an identical argument to that for the scalar equation: we write problem \eqref{eq:general_system} as an abstract semiflow of the form $\mathbf{u}^\prime (t) = A \mathbf{u} + \mathcal{G}(\mathbf{u})$, where we now work in the product space $\mathbb{H}$, and then ensure that the operator $A$ is sectorial. Indeed, the linearised operator $\mathcal{L}: [ D( \mathcal{L}) ]^2 \mapsto \mathbb{H}$ can be written as $\sigma \Delta I + B$, where $D(\mathcal{L}) = [H^2(\mathbb{T})]^2$ so that $\sigma \Delta I : D( \mathcal{L}) \mapsto \mathbb{H}$ is sectorial and $B$ is given by
$$
B\mathbf{u} = L^{-1} \begin{pmatrix} 
    \chi_1 \alpha_1 \left( W * \cdot  \right)_{xx} & \gamma \left( W * \cdot \right)_{xx} \\
    \gamma \left( W * \cdot \right)_{xx} & \chi_2 \alpha_2 \left( W * \cdot  \right)_{xx} 
\end{pmatrix} \begin{pmatrix}
        u_1 \\ u_2
    \end{pmatrix} .
$$
As in the scalar case, we find that $B: D(\mathcal{L}) \mapsto \mathbb{H}$ is compact by taking a bounded sequence $\{ \mathbf{u}^k \}_{k\geq 1} \subset D(\mathcal{L})$ and applying the Rellich-Kondrachov compactness theorem with Lemma \ref{lemma:BV_TV_embeddings} to each component. We then conclude that $\mathcal{L}$ is sectorial by \cite[Proposition 2.4.3]{lunardi1995analytic}. Next, we then define the function
$$
\mathcal{G} (\mathbf{u}) := \begin{pmatrix}
    [(u_1 - u_\infty)(\chi_1 \alpha_1 W * u_1 + \gamma W * u_2)_x]_x \\
    [(u_2 - u_\infty)(\chi_2 \alpha_2 W * u_2 + \gamma W * u_1)_x]_x 
\end{pmatrix},
$$
so that $\mathcal{G}(\mathbf{u}_\infty) = D_\mathbf{u} \mathcal{G}( \mathbf{u}_\infty ) = 0$. Therefore, by Theorems \cite[Theorem 9.1.2]{lunardi1995analytic} and \cite[Theorem 9.1.3]{lunardi1995analytic}, the Principle of Linearised Stability holds for the coupled system. Note that this property holds independently of the bifurcation parameter chosen, and we may now upgrade all notions of linear stability/instability to local asymptotic stability/nonlinear instability, respectively.

The conclusion of Theorem \ref{thm:main_results_local_stability_gamma_system} is then obtained as follows. From the analysis of Sections \ref{sec:linear_stability_analysis} and \ref{sec:prep_bif_proofs_2}, for any $\chi_i \alpha_i \in (- \alpha^*(-W), \alpha^*(W))$ fixed there exists a point of critical stability $\gamma^*>0$ so that the homogeneous state is linearly stable for $\gamma \in [0, \gamma^*)$ and is unstable for $\gamma^*>0$. More precisely, $\gamma^*$ is given by \eqref{eq:gamma_star_special_case}, as found in the statement of the theorem, and as the Principle of Linearised Stability holds, by Theorem \ref{thm:bifurcations_gamma_1}, the criticality of the emergent branch is determined by the sign of $h_{k_{\pm W}} / c_{\gamma_{k_{\pm W}}}$ when $\as{\widetilde W(2k))} \ll 1$. Explicitly, since $-\alpha^*(-W) < \chi_2 \alpha_2 < \alpha^*(W)$, we find that by formula \eqref{eq:final_bifurcation_direction_gamma} there holds
$$
- \sign \left( \gamma_{{k_W}}^{\prime \prime } (0) \right) = \sign \left( \frac{ h_{k_{W}} }{ c_{\gamma_{k_{W}}} } \right) = \sign \left( \frac{- \alpha^*(W)}{- \sign (- \alpha^*(W) + \chi_2 \alpha_2)} \right) = - \sign \left( \frac{\alpha^*(W)}{\alpha^*(W) - \chi_2 \alpha_2} \right) = - 1.
$$
and similarly, 
$$
- \sign \left( \gamma_{{k_{-W}}}^{\prime \prime } (0) \right) = \sign \left( \frac{ h_{k_{-W}} }{ c_{\gamma_{k_{-W}}} } \right) = \sign \left( \frac{ \alpha^*(-W)}{- \sign (\alpha^*(-W) + \chi_2 \alpha_2)} \right) = - \sign \left( \frac{\alpha^*(-W)}{\alpha^*(-W) + \chi_2 \alpha_2} \right) = - 1.
$$
In either case, when $\as{\widetilde W(2k))} \ll 1$ we find that $\gamma_k^{\prime \prime} (0) > 0$ at $k = k_W$ or $k = k_{-W}$, and the bifurcation is always supercritical, and the phase relationship is obtained from the sign of $c_{\gamma_{k_{\pm W}}}$. According to \cite[Theorem I.7.4]{kielh2004bifurcation}, an exchange of stability occurs at $\gamma = \gamma^*$, proving the first part of Theorem \ref{thm:main_results_local_stability_gamma_system}. Finally, when $\chi_i \alpha_i \not\in (-\alpha^*(-W), \alpha^*(W))$ for at least one $i=1,2$, \eqref{eq:stability_region_exists} is violated so that the homogeneous state is unstable for all $\gamma > 0$, and no point of critical stability exists.

The conclusion of Theorem \ref{thm:main_results_local_stability_alpha_system} follows in a similar fashion. First, by fixing $(\chi_2 \alpha_2, \gamma)$ according to \eqref{eq:alpha_bif_thm_2_cond_1}, we ensure that there always exists a non-empty region of linear stability in the $(\chi_2 \alpha_2, \gamma)$-plane, which follows from Proposition \ref{prop:local_stability} and the analysis of Section \ref{sec:prep_bif_proofs_2}, i.e., the criteria of \eqref{eq:stability_region_exists}. Cases 1.-3. in the statement of Theorem \ref{thm:main_results_local_stability_alpha_system} then follow from an application of \cite[Theorem I.7.4]{kielh2004bifurcation} by using \textbf{Cases Ia-c} of Section \ref{sec:prep_bif_proofs_2}, the Principle of Linearised Stability, and the bifurcation results of Theorem \ref{thm:bifurcations_alpha1_1}.

Case 1. is the simplest, as it is similar to the $\gamma$ case: we have a single point of critical stability for each $\chi_1 = \pm 1$, each identified in \textbf{Case Ia} of Section \ref{sec:prep_bif_proofs_2} as found in the statement of the theorem, and both branches are super- or subcritical according to formula \eqref{eq:alpha_second_derivative_generalcase1}. The \textup{in particular} part follows from the simplified formula \eqref{eq:final_bifurcation_direction_alpha1}. The corresponding phase relationships follow from relation \eqref{eq:phase_relation_alpha_1_case} of Theorem \ref{thm:bifurcations_alpha1_1}. The exchange of stability then follows from \cite[Theorem I.7.4]{kielh2004bifurcation}.

Case 2. is more complicated because there are two points of critical stability, and the direction of the bifurcation parameter changes depending on which point of critical stability you approach (from within the region of local stability). First, from the analysis of \textbf{Case Ib} of Section \ref{sec:prep_bif_proofs_2}, when $\chi_1 = -1$, there is no point of critical stability and the homogeneous state is unstable for all $\alpha_1\geq 0$. When $\chi_1 = +1$, there exists an interval of local stability, and it is given by $(\alpha_{1,k_{-W}}, \alpha_{1,k_W})$. As $\alpha_1$ passes through $\alpha_{1,k_W}$, a super- or subcritical bifurcation occurs according to formula \eqref{eq:alpha_second_derivative_generalcase1}, and an exchange of stability occurs by \cite[Theorem I.7.4]{kielh2004bifurcation} when the bifurcation is supercritical. The \textup{in particular} part follows from the simplified formula \eqref{eq:final_bifurcation_direction_alpha1}. The corresponding phase relationship follows from relation \eqref{eq:phase_relation_alpha_1_case} of Theorem \ref{thm:bifurcations_alpha1_1}. On the other hand, as $\alpha_1$ passes through $\alpha_{1,k_{-W}}$, a subcritical bifurcation occurs according to formula \eqref{eq:alpha_second_derivative_generalcase1}; however, the resulting exchange of stability must also take into account the fact that we have (implicitly) reversed the bifurcation parameter direction ($\alpha_1$ is decreasing through $\alpha_{1,k_{-W}}$), and so an exchange of stability still occurs at $\alpha_1 = \alpha_{1,k_{-W}}$ whenever the bifurcation is subcritical, which again follows from \cite[Theorem I.7.4]{kielh2004bifurcation}. The \textup{in particular} part follows from the simplified formula \eqref{eq:final_bifurcation_direction_alpha1}. The phase relationship again follows from relation \eqref{eq:phase_relation_alpha_1_case} of Theorem \ref{thm:bifurcations_alpha1_1}.

Case 3. follows in the same way as Case 2. with the signs reversed.

Finally, if \eqref{eq:alpha_bif_thm_2_cond_1} is violated, then so is \eqref{eq:stability_region_exists} so that the stability region is empty, and hence no point of critical stability exists. We conclude that the homogeneous state is unstable for all $\alpha_1 \geq 0$, and the proof is complete.
\end{proof}

\section{Acknowledgments}\label{sec:ack}
JAC was supported by the Advanced Grant Nonlocal-CPD (Nonlocal PDEs for Complex Particle Dynamics: Phase Transitions, Patterns and Synchronization) of the European Research Council Executive Agency (ERC) under the European Union Horizon 2020 research and innovation programme (grant agreement No. 883363) and partially supported by the EPSRC EP/V051121/1. JAC was partially supported by the ``Maria de Maeztu'' Excellence Unit IMAG, reference CEX2020-001105-M, funded by MCIN/AEI/10.13039/501100011033/.
YS is supported by the Natural Sciences and Engineering Research Council of Canada (NSERC Grant PDF-578181-2023). 

\section*{Data Availability}
We do not analyse or generate any datasets, because our work proceeds within a theoretical and mathematical approach.

\section*{Competing Interests}
The authors have no relevant financial or non-financial interests to disclose.


\bibliographystyle{abbrv}
\bibliography{references} 

\begin{thebibliography}{10}

\bibitem{AmbrosioFuscoPallara2000}
L.~Ambrosio, N.~Fusco, and D.~Pallara.
\newblock {\em Functions of Bounded Variation and Free Discontinuity Problems}.
\newblock Oxford Mathematical Monographs. The Clarendon Press, Oxford
  University Press, Oxford, 2000.

\bibitem{AGS08}
L.~Ambrosio, N.~Gigli, and G.~Savar\'e.
\newblock {\em Gradient flows in metric spaces and in the space of probability
  measures}.
\newblock Lectures in Mathematics ETH Z\"urich. Birkh\"auser Verlag, Basel,
  second edition, 2008.

\bibitem{boltzmann1872}
L.~Boltzmann.
\newblock Weitere studien über das wärmegleichgewicht unter gasmolekülen.
\newblock {\em Sitzungsberichte der Kaiserlichen Akademie der Wissenschaften,
  Mathematisch-Naturwissenschaftliche Classe}, 66:275--370, 1872.

\bibitem{BDFS18}
M.~Burger, M.~Di~Francesco, S.~Fagioli, and A.~Stevens.
\newblock Sorting phenomena in a mathematical model for two mutually
  attracting/repelling species.
\newblock {\em SIAM J. Math. Anal.}, 50(3):3210--3250, 2018.

\bibitem{BH21}
A.~Buttenschön and T.~Hillen.
\newblock {\em Non‑Local Cell Adhesion Models: Symmetries and Bifurcations in
  1‑D}.
\newblock CMS/CAIMS Books in Mathematics, Springer, Cham, Switzerland, 2021.

\bibitem{CCP15}
J.~A. Ca\~nizo, J.~A. Carrillo, and F.~S. Patacchini.
\newblock Existence of compactly supported global minimisers for the
  interaction energy.
\newblock {\em Arch. Ration. Mech. Anal.}, 217(3):1197--1217, 2015.

\bibitem{carrillo2019aggregation}
J.~A. Carrillo, K.~Craig, and Y.~Yao.
\newblock Aggregation-diffusion equations: dynamics, asymptotics, and singular
  limits.
\newblock {\em Active Particles, Volume 2: Advances in Theory, Models, and
  Applications}, pages 65--108, 2019.

\bibitem{Carrillo2020}
J.~A. Carrillo, R.~S. Gvalani, G.~A. Pavliotis, and A.~Schlichting.
\newblock Long-time behaviour and phase transitions for the mckean–vlasov
  equation on the torus.
\newblock {\em Arch. Rational Mech. Anal.}, 235:635--690, 2020.

\bibitem{CMV03}
J.~A. Carrillo, R.~J. McCann, and C.~Villani.
\newblock Kinetic equilibration rates for granular media and related equations:
  entropy dissipation and mass transportation estimates.
\newblock {\em Rev. Mat. Iberoamericana}, 19(3):971--1018, 2003.

\bibitem{carrillo2019adhesion}
J.~A. Carrillo, H.~Murakawa, M.~Sato, H.~Togashi, and O.~Trush.
\newblock A population dynamics model of cell-cell adhesion incorporating
  population pressure and density saturation.
\newblock {\em Journal of Theoretical Biology}, 474:14--24, 2019.

\bibitem{carrillo2024wellposedness}
J.~A. Carrillo, Y.~Salmaniw, and J.~Skrzeczkowski.
\newblock Well-posedness of aggregation-diffusion systems with irregular
  kernels, 2024.
\newblock Preprint, available at \url{https://arxiv.org/abs/2406.09227}.

\bibitem{Chazelle2017WellPosedness}
B.~Chazelle, Q.~Jiu, Q.~Li, and C.~Wang.
\newblock Well-posedness of the limiting equation of a noisy consensus model in
  opinion dynamics.
\newblock {\em Journal of Differential Equations}, 263(1):365--397, 2017.

\bibitem{CR71}
M.~G. Crandall and P.~H. Rabinowitz.
\newblock Bifurcation from simple eigenvalues.
\newblock {\em J. Functional Analysis}, 8:321--340, 1971.

\bibitem{Davies2007}
E.~Davies.
\newblock {\em Linear Operators and Their Spectra}, volume 106 of {\em
  Cambridge Studies in Advanced Mathematics}.
\newblock Cambridge University Press, Cambridge, 2007.

\bibitem{EvansGariepy2015}
L.~C. Evans and R.~F. Gariepy.
\newblock {\em Measure Theory and Fine Properties of Functions}.
\newblock Textbooks in Mathematics. CRC Press, Boca Raton, FL, revised edition,
  2015.

\bibitem{Fagan2017}
W.~F. Fagan, E.~Gurarie, S.~Bewick, A.~Howard, R.~S. Cantrell, and C.~Cosner.
\newblock Perceptual ranges, information gathering, and foraging success in
  dynamic landscapes.
\newblock {\em The American Naturalist}, 189(5):474--489, 2017.

\bibitem{FBC24}
C.~Falc\'{o}, R.~E. Baker, and J.~A. Carrillo.
\newblock A local continuum model of cell-cell adhesion.
\newblock {\em SIAM J. Appl. Math.}, 84(3):S17--S42, 2024.

\bibitem{foty2005differential}
R.~A. Foty and M.~S. Steinberg.
\newblock The differential adhesion hypothesis: a direct evaluation.
\newblock {\em Developmental Biology}, 278(1):255--263, Feb 2005.

\bibitem{giunta2022detecting}
V.~Giunta, T.~Hillen, M.~A. Lewis, and J.~R. Potts.
\newblock Detecting minimum energy states and multi-stability in nonlocal
  advection--diffusion models for interacting species.
\newblock {\em Journal of Mathematical Biology}, 85(56), 2022.

\bibitem{giunta2022local}
V.~Giunta, T.~Hillen, M.~A. Lewis, and J.~R. Potts.
\newblock Local and global existence for nonlocal multispecies
  advection-diffusion models.
\newblock {\em SIAM Journal on Applied Dynamical Systems}, 21(3), 2022.

\bibitem{giunta2024weakly}
V.~Giunta, T.~Hillen, M.~A. Lewis, and J.~R. Potts.
\newblock Weakly nonlinear analysis of a two-species non-local
  advection-diffusion system.
\newblock {\em Nonlinear Analysis: Real World Applications}, 78, 2024.

\bibitem{G24}
D.~G\'omez-Castro.
\newblock Beginner's guide to aggregation-diffusion equations.
\newblock {\em SeMA J.}, 81(4):531--587, 2024.

\bibitem{JKME23}
T.~J. Jewell, A.~L. Krause, P.~K. Maini, and E.~A. Gaffney.
\newblock Patterning of nonlocal transport models in biology: The impact of
  spatial dimension.
\newblock {\em Mathematical Biosciences}, 366:109093, 2023.

\bibitem{Jordan1998}
R.~Jordan, D.~Kinderlehrer, and F.~Otto.
\newblock The variational formulation of the fokker–planck equation.
\newblock {\em SIAM Journal on Mathematical Analysis}, 29(1):1--17, 1998.

\bibitem{jungel2022nonlocal}
A.~J{\"u}ngel, S.~Portisch, and A.~Zurek.
\newblock Nonlocal cross-diffusion systems for multi-species populations and
  networks.
\newblock {\em Nonlinear Analysis}, 219, 2022.

\bibitem{kielh2004bifurcation}
H.~Kielh{\"o}fer.
\newblock {\em Bifurcation Theory: An Introduction with Applications to PDEs}.
\newblock Number v. 156 in Applied mathematical sciences. Kluwer Academic
  Publishers, 2004.

\bibitem{lunardi1995analytic}
A.~Lunardi.
\newblock {\em Analytic Semigroups and Optimal Regularity in Parabolic
  Problems}.
\newblock Modern Birkh{\"a}user Classics. Birkh{\"a}user/Springer Basel AG,
  1995.
\newblock 2013 reprint of the 1995 original.

\bibitem{Otto2001}
F.~Otto.
\newblock The geometry of dissipative evolution equations: the porous medium
  equation.
\newblock {\em Communications in Partial Differential Equations},
  26(1-2):101--174, 2001.

\bibitem{painter2023biological}
K.~J. Painter, T.~Hillen, and J.~R. Potts.
\newblock Biological modeling with nonlocal advection-diffusion equations.
\newblock {\em Mathematical Models and Methods in Applied Sciences},
  34(1):57--107, 2024.

\bibitem{Potts2016}
J.~R. Potts and M.~A. Lewis.
\newblock {How memory of direct animal interactions can lead to territorial
  pattern formation}.
\newblock {\em Journal of the Royal Society Interface}, 13, 2016.

\bibitem{potts2016territorial}
J.~R. Potts and M.~A. Lewis.
\newblock Territorial pattern formation in the absence of an attractive
  potential.
\newblock {\em Journal of Mathematical Biology}, 72(1-2):25--46, 2016.

\bibitem{pottslewis2019}
J.~R. Potts and M.~A. Lewis.
\newblock Spatial memory and taxis-driven pattern formation in model
  ecosystems.
\newblock {\em Bulletin of mathematical biology}, 81:2725--2747, 2019.

\bibitem{wangsalmaniw2022}
H.~Wang and Y.~Salmaniw.
\newblock Open problems in pde models for knowledge-based animal movement via
  nonlocal perception and cognitive mapping.
\newblock {\em Journal of Mathematical Biology}, 86(4):71, 2023.

\end{thebibliography}

\clearpage
\appendix
\section*{Appendix}\label{sec:appendix}
\renewcommand{\thesection}{A.\arabic{section}}
\setcounter{section}{0}

Here we compile some proofs and other technical computations used in the main text.

\section{\ Functions of Bounded Variation and a Technical Lemma}\label{sec:BV_functions}

First, we introduce the notion of bounded variation, referring to functions whose weak derivatives are Radon measures. While the present work focuses on dimension $d=1$, we state them for the general $d$-dimensional case. We say that a function $W \in L^1(\mathbb{T}^d)$ has \textit{Bounded Variation} in $\mathbb{T}^d$, written $W \in \textup{BV}(\mathbb{T}^d)$, if
\begin{align}\label{eq:total_variation}
    \sup \left\{ \int_{\mathbb{T}^d} W \Div \phi \ \dx \ \bigr\vert \ \phi \in C_c^1 (\mathbb{T}^d), \norm{\phi}_{L^\infty(\mathbb{T}^d)} \leq 1 \right\} < \infty.
\end{align}
Equivalently, if $W \in \textup{BV}(\mathbb{T}^d)$, then there exists a Radon measure $\mu$ on $\mathbb{T}^d$ and a $\mu$-measurable function $D W : \mathbb{T}^d \mapsto \mathbb{R}^d$ such that $| DW (x) | = 1$ $\mu$-a.e., and there holds $\int_\mathbb{T}^d W \Div \phi \ \dx = - \int_{\mathbb{T}^d} \phi \cdot DW \ \textup{d}\mu$ for all $\phi \in C_c^1(\mathbb{T}^d)$. The quantity \eqref{eq:total_variation} is the \textit{Total Variation} of $W$, which we denote $\norm{DW}_{\textup{TV}}$. 

The set of functions $\textup{BV}(\mathbb{T}^d)$ is a Banach space equipped with the norm
\begin{align}
    \norm{W}_{\textup{BV}(\mathbb{T}^d)} := \norm{W}_{L^1(\mathbb{T}^d)} + \norm{DW}_{\textup{TV}(\mathbb{T}^d)}.
\end{align}
We refer to \cite[Ch. 5]{EvansGariepy2015} and \cite[Ch. 3]{AmbrosioFuscoPallara2000} for further properties of the space $\textup{BV}$.

Our motivation is the following Lemma, which is used several times throughout the manuscript. 

\begin{lemma}\label{lemma:BV_TV_embeddings}
    Suppose $W \in \textup{BV}( \mathbb{T}^d ) \cap L^\infty( \mathbb{T}^d)$ is even with (distributional) derivative $DW = (D_1, \ldots, D_d W)$. Then, given $u \in L^p (\mathbb{T}^d)$ for some $1 \leq p \leq \infty$ there holds
    \begin{enumerate}
        \item $W * u \in L^\infty( \mathbb{T}^d)$ with $\norm{W * u}_{L^\infty(\mathbb{T}^d)} \leq \norm{W}_{L^\infty(\mathbb{T}^d)} \norm{u}_{L^1(\mathbb{T}^d)}$;
        \item $W * u \in W^{1,p}(\mathbb{T}^d)$ with $\norm{\grad (W * u)}_{L^p(\mathbb{T}^d)} \leq \norm{DW}_{\textup{TV}(\mathbb{T}^d)} \norm{u}_{L^p (\mathbb{T}^d)}$.
    \end{enumerate}
\end{lemma}
\begin{proof}[Proof of Lemma \ref{lemma:BV_TV_embeddings}]
    For $1.$, it is enough to apply Young's convolution inequality for $L^p$-functions.

    For $2.$, one notes that since $W \in \textup{BV}(\mathbb{T}^d)$ has a distributional derivative $DW$ given as a Radon measure, we first identify the distributional derivative $(W * u )_{x_i} = D_i W * u$. Young's inequality for measures with $\mu = D_i W$ then implies that in fact $(W * u)_{x_i} \in L^p(\mathbb{T}^d)$ and the result follows.
\end{proof}

\begin{remark*}
    We consider this notion as it is the most appropriate setting to allow kernels of the form $W(x) := \mathbb{I}_E(x)$, the indicator function of some (sufficiently regular) open subset $E \subset \mathbb{T}^d$, particularly if $E$ is a cube or ball \cite[Sec. 2]{carrillo2024wellposedness}. This naturally includes the top-hat detection function and is compatible with notions of compactness required from the time-dependent problem.
\end{remark*}

\section{\ Fr{\'e}chet derivatives, scalar case}

A key aspect of our analysis is an effective evaluation of the Fr{\'e}chet derivatives of our nonlinear map. First, we recall the nonlinear map $T:L^2(\mathbb{T}) \mapsto L^2(\mathbb{T})$ for the scalar case:
$$
Tu := \frac{\exp ( - \tfrac{\alpha}{\sigma} W * u ) }{Z(u)},
$$
where
$$
Z(u) := \int_\mathbb{T} \exp ( - \tfrac{\alpha}{\sigma} W * u ) \, \dx 
$$
is the normalisation factor. Note carefully that this is \textit{not exactly} the map we work with for the bifurcation analysis; we use the map $(I - T)u$ for the bifurcation analysis. The point is that the difficulty arises when computing the Fr{\'e}chet derivative of $T$ alone, and so we focus our efforts there.

We now recall the notion of Fr{\'e}chet derivative. Given an operator $T : X \mapsto X$, we say that $T$ is \textit{Fr{\'e}chet differentiable} at $u \in X$ if there exists a bounded linear operator $A: X \mapsto X$ such that
$$
\lim_{\norm{\phi}_X \to 0} \frac{\norm{T(u + \phi) - T(u) - A\phi}_X}{\norm{\phi}_X} = 0.
$$
When such an operator $A$ exists, we write $D_u Tu = A$, and we then write $D_u Tu [\phi] = A[\phi]$ to mean $D_u Tu$ evaluated at some $u \in X$, acting on a variation in the direction $\phi$. Our goal is to compute up to and including the third Fr{\'e}chet derivative of our nonlinear map $T$.

To this end, given a fixed kernel $W$ satisfying Hypothesis \textbf{\ref{hyp:kernel_shape}}, we define the linear map $F : L^2 (\mathbb{T}) \mapsto L^2 (\mathbb{T})$ by
\begin{align}\label{frechet:special_F}
    F(\eta ; W) := W * \eta (x) - \frac{1}{L} \int_\mathbb{T} W * \eta \, \dy,
\end{align}
where $\eta \in L^2$ is either constant or a mean-zero variation function. As $F$ is linear, there obviously holds $F(c \eta; W) = c F(\eta ; W)$ for any constant $c \in \mathbb{R}$. This map has two useful identities. First, for any constant $c \in \mathbb{R}$ there holds
\begin{align}\label{frechet:special_f_identity1}
    F(c; W) = 0.
\end{align}
Second, given any element of our orthonormal basis $\{ w_k \}_{k \in \mathbb{N}} \subset L_s ^2(\mathbb{T})$, there holds
\begin{align}\label{frechet:special_f_identity2}
    F(w_k ; W) = \sqrt{\frac{L}{2}} \widetilde W (k) w_k (x),
\end{align}
where $\widetilde W(k)$ is the $k^{\textup{th}}$ Fourier coefficient of the kernel $W$ as defined in \eqref{eq:cos_transform_basis_function}. Relation \eqref{frechet:special_f_identity2} follows directly from the homogeneity of $F(W,\cdot)$, the property that $W*w_k = \sqrt{\tfrac{L}{2}} \widetilde W(k) w_k(x)$, and $\int_\mathbb{T} w_k(x) \dx = 0$. 

Given the exponential nature of the map, the strategy moving forward is to compute the higher-order derivatives in terms of the first derivative through a recursive procedure; this is our primary motivation for introducing the function $F$ in \eqref{frechet:special_F}: it appears in the Fr{\'e}chet derivatives of the nonlinear map $T$ and will simplify subsequent computations nicely. 

\subsection{Derivatives w.r.t. \texorpdfstring{$u$}{}}

As we require up to and including the third Fr{\'e}chet derivative, it does not make sense to evaluate the Fr{\'e}chet derivatives at some fixed $u$ prematurely; therefore, we consider three stages. First, we compute the derivatives before evaulation at some element $u \in L^2$; second, we evaluate these derivatives at the homogeneous state $u = u_\infty$; third, we evaluate these derivatives in the direction of the basis vectors $w_k$, as this is what is required in the proof of our results.

\textbf{Pre-evaluation:} We first compute up to the third derivative before evaluating at the homogeneous state. As our nonlinear map is given by several compositions of smooth functions, we really only require the Fr{\'e}chet derivative of the term $W * u$. Most generally, given $W \in L^1(\mathbb{T})$ fixed, Young's inequality gives that $W* \cdot : L^2 (\mathbb{T}) \mapsto L^2 (\mathbb{T})$ is a bounded linear operator, and so its Fr{\'e}chet derivative is simply $W * \phi$. Then, the first derivative of $T$ is given by
\begin{align}\label{eq:scalar_frechet_D_uTu}
    D_u Tu [ \eta ] &= - \tfrac{\alpha}{\sigma} Tu W * \eta - \frac{\exp (-\tfrac{\alpha}{\sigma} W * u )}{z^2(u)} \left( - \tfrac{\alpha}{\sigma} \int_\mathbb{T} \exp ( - \tfrac{\alpha}{\sigma} W * u ) \, W * \eta \, \dx \right) \nonumber \\
    &= - \tfrac{\alpha}{\sigma} Tu \left( W * \eta - \int_\mathbb{T} Tu \, W * \eta \, \dx \right) .
\end{align}
Before moving to higher-order derivatives, we make note of the following Proposition establishing the Lipschitz continuity of the map $T$ on balls in $L^2$, which will utilise the derivative obtained in \eqref{eq:scalar_frechet_D_uTu}.
\begin{proposition}\label{prop:T_is_Lipschitz}
    Fix $r>0$ and a kernel $W$ satisfying hypothesis \textbf{\ref{hyp:kernel_shape}}. Then, the operator $T: L^2 (\mathbb{T}) \mapsto L^2 (\mathbb{T})$ is Lipschitz on $E_r := \{ u \in L^2 (\mathbb{T})\ : \ \norm{u}_{L^2(\mathbb{T})} \leq r \}$ in the sense that there exists a constant $M = M(r)>0$ so that
    $$
\norm{Tu - Tv}_{L^2 (\mathbb{T})} \leq M \norm{u - v}_{L^2 (\mathbb{T})}, \quad \forall u,v \in E_r. 
    $$
\end{proposition}
\begin{proof}[Proof of Proposition \ref{prop:T_is_Lipschitz}]
    First, note that since $W \in L^\infty$, for any $u \in E_r$ there holds $\norm{W * u}_{L^\infty} \leq \norm{W}_{L^\infty} \sqrt{L} r$. It follows that
    $$
\exp ( - \alpha \norm{W}_{L^\infty} \sqrt{L} r ) \leq \exp( - \alpha W * u ) \leq \exp ( \alpha \norm{W}_{L^\infty} \sqrt{L} r ).
    $$
    Hence, we find
    $$
Z(u) \geq L \exp ( - \alpha \norm{W}_{L^\infty} \sqrt{L} r )
    $$
    from which we conclude that in $E_r$ there holds
    $$
\norm{Tu}_{L^\infty}\leq L^{-1} \exp (2 \alpha \norm{W}_{L^\infty} \sqrt{L} r) .
    $$
    Turning to the Fr{\'e}chet derivative obtained in \eqref{eq:scalar_frechet_D_uTu}, for $u \in E_r$ and $\eta \in L^2(\mathbb{T})$ we obtain the estimate
    \begin{align}
\norm{D_u Tu [ \eta ]}_{L^2} \leq \tfrac{\alpha}{\sigma} \norm{Tu}_{L^\infty} \norm{W}_{L^2} \sqrt{1 + L \norm{Tu}_{L^\infty}^2} \norm{\eta}_{L^2},
    \end{align}
    and so we have shown that $\norm{D_u Tu [ \cdot ]}_{\textup{op}} \leq M$ uniformly in $E_r$, where $M$ depends only on $\alpha$, $\sigma$, $L$, $\norm{W}_{L^\infty}$ and $r$. The conclusion then follows from the mean value theorem applied to the map $T$ in $E_r$, completing the proof.
\end{proof}

We now compute the higher-order derivatives. Noticing that $u$ appears in $D_uTu$ only in terms of $Tu$, the second derivative is simply calculated via the product rule and applying the first-order derivative obtained:
\begin{align}
    D^2_{uu} Tu [ \eta, \theta ] &= -\tfrac{\alpha}{\sigma} D_u Tu [ \theta ] \left( W * \eta - \int_\mathbb{T} Tu \, W * \eta \, \dx \right) + \tfrac{\alpha}{\sigma} Tu \left( \int_\mathbb{T} D_u Tu [\theta] \, W * \eta \, \dx \right) .
\end{align}
Notice also that in calculating the second derivative, $\eta$ is now held fixed and an independent variation $\theta$ is introduced. 

Continuing in this fashion, the third derivative is calculated with a variation $\phi$ as follows.
\begin{align}
    D^3_{uuu} Tu [ \eta, \theta, \phi ] =&\ -\tfrac{\alpha}{\sigma} D^2_{uu} Tu [ \theta, \phi ] \left( W * \eta - \int_\mathbb{T} Tu \, W * \eta \, \dx \right) +\tfrac{\alpha}{\sigma} D_u Tu [ \theta ] \int_\mathbb{T} D_u Tu [\phi] \, W * \eta \, \dx  \nonumber \\
    &\ + \tfrac{\alpha}{\sigma} D_u Tu [\phi] \left( \int_\mathbb{T} D_u Tu [\theta] \, W * \eta \, \dx \right) + \tfrac{\alpha}{\sigma} Tu \left( \int_\mathbb{T} D^2_{uu} Tu [\theta, \phi] \, W * \eta \, \dx \right) .
\end{align}
When evaluating at a particular state $u$, we may now simplify all higher-order derivatives in terms of the first one.

\textbf{Post-evaluation:} We now evaluate all derivatives at the homogeneous state $u = u_\infty = L^{-1}$. The first derivative at $u = u_\infty$ is
\begin{align}\label{app:first_frechet_scalar_preeval}
    D_u Tu [ \eta ] \biggr\vert_{u = u_\infty} =& - \frac{\alpha}{\sigma} \frac{1}{L} \left( W * \eta - \frac{1}{L}\int_\mathbb{T}  W * \eta \, \dx \right) = -\frac{\alpha}{\sigma L} F(\eta; W ),
\end{align}
where $F$ is as defined in \eqref{frechet:special_F}. 

The second derivative at $u = u_\infty$ is
\begin{align}\label{eq:second_der_scalar_aux}
    D^2_{uu} Tu [ \eta, \theta ] \biggr\vert_{u=u_\infty} &= -\tfrac{\alpha}{\sigma} D_u Tu [ \theta ] \biggr\vert_{u=u_\infty} \cdot F(\eta; W) + \frac{\alpha}{\sigma L} \left( \int_\mathbb{T} D_u Tu [\theta]\biggr\vert_{u=u_\infty} \cdot  W * \eta \, \dx \right) \nonumber \\
    &= \frac{\alpha^2}{\sigma^2 L} \left( F(\theta; W) F(\eta; W) - \frac{1}{L} \int_\mathbb{T} F( \theta ; W) W * \eta \ \dx \right).
\end{align}

The third derivative at $u = u_\infty$ is
\begin{align}
        D^3_{uuu} Tu [ \eta, \theta, \phi ] \biggr\vert_{u=u_\infty} &= -\tfrac{\alpha}{\sigma} D^2_{uu} Tu [ \theta, \phi ] \biggr\vert_{u=u_\infty} \cdot F(\eta; W) \nonumber \\
        &\quad +\tfrac{\alpha}{\sigma} D_u Tu [ \theta ] \biggr\vert_{u=u_\infty} \cdot\int_\mathbb{T} D_u Tu [\phi] \biggr\vert_{u=u_\infty}\cdot W * \eta \, \dx  \nonumber \\
    &\quad + \tfrac{\alpha}{\sigma} D_u Tu [\phi] \biggr\vert_{u=u_\infty} \cdot \left( \int_\mathbb{T} D_u Tu [\theta] \biggr\vert_{u=u_\infty} \cdot W * \eta \, \dx \right) \nonumber \\
    & \quad + \tfrac{\alpha}{\sigma} \frac{1}{L} \left( \int_\mathbb{T} D^2_{uu} Tu [\theta, \phi] \biggr\vert_{u=u_\infty} \cdot W * \eta \, \dx \right) \nonumber \\
    &= - \frac{\alpha^3}{\sigma^3 L} F(\eta ; W) \left( F(\phi ; W) F(\theta ; W) - \frac{1}{L} \int_\mathbb{T} F(\phi; W) W * \theta \dx  \right) \nonumber \\
    &\quad + \frac{\alpha^3}{\sigma^3 L^2} F(\theta ; W) \int_\mathbb{T} F(\phi; W) W * \eta \dx \nonumber \\
    &\quad + \frac{\alpha^3}{\sigma^3 L^2} F(\phi ; W) \int_\mathbb{T} F(\theta ; W) W * \eta \dx \nonumber \\
    &\quad + \frac{\alpha^3}{\sigma^3 L^2} \int_\mathbb{T} \left( F(\phi ; W) F( \theta ; W) - \frac{1}{L} \int_\mathbb{T} F(\phi ; W) W * \theta \dy  \right) W * \eta \dx  \nonumber \\
    &= - \tfrac{\alpha^3}{\sigma^3 L} \left[ F(\eta ; W) F(\theta ; W) F(\phi ; W)- \frac{1}{L} \left( F(\eta ; W) \int_\mathbb{T} F( \phi ; W) W * \theta \dx \right. \right. \nonumber \\  
    &\quad \left. \left. + F(\theta ; W) \int_\mathbb{T} F(\phi ; W) W * \eta \dx + F(\phi ; W) \int_\mathbb{T} F(\theta ; W) W * \eta \dx   \right) \right. \nonumber \\
    &\quad \left. - \frac{1}{L} \int_\mathbb{T} \left( F( \phi ; W) F(\theta ; W) - \frac{1}{L} \int_\mathbb{T} F(\phi ; W) W * \theta \dy  \right) W * \eta \dx \right].
\end{align}

\textbf{Variations of basis vector form:} These formulas will simplify even further upon evaluation at a variation of the form $w_k$ for $k \in \mathbb{N}$ via identity \eqref{frechet:special_f_identity2}; this is precisely what we require in our bifurcation analysis, and so we present them here now. Fix $k \in \mathbb{N}$ and recall that $\int_\mathbb{T} w_k (x) \dx = 0$ and $\int_\mathbb{T} w_k ^2(x) \dx = 1$. 

For the first derivative, we have
\begin{align}\label{app:first_derivative_scalar}
    D_u Tu [ w_k ] \biggr\vert_{u = u_\infty} &= -\frac{\alpha}{\sigma L} F(w_k ; W) = - \frac{\alpha}{\sigma \sqrt{2L}} \widetilde W(k) w_k (x).
\end{align}

For the second derivative, we have
\begin{align}\label{app:second_derivative_scalar}
    D^2_{uu} Tu [ w_k, w_k ] \biggr\vert_{u=u_\infty}
    &= \frac{\alpha^2}{\sigma^2 L} \left( F^2( w_k ; W) - \frac{1}{L} \int_\mathbb{T} F^2 ( w_k ; W)  \dx \right) = \frac{\alpha^2 \magg{\widetilde W(k)}}{2 \sigma^2} \left( w_k^2 (x) - \frac{1}{L} \right) .
\end{align}

For the third derivative, we have
\begin{align}\label{app:third_derivative_scalar}
    D^3_{uuu} Tu [ w_k, w_k, w_k ] \biggr\vert_{u=u_\infty} &= - \tfrac{\alpha^3 }{\sigma^3 L} \left[ F^3 ( w_k ; W) - \frac{3}{L} F( w_k ; W) \int_\mathbb{T} F( w_k ; W) W * w_k \dx  \right. \nonumber \\
   &\quad  \left. - \frac{1}{L} \int_\mathbb{T} \left( F^2 ( w_k ; W) - \frac{1}{L} \int_\mathbb{T} F(w_k ; W) W * w_k \dy  \right) W * w_k \dx \right] \nonumber \\
   &= - \tfrac{\alpha^3 }{\sigma^3 L} \cdot \tfrac{L^{3/2}}{2^{3/2}} \as{\widetilde W(k)}^3 \left[ w_k^3(x)  - \frac{3}{L} w_k(x) \right]
\end{align}

\subsection{Derivatives w.r.t. parameters}
We also require derivatives with respect to the parameter(s) of interest for the bifurcation analysis. These are easier since they are closer to taking a regular derivative. In the scalar case, we will focus only on the parameter $\alpha$; one might also consider $\sigma$ or $L$, for example.

\textbf{Pre-evaluation:} The first derivative of $Tu$ with respect to $\alpha$ is
\begin{align}\label{eq:scalar_frechet_D_alphaTu}
    D_\alpha Tu &= Tu \left( - \tfrac{1}{\sigma} W * u \right) - \frac{\exp (-\tfrac{\alpha}{\sigma} W * u )}{z^2(u)} \left( - \tfrac{1}{\sigma} \int_\mathbb{T} \exp ( - \tfrac{\alpha}{\sigma} W * u ) \, W * u \, \dx   \right) \nonumber \\
    &= - \tfrac{1}{\sigma} Tu \left( W * u - \int_\mathbb{T} Tu \, W * u \, \dx \right).
\end{align}
The mixed derivative is obtained directly from \eqref{eq:scalar_frechet_D_uTu}:
\begin{align}
   D^2 _{u \alpha} Tu [ \eta ] = - \tfrac{1}{\sigma} Tu \left( W * \eta - \int_\mathbb{T} Tu \, W * \eta \, \dx \right)
\end{align}
From the smoothness of the map $T$, it is immediate that $D^2 _{u \alpha} Tu [ \eta ] = D^2 _{\alpha u} Tu [ \eta ]$.

\textbf{Post-evaluation:} We now evaluate at the homogeneous state $u = u_\infty$. The first derivative is
\begin{align}
   D_\alpha Tu \biggr\vert_{u=u_\infty} = - \frac{1}{\sigma L} F( L^{-1} ; W) = 0,
\end{align}
which follows from identity \eqref{frechet:special_f_identity1}.

The mixed derivative is
\begin{align}
    D^2 _{u \alpha} Tu [ \eta ] \biggr\vert_{u=u_\infty} = - \frac{1}{\sigma L} F( \eta ; W).
\end{align}

\textbf{Variations of basis vector form:} Finally, we evaluate at a basis vector $w_k$. The first derivative is always zero; the second derivative is
\begin{align}\label{app:mixed_derivative_scalar}
    D^2 _{u \alpha} Tu [ w_k ] \biggr\vert_{u=u_\infty} = - \frac{1}{\sigma \sqrt{2L}} \widetilde W(k) w_k (x).
\end{align}

\section{\ Fr{\'e}chet derivatives, \texorpdfstring{$n$-}{}species case}\label{sec:appendix_frechet_derivatives_n_species}

We now collect the required information for the two-species case. We present the Fr{\'e}chet derivatives in a very general setting as this appears to be the easiest approach notation-wise. For simplicity, we will assume that $\alpha_{ij} W_{ij} = \alpha_{ji} W_{ji}$ for all $i,j=1,\ldots, n$. Denote by $\mathbf{u} = (u_1, \ldots, u_n)$. Given a Banach space $X$, we write $X^n = X \times X \times \cdots \times X$, where there are $n$ copies of $X$. 

In general, stationary solutions of \eqref{eq:general_system} in the $n$-species case under the detailed-balance condition \eqref{condition:detailed_balance} can be identified by fixed points of the map $\mathcal{T} : X^n \mapsto X^n$ via
$$
\mathcal{T} \mathbf{u} := ( T_1 \mathbf{u}, \ldots, T_n \mathbf{u})
$$
where
$$
T_i \mathbf{u} := \frac{\exp \left( - \tfrac{1}{\sigma} \sum_{j=1}^n \alpha_{ij} W_{ij} * u_j \right)}{Z_i (\mathbf{u})}
$$
and
$$
Z_i(\mathbf{u}) := \int_\mathbb{T} \exp \left( - \tfrac{1}{\sigma} \sum_{j=1}^n \alpha_{ij} W_{ij} * u_j \right) \, \dx ,
$$
where we work in the Banach space $X=L^2(\mathbb{T})$.

\subsection{Derivatives w.r.t. \texorpdfstring{$u_k$}{}}\label{sec:frechet_derivatives_section}

The derivative of $\mathcal{T}$, which we denote by $D_{\mathbf{u}} \mathcal{T} \mathbf{u} [ \bm{\eta} ]$ for a variation $\bm{\eta} = ( \eta_1, \ldots, \eta_n)$ is resemblant of the Jacobian and takes the following form:
\begin{align}
    D_{\mathbf{u}} \mathcal{T} \mathbf{u} [ \bm{\eta} ] := \begin{pmatrix}
        D_{u_1} T_1 \mathbf{u} [\eta_1] + \ldots +  D_{u_n} T_1\mathbf{u}  [\eta_n]\\
        \vdots \\
        D_{u_1} T_n\mathbf{u}  [\eta_1] + \ldots + D_{u_n} T_n\mathbf{u}  [\eta_n]
    \end{pmatrix} 
    = \begin{pmatrix}
        \sum_{k=1}^n D_{u_k} T_1 \mathbf{u} [\eta_k] \\
        \vdots \\
        \sum_{k=1}^n D_{u_k} T_n \mathbf{u} [\eta_k]
    \end{pmatrix}
\end{align}
This pattern continues for higher order derivatives, which take the following forms:
\begin{align}\label{eq:full_second_derivative_system}
    D^2 _{\mathbf{u} \mathbf{u}} \mathcal{T} \mathbf{u} [ \bm{\eta}, \bm{\theta} ] : = \begin{pmatrix}
        \sum_{k,l =1}^n D^2_{u_k u_l} T_1 \mathbf{u} [\eta_k, \theta_l] \\
        \vdots \\
        \sum_{k,l=1}^n D_{u_k u_l} T_n \mathbf{u} [\eta_k, \theta_l]
    \end{pmatrix}
\end{align}
and
\begin{align}
    D^3 _{\mathbf{u} \mathbf{u} \mathbf{u}} \mathcal{T} \mathbf{u} [ \bm{\eta}, \bm{\theta}, \bm{\phi} ] : = \begin{pmatrix}
        \sum_{k,l,m =1}^n D^2_{u_k u_l u_m} T_1 \mathbf{u} [\eta_k, \theta_l, \phi_m] \\
        \vdots \\
        \sum_{k,l,m=1}^n D_{u_k u_l u_m} T_n \mathbf{u} [\eta_k, \theta_l, \phi_m].
    \end{pmatrix}
\end{align}
Therefore, we need only compute a general formula for $D_{u_k} T_i \mathbf{u} [\eta_k]$, $D^2_{u_k u_l} T_i \mathbf{u}[\eta_k, \theta_l]$ and $D^3_{u_k u_l u_m} T_i\mathbf{u} [\eta_k, \theta_l, \phi_m]$ for all combinations of $i,k,l,m$ (note: we reserve the index $j$ for the sums appearing within each instance of $T_i$). These will follow relatively easily from the formulas derived in the scalar case, modulo bookkeeping of indices. We will present the formulas in general; for evaluation steps, we will include the details only for the $n=2$ case we are concerned with here.

\textbf{Pre-evaluation:} The first derivatives are of the form
\begin{align}
    D_{u_k} T_i \mathbf{u} [\eta_k] = - \tfrac{\alpha_{ik}}{\sigma} T_i \mathbf{u} \left( W_{ik} * \eta_k - \int_\mathbb{T} T_i \mathbf{u}\, W_{ik} * \eta_k \dx \right)
\end{align}
As in the scalar case, we make note of the Lipschitz continuity of $\mathcal{T}$ on $L^2$-balls using the boundedness of $D_{\mathbf{u}} \mathcal{T} \mathbf{u} [ \cdot ]$. 
\begin{proposition}\label{prop:T_multispecies_is_Lipschitz}
    Fix $r>0$ and kernels $W_{ij}$ satisfying hypothesis \textbf{\ref{hyp:kernel_shape}}. Then, the operator $\mathcal{T} : [ L^2(\mathbb{T}) ]^ n \mapsto \mathcal{T} : [ L^2(\mathbb{T}) ]^ n$ is Lipschitz on $E_r := \{ \mathbf{u} \in [ L^2 (\mathbb{T}) ]^n\ : \ \norm{\mathbf{u}}_{[L^2(\mathbb{T})]^n} \leq r \}$ in the sense that there exists a constant $M = M(r)>0$ so that
    $$
\norm{\mathcal{T} \mathbf{u} - \mathcal{T}\mathbf{v}}_{[ L^2 (\mathbb{T}) ]^n} \leq M \norm{\mathbf{u} - \mathbf{v}}_{ [L^2 (\mathbb{T})]^n}, \quad \forall \mathbf{u}, \mathbf{v} \in E_r. 
    $$
\end{proposition}
\begin{proof}[Proof of Proposition \ref{prop:T_multispecies_is_Lipschitz}]
    The proof is similar to the scalar case, Proposition \ref{prop:T_is_Lipschitz}, and so we provide the key details only. As in the proof of Proposition \ref{prop:T_is_Lipschitz}, we first apply Young's inequality to find that $W_{ij} * u_i \in L^\infty (\mathbb{T})$ for any $i,j = 1,\ldots,n$ so long as $u_i \in L^2 (\mathbb{T})$. Consequently, for each $i=1,\ldots,n$ there holds $T_i \mathbf{u} \in L^\infty (\mathbb{T})$ uniformly for $\mathbf{u} \in E_r$. 

    Again, as in the scalar case, we use the boundedness of each $T_i \mathbf{u}$ to conclude that for each $i=1,\ldots,n$ there exists $M_i>0$ such that $\norm{\sum_{k=1}^n D_{u_k}T_i \mathbf{u} [\eta_k]}_{L^2} \leq M_i \norm{\bm{\eta}}_{L^2(\mathbb{T})}$ uniformly in $E_r$, for any $\bm{\eta} \in [ L^2(\mathbb{T})]^n$. Consequently, 
    $$
\norm{D_{\mathbf{u}} \mathcal{T} \mathbf{u} [ \bm{\eta} ] }_{[L^2 (\mathbb{T})]^n}^2 = \sum_{i=1}^n \norm{\sum_{k=1}^n D_{u_k} T_i \mathbf{u} [ \eta_k ]}_{L^2(\mathbb{T})}^2 \leq M ^2 \norm{\bm{\eta}}_{[L^2(\mathbb{T})]^n}^2 ,
    $$
    where $M^2 := n \cdot \max_i \{ M_i^2 \}$. The proof is complete upon application of the mean value theorem to the map $\mathcal{T}$ in $E_r$.
\end{proof}

We now move to the higher-order derivatives. The second derivatives are of the form
\begin{align}
       D^2_{u_k u_l} T_i \mathbf{u} [\eta_k, \theta_l] &= - \tfrac{\alpha_{ik}}{\sigma} \left[ D_{u_l} T_i \mathbf{u} [\theta_l]\left( W_{ik} * \eta_k - \int_\mathbb{T} T_i \mathbf{u}\, W_{ik} * \eta_k \dx \right)  \right. \nonumber \\
      &\quad  \left. - T_i \mathbf{u} \int_\mathbb{T} D_{u_l} T_i \mathbf{u} [\theta_l]\, W_{ik} * \eta_k \dx  \right]
\end{align}

The third derivatives are of the form
\begin{align}
       D^3_{u_k u_l u_m} T_i \mathbf{u} [\eta_k, \theta_l, \phi_m] &= - \tfrac{\alpha_{ik}}{\sigma} \left[ D^2_{u_l u_m} T_i \mathbf{u} [\theta_l, \phi_m] \left( W_{ik} * \eta_k - \int_\mathbb{T} T_i \mathbf{u}\, W_{ik} * \eta_k \dx \right) \right. \nonumber \\
       &\quad \left. - D_{u_l} T_i \mathbf{u} [\theta_l] \int_\mathbb{T} D_{u_m}T_i \mathbf{u} [\phi_m]\, W_{ik} * \eta_k \dx  \right. \nonumber \\
    &\quad  \left. - D_{u_m} T_i \mathbf{u} [\phi_m] \int_\mathbb{T} D_{u_l} T_i \mathbf{u} [\theta_l]\, W_{ik} * \eta_k \dx  \right. \nonumber \\
      &\quad  \left. - T_i \mathbf{u} \int_\mathbb{T} D^2_{u_l u_m} T_i \mathbf{u} [\theta_l, \phi_m]\, W_{ik} * \eta_k \dx  \right].
\end{align}
The utility in approaching the derivatives this way is the same as in the scalar case: once the first derivatives are evaluated, higher order derivatives can be evaluated in terms of previously calculated derivatives. 

\textbf{Post-evaluation:} We now evaluate the formulas above at the homogeneous state $\mathbf{u} = \mathbf{u}_\infty = (L^{-1}, \ldots, L^{-1})$. 

The first derivatives at the homogeneous state are
\begin{align}
    D_{u_k} T_i \mathbf{u} [\eta_k] \biggr\vert_{\mathbf{u}=\mathbf{u}_\infty} = - \frac{\alpha_{ik}}{\sigma L} F( \eta_k ; W_{ik} ).
\end{align}

The second derivatives at the homogeneous state are
\begin{align}\label{eq:second_der_hom_general}
       D^2_{u_k u_l} T_i \mathbf{u} [\eta_k, \theta_l] \biggr\vert_{\mathbf{u}=\mathbf{u}_\infty} &=  \tfrac{\alpha_{ik} \alpha_{il}}{\sigma^2 L} \left[  F( \theta_l ; W_{il}) F( \eta_k ; W_{ik}) -\frac{1}{L} \int_\mathbb{T}F( \theta_l ; W_{il} ) \, W_{ik} * \eta_k \dx  \right].
\end{align}

The third derivatives at the homogeneous state are:
\begin{align}
       D^3_{u_k u_l u_m} T_i \mathbf{u} [\eta_k, \theta_l, \phi_m] & \biggr\vert_{\mathbf{u}=\mathbf{u}_\infty} = - \tfrac{\alpha_{ik} \alpha_{il} \alpha_{im}}{\sigma^3 L} \left[ \phantom{\int}\hspace{-0.3cm}F( \eta_k ; W_{ik}) F( \theta_l ; W_{il} ) F( \phi_m ; W_{im} )\right. \nonumber \\
       &- \tfrac{1}{L}\left( F( \eta_k ; W_{ik}) \int_\mathbb{T} F( \phi_m ; W_{im} ) W_{il} * \theta_l \dx  + F( \theta_l ; W_{il} ) \int_\mathbb{T} F( \phi_m ; W_{im} ) W_{il} * \eta_k \dx  \right. \nonumber \\
       &\, + \left. F( \phi_m ; W_{im} ) \int_\mathbb{T} F( \theta_l ; W_{il} ) W_{ik} * \eta_k \dx\right)  \nonumber \\
       &\left.- \tfrac{1}{L} \left( \int_\mathbb{T} \left( F( \phi_m ; W_{im} ) F( \theta_l ; W_{il} ) - \tfrac{1}{L} \int_\mathbb{T} F( \phi_m ; W_{im} ) W_{il} * \theta_l \dy \right) \dx \right)  \right].
\end{align}

In all formulas above, if one chooses $i=k=l=m=1$ these immediately reduce to the formulas for the scalar equation.

\textbf{Variations of basis vector form:} Finally, we evaluate all quantities at a basis vector of the form $w_q$, $(w_q, w_q)$, and $(w_q, w_q, w_q)$, where $q \in \mathbb{N}$. For the first derivative we have
\begin{align}
        D_{u_k} T_i \mathbf{u} [w_q] \biggr\vert_{\mathbf{u}=\mathbf{u}_\infty} = - \frac{\alpha_{ik}}{\sigma \sqrt{2L}} \widetilde W_{ik} (q) w_q (x).
\end{align}
The second derivative is
\begin{align}\label{appendix:second_frechet_system_basis_form}
       D^2_{u_k u_l} T_i \mathbf{u} [w_q, w_q] \biggr\vert_{\mathbf{u}=\mathbf{u}_\infty} &=  \tfrac{\alpha_{ik} \alpha_{il}}{\sigma^2 L} \left[  F( w_q ; W_{il} ) F( w_q ; W_{ik} ) -\frac{1}{L} \int_\mathbb{T}F( w_q ; W_{il} ) \, W_{ik} * w_q \dx  \right] \nonumber \\
       &= \tfrac{\alpha_{ik} \alpha_{il} \widetilde W_{il} (q) \widetilde W_{ik} (q)}{2 \sigma^2 } \left[ w_q^2 (x) -\frac{1}{L}  \right].
\end{align}
The third derivative is
\begin{align}\label{appendix:third_frechet_system_basis_form}
       &D^3_{u_k u_l u_m} T_i \mathbf{u} [w_q, w_q, w_q]\biggr\vert_{\mathbf{u}=\mathbf{u}_\infty} = - \frac{\alpha_{ik} \alpha_{il} \alpha_{im} \widetilde W_{ik} (q) \widetilde W_{il} (q) \widetilde W_{im} (q)}{\sigma^3 L} \left( \tfrac{L}{2} \right)^{3/2} \left[ w_q^3 (x) - \tfrac{3}{L} w_q (x)  \right]
\end{align}

\subsection{Derivatives w.r.t. parameters}

As in the scalar case, we find the derivative of $T_i \mathbf{u}$ with respect to $\alpha_{ik}$ to be
\begin{align}
    D_{\alpha_{ik}} T_i \mathbf{u} = -\tfrac{1}{\sigma} T_i \mathbf{u} \left( W_{ik} * u_k - \int_\mathbb{T} T_i \mathbf{u} W_{ik} * u_k \dx \right).
\end{align}
Notice that we need only consider derivatives of $T_i$ with respect to $\alpha_{ik}$, $k = 1,\ldots,n$; each $T_i$ does not depend on $\alpha_{lk}$ for $l \neq i$.

The second-order mixed derivative is identified as
\begin{align}
    D^2 _{\alpha_{il} u_k} T_i \mathbf{u} [ \eta_k ] = - \frac{\delta_{lk}}{\sigma} T_i \mathbf{u} \left( W_{ik} * \eta_k - \int_\mathbb{T} T_i \mathbf{u}\, W_{ik} * \eta_k \dx \right),
\end{align}
where $\delta_{lk}$ is the Kronecker-delta function.

\textbf{Post-evaluation:} We now evaluate at the stationary state. As in the scalar case, we find for the first derivative that 
\begin{align}
    D_{\alpha_{ik}} T_i \mathbf{u} \biggr\vert_{\mathbf{u} = \mathbf{u}_\infty} = 0.
\end{align}
For the second derivative, we find
\begin{align}
    D^2 _{\alpha_{il} u_k} T_i \mathbf{u} [ \eta_k ] \biggr\vert_{\mathbf{u}=\mathbf{u}_\infty} = - \frac{\delta_{lk}}{\sigma L} F( \eta_k ; W_{ik}  ).
\end{align}

\textbf{Variations of basis vector form:} Finally, we evaluate the mixed derivative at the basis vector $w_q$, $q \in \mathbb{N}$. We find
\begin{align}
    D^2 _{\alpha_{il} u_k} T_i \mathbf{u} [ w_q ] \biggr\vert_{\mathbf{u}=\mathbf{u}_\infty} = - \frac{\delta_{lk} \widetilde W_{ik} (k)}{\sigma \sqrt{2L}} w_q (x).
\end{align}

\section{\ Proof of Theorem \ref{thm:existenceregularitySS}}

We begin with the proof of Theorem \ref{thm:existenceregularitySS}.

\begin{proof}[Theorem \ref{thm:existenceregularitySS}]
For the following proof, since all quantities are over space only, we drop the dependence on $\mathbb{T}$ in most estimates for brevity when the context is clear. Without loss of generality, we assume $\sigma = 1$. We first prove part a.). 

\textbf{Step 1: Preliminary estimates.} The weak formulation is given by
\begin{align}\label{eq:weakformSS}
    \int_\mathbb{T}  \frac{\p u_i}{\p x} \frac{\p \phi_i}{\p x} {\rm{d}}x  + \int_\mathbb{T} u_i \frac{\p \phi_i}{\p x} \frac{\partial}{\partial x} [ \alpha_i W_i * u_i + \gamma W * u_j ] {\rm{d}}x = 0, \quad\quad \forall \phi_i \in H^1 (\mathbb{T}),
\end{align}
for $i=1,2$, $i\neq j$, where we seek solutions belonging to $\left[  H^1 (\mathbb{T}) \cap \mathcal{P}_{\textup{ac}} (\mathbb{T})\right]^2$. We show existence using a fixed point argument. To this end, we use the norm $\left< \cdot, \cdot \right>_{\mathbb{H}}$ as defined in \eqref{euclidian_operator_norm_def}:
\begin{align}
    \left< \mathcal{T}\mathbf{u}, \mathcal{T}\mathbf{u}\right>_{\mathbb{H}} = \sum_{i=1}^2 \norm{T_i \mathbf{u}}_{L^2 (\mathbb{T})}^2 .
\end{align}
For any $\mathbf{u}\in \left[  H^1 (\mathbb{T}) \cap \mathcal{P}_{\textup{ac}} (\mathbb{T})\right]^2$ H{\"o}lder's inequality gives for each component that
\begin{align}\label{bnd:TL2Linf}
    \norm{T_i \mathbf{u}}_{L^2 }^2 \leq \norm{T_i \mathbf{u}}_{L^\infty },
\end{align}
where we use the fact that $\norm{\mathcal{T}_i \mathbf{u}}_{L^1} = 1$ for each $i=1,2$ due to the normalization factor $Z_i$. 

Next, using that all kernels are bounded and $u_i \in L^2 (\mathbb{T}) \cap \mathcal{P}_{\textup{ac}} (\mathbb{T})$, Young's inequality gives that $W_i * u_j, W * u_i \in L^\infty$ for all $i,j=1,2$, yielding the estimate
\begin{align}
    Z_i ^{-1} \leq L^{-1} e^{ \alpha_i \norm{W_i}_{L^\infty } + \gamma \norm{W}_{L^\infty } },
\end{align}
from which we conclude
\begin{align}\label{bnd:TboundLinf}
    \norm{T_i \mathbf{u}}_{L^\infty } \leq L^{-1} e^{2 ( \alpha_i \norm{W_i}_{L^\infty } + \gamma \norm{W}_{L^\infty } )}.
\end{align}
Thus, combining estimates \eqref{bnd:TL2Linf} and \eqref{bnd:TboundLinf} there holds
\begin{align}\label{bnd:TboundL2}
    \left< \mathcal{T} \mathbf{u}, \mathcal{T} \mathbf{u} \right>_{\mathbb{H}} \leq L^{-1} e^{2 \gamma \norm{W}_{L^\infty }} \sum_{i=1}^2 e^{2  \alpha_i \norm{W_i}_{L^\infty } } =: M_0
\end{align}

\textbf{Step 2: Existence of a fixed point.} Motivated by the estimates above, we seek fixed points of the map $\mathcal{T}$ belonging to
\begin{align}
    \Gamma := \left\{ \mathbf{u} \in \left[ L^2 (\mathbb{T}) \cap \mathcal{P}_{\textup{ac}} (\mathbb{T})  \right]^2 : \left< \mathbf{u}, \mathbf{u} \right>_{\mathbb{H}} \leq M_0 \right\}.
\end{align}
First, notice that $\Gamma$ is a closed, convex subset of $L^2 (\mathbb{T}) \times L^2 (\mathbb{T})$. Closedness follows from the uniformity of the bound; convexity follows from the absolute homogeneity of the norm. Hence, we redefine $\mathcal{T}$ to act on $\Gamma$. Furthermore, from bound \eqref{bnd:TboundL2}, we have that $\mathcal{T} \Gamma \subset \Gamma$, i.e., $\mathcal{T}$ maps $\Gamma$ into itself. 

We now seek some compactness of the map $\mathcal{T}$. Different from \cite{carrillo2019aggregation}, we do not assume any weak differentiability of the kernels themselves, and some additional care must be taken. Instead, we use Hypothesis \textbf{\ref{hyp:kernel_shape}} and Lemma \ref{lemma:BV_TV_embeddings}: given $u \in L^2(\mathbb{T})$ and $W \in \textup{BV}(\mathbb{T})$, it follows that $(W * u)$ is weakly differentiable with $(W*u)_x = DW * u$, where $DW$ denotes the distributional derivative of $W$. Therefore, for any $\mathbf{u} \in \Gamma$, each component $T_i \mathbf{u}$ is weakly differentiable with
\begin{align*}
    \frac{\partial T_i \mathbf{u}}{\partial x} = - T_i \mathbf{u} \left( \alpha_i DW_i * u_i + \gamma DW * u_j  \right), \quad i=1,2, \ \ i \neq j.
\end{align*}
Consequently, we have the estimate
\begin{align}\label{eq:T_maps_to_H1}
    \norm{\frac{\partial T_i \mathbf{u}}{\partial x}}_{L^2} \leq &\ \norm{T_i \mathbf{u}}_{L^\infty} \left( \alpha_i \norm{DW_i}_{\textup{TV}} \norm{u_i}_{L^2} + \gamma \norm{DW}_{\textup{TV}} \norm{u_j}_{L^2}  \right) \nonumber \\
    \leq &\  M_0^2 \left( \alpha_i \norm{DW_i}_{\textup{TV}} + \gamma \norm{DW}_{\textup{TV}} \right), \quad i=1,2,
\end{align}
and so combined with estimate \eqref{bnd:TboundL2} we find that $\mathcal{T} \Gamma \subset \Gamma$ is uniformly bounded in $H^1(\mathbb{T})$. By the Rellich-Kondrachov compactness theorem, $\mathcal{T} \Gamma$ is relatively compact in $L^2 (\mathbb{T}) \times L^2(\mathbb{T})$, and hence in $\Gamma$, since $\Gamma$ is closed. 

Finally, by Proposition \ref{prop:T_multispecies_is_Lipschitz}, we have that $\mathcal{T}$ is a Lipschitz-continuous map acting on $\Gamma$. 

Thus, by Schauder's fixed point theorem, we conclude that $\mathcal{T}$ has a fixed point $\mathbf{u}\in \Gamma$. From the exponential form of the nonlinear map and the fact that $ \norm{u_i}_{L^\infty} = \norm{T_i \mathbf{u}}_{L^\infty} < \infty$, it is not difficult to conclude that $u_i = T_i \mathbf{u} \geq m_0 >  0$ a.e. in $\mathbb{T}$, $i=1,2$, for some $m_0 > 0$. From estimate \eqref{eq:T_maps_to_H1} and the preceding calculation, one necessarily has that $\mathcal{T} \mathbf{u} \in H^1 (\mathbb{T}) \times H^1 (\mathbb{T})$. 

\textbf{Step 3: Any solution is a fixed point.} We now argue that any weak solution belonging to  $\left[ H^1 (\mathbb{T}) \cap \mathcal{P}_{\textup{ac}} (\mathbb{T}) \right]^2$ must be a fixed point of the map $\mathcal{T}$. To this end, we consider the ``frozen" problem
\begin{align}\label{eq:weakformSSfrozen}
    \int_\mathbb{T}  \frac{\p v_i}{\p x} \frac{\p \phi_i}{\p x} {\rm{d}}x  + \int_\mathbb{T} v_i \frac{\p \phi_i}{\p x} \frac{\p}{\p x}\left( \alpha_i W_i * u_i + \gamma W * u_j \right) {\rm{d}}x = 0, \quad\quad \forall \phi_i \in H^1 (\mathbb{T}), \ i=1,2,\, i \neq j,
\end{align}
where $\mathbf{u} = (u_1,u_2)$ is a weak solution obtained in the previous step, and $(v_1,v_2)$ is the unknown. This is indeed a weak form of a uniformly elliptic PDE with associated bilinear form coercive in the weighted space $\left[ H_0 ^1 (\mathbb{T}, \mathcal{T}\mathbf{u} )\right]^2$ with $H_0 ^1 (\mathbb{T}) = H^1 (\mathbb{T}) / \mathbb{R}$. To see this, define $v_i := h_i (x) T_i \mathbf{u}$ for each $i=1,2$. Then, problem \eqref{eq:weakformSSfrozen} is equivalent to
\begin{align}\label{eq:weakformfrozen2}
    0 = \int_\mathbb{T} \frac{\p \phi_i}{\p x} \frac{\p h_i}{\p x} T_i \mathbf{u} \, \dx , \quad\quad \forall \phi_i \in H^1 (\mathbb{T}),\ i=1,2.
\end{align}
Then, given $(u_1,u_2)$, suppose $(h_1^1, h_2^2)$ and $(h_1^2, h_2^2)$ are two weak solution pairs. Choose $\phi_i = h_i ^1 - h_i ^2$ and (without loss of generality) $h_i := h_i ^1$. Then,
\begin{align}
    0 = \int_\mathbb{T} \frac{\p (h_i ^1 - h_i ^2)}{\p x} \frac{\p (h_i^1 - h_i ^2)}{\p x} T_i \mathbf{u}\, \dx \geq m_0  \int_\mathbb{T} \left|\frac{\p }{\p x} (h_i ^1 - h_i ^2)\right|^2 \dx \geq 0,
\end{align}
where we have used that $T_i \mathbf{u} \geq m_0 > 0$ in $\mathbb{T}$. Hence, any solution to problem \eqref{eq:weakformfrozen2} is unique up to scaling (notice carefully that $\mathbf{u}$ is certainly not unique in general; instead, for a \textit{given} fixed point $\mathbf{u}$, the solution to the frozen problem is unique). In particular, the solution is unique in $\left[ \mathcal{P}_{\textup{ac}}^+ (\mathbb{T}) \right]^2$. Thus, we conclude that $(u_1,u_2) = (v_1,v_2)$ is a weak solution, and $\mathcal{T}(u_1,u_2) = (v_1, v_2)$ is also a weak solution. Therefore, $(u_1,u_2) = \mathcal{T}(u_1,u_2)$ and is a fixed point of the map $\mathcal{T}$. This completes the proof of part a.).

For part b., we have already shown that the image of $\mathcal{T}$ belongs to $H^1 (\mathbb{T}) \times H^1 (\mathbb{T})$. Moreover, such a fixed point is necessarily uniformly bounded. Hence, we may estimate as in \eqref{eq:T_maps_to_H1} by replacing $L^2$ with $L^p$ for any $p \geq 1$ and use Lemma \ref{lemma:BV_TV_embeddings} once more to conclude that in fact $\mathbf{u} = \mathcal{T} \mathbf{u} \in W^{1,p} (\mathbb{T}) \times W^{1,p} (\mathbb{T})$ for any $p \geq 1$.

Consequently, for $i=1,2$ we have that $(u_i)_x = (T_i \mathbf{u})_{x} \in L^p(\mathbb{T})$ for any $p \geq 1$, and so by Lemma \ref{lemma:BV_TV_embeddings}, $u_i$ is twice weakly differentiable with
\begin{align*}
    \frac{\partial^2 u_i}{\partial x^2} = &\ - \frac{ \partial u_i}{\partial x} \left( \alpha_i DW_i * u_i + \gamma DW * u_j  \right) - u_i \mathbf{u} \left( \alpha_i DW_i * \frac{ \partial u_i}{\partial x} + \gamma DW * \frac{ \partial u_j }{\partial x} \right) \ \in L^p (\mathbb{T}),
\end{align*}
for any $p \geq 1$. Therefore, $\norm{(T_i \mathbf{u})_{xx}}_{L^p(\mathbb{T})}$ is bounded and we have that $\mathbf{u} = \mathcal{T} \mathbf{u} \in W^{2,p}(\mathbb{T}) \times W^{2,p}(\mathbb{T})$ for all $ p \geq 1$. 

Using Lemma \ref{lemma:BV_TV_embeddings}, one may now proceed inductively to conclude that $\tfrac{\partial^{k} u_i }{\partial x^k} \in L^p(\mathbb{T}) \Rightarrow \tfrac{\partial^{k+1}( W * u_i ) }{\partial x^{k+1}} \in L^p(\mathbb{T})$. One then obtains that $\mathbf{u} \in W^{k,p}(\mathbb{T}) \times W^{k,p}(\mathbb{T})$ for all $k \geq 1$, for any $p \geq 1$. The final conclusion then follows from the Sobolev embedding, and part b.) is proven.
\end{proof}

\section{\ Proof of Proposition \ref{prop:equivalencies1}}

Next, we prove Proposition \ref{prop:equivalencies1}.

\begin{proof}[Proposition \ref{prop:equivalencies1}]
Without loss of generality, we again assume that $\sigma = 1$.

$(1) \Leftrightarrow (2)$: Notice that $\mathbf{u}$ is a zero  of the map $\widehat{G}$ if and only if it is a fixed point of $\mathcal{T}$. Hence, part a.) of Theorem \ref{thm:existenceregularitySS} gives the equivalence immediately.

$(2) \Rightarrow (3)$: The observation of note is that zeros of the map $\widehat{G}$ represent solutions of the Euler-Lagrange equations for the free energy functional $\mathcal{F}$. To this end, let $\mathbf{u} = (u_1,u_2)$ and $\mathbf{\tilde u} = (\tilde u_1, \tilde u_2)$ belong to $\left[ \mathcal{P}_{\textup{ac}}^+ (\mathbb{T})  \right]^2$. Define the convex interpolant
\begin{align*}
    \mathbf{u}_s := s \mathbf{u} + (1-s) \mathbf{\tilde u} , \quad s \in (0,1),
\end{align*}
where $\mathcal{F} ( \mathbf{u})$, $\mathcal{F} (\mathbf{\tilde u}) < \infty$. The Euler-Lagrange equations, well-defined for $\mathbf{u}$, $\mathbf{\tilde u} \in \left[ \mathcal{P}_{\textup{ac}}^+ (\mathbb{T})  \right]^2$, are given by
\begin{align}\label{eqn:ELeqns}
    \frac{{\rm d}}{{\rm d} s} \mathcal{F} (\mathbf{u}_s) \biggr\vert_{s=0} = \sum_{i=1}^2 \int_\mathbb{T} \eta_i \left( \log (\tilde u_i) + \alpha_i W_i * \tilde u_i +  \gamma W * \tilde u_j \right) \dx = 0, \quad i \neq j,
\end{align}
where $\eta_i = u_i - \tilde u_i$. Notice then that if $\mathbf{\tilde u}$ is a zero of $\widehat{G}$, the expression above is zero for any given $\mathbf{u} \in \left[ \mathcal{P}_{\textup{ac}}^+ (\mathbb{T})  \right]^2$. Indeed, this follows from the definition of $\widehat{G}$ and filling in the expression for each $\tilde u_i$ directly into \eqref{eqn:ELeqns}.

$(3) \Rightarrow (2)$: Suppose now that $\mathbf{\tilde u} = ( \tilde u_1,  \tilde u_2) \in \left[ \mathcal{P}_{\textup{ac}}^+ (\mathbb{T})  \right]^2$ is a critical point of the free energy functional $\mathcal{F}$. Define for each $i=1,2$ a function $f_i : \left[ \mathcal{P}_{\textup{ac}}^+ (\mathbb{T})  \right]^2 \mapsto \mathbb{R}$ by
\begin{align*}
    f_i (\mathbf{u}) = \log (u_i) + \alpha_i W_i * u_i +  \gamma W * u_j , \quad i \neq j.
\end{align*}
Suppose now that $f_i (\mathbf{\tilde u})$ is not constant almost everywhere in $\mathbb{T}$ for at least one $i=1,2$. Then, there exists sets $A_i \in \mathcal{B}(\mathbb{T})$ of the form
\begin{align}
    A_i := \{ x \in \mathbb{T} : f_i (\mathbf{u}) - \frac{1}{\as{\mathbb{T}}} \int_\mathbb{T} f_i (\mathbf{u}) {\rm d} y > 0 \},
\end{align}
where at least one $A_i$ has positive Lebesgue measure. Without loss of generality, we may assume that it is $A_1$. Then we set
\begin{align}
    v_1 := \frac{\chi_{A_1} (x)}{\as{A_1}} - \frac{\chi_{A_1 ^\mathsf{c}} (x)}{\as{A_1 ^\mathsf{c}}},
\end{align}
where we note that $\int_\mathbb{T} v_1 \dx = 0$. Then, since the critical point $\mathbf{\tilde u}$ has strictly positive density almost everywhere, for $\varepsilon>0$ sufficiently small we have that
\begin{align}
    \mathbf{u} := ( \tilde u_1 + \varepsilon v_1 , \tilde u_2 ) \in \left[ \mathcal{P}_{\textup{ac}}^+ (\mathbb{T})  \right]^2  .
\end{align}
We then compute
\begin{align}
     \frac{{\rm d}}{{\rm d} s} \mathcal{F} (\mathbf{u}_s) \biggr\vert_{s=0} =&\ \sum_{i=1}^2 \int_\mathbb{T} (u_i - \tilde u_i) f_i (\mathbf{\tilde u}) \dx \nonumber \\
     =&\ \frac{\varepsilon}{\as{A_1}} \int_{A_1} f_1 (\mathbf{\tilde u}) \dx - \frac{\varepsilon}{\as{A_1 ^\mathsf{c}}} \int_{A_1 ^\mathsf{c}} f_1 (\mathbf{\tilde u}) \dx \nonumber \\
     =&\ \frac{\varepsilon}{\as{A_1}} \int_{A_1} \left( f_1 (\mathbf{\tilde u}) - \frac{1}{\as{\mathbb{T}}} \int_\mathbb{T} f_i (\mathbf{u}) {\rm d} y \right) \dx \nonumber + \frac{\varepsilon}{\as{\mathbb{T}}} \int_\mathbb{T} f_i (\mathbf{u}) {\rm d} y \nonumber \\
     & - \frac{\varepsilon}{\as{A_1 ^\mathsf{c}}} \int_{A_1 ^\mathsf{c}} \left(  f_1 (\mathbf{\tilde u}) - \frac{1}{\as{\mathbb{T}}} \int_\mathbb{T} f_i (\mathbf{u}) {\rm d} y \right) \dx - \frac{\varepsilon}{\as{\mathbb{T}}} \int_\mathbb{T} f_i (\mathbf{u}) {\rm d} y \nonumber \\
     &>0 ,
\end{align}
since the integrand of the positive term is positive, and the integrand of the negative term is negative. An identical calculation holds if $A_2$ has positive measure or if both $A_i$ have positive measure. Hence, for $\mathbf{u}$ defined above with $\varepsilon$ sufficiently small but positive, we have derived a contradiction to the fact that $\mathbf{\tilde u}$ is a critical point of $\mathcal{F}$. Hence, it must be the case that $f_i (\mathbf{\tilde u})$ is constant almost everywhere over $\mathbb{T}$ for each $i=1,2$, and consequently satisfies the Euler-Lagrange equations, completing the case $(3) \Rightarrow (2)$.

$(2) \Rightarrow (4)$: Suppose $\mathbf{u}$ is a zero of the map $\widehat{G}$. Clearly $\mathcal{J}\geq0$ by definition. From the definition of the map $\widehat{G}$, we may then plug any such zero $\mathbf{u}$
into $\mathcal{J}$ to find $\mathcal{J}(\mathbf{u}) = 0$.

$(4) \Rightarrow (2)$: Since $\mathbf{u}$ is assumed strictly positive in each component, and since $\mathcal{J}(\mathbf{u}) = 0$, this can only hold if
\begin{align*}
    \frac{\p}{\p x} \left( \log (u_i) + \alpha_i W_i * u_i + \gamma W*u_j \right) = 0 \quad \text{a.e. in } \mathbb{T},\quad i=1,2, \quad i \neq j.
\end{align*}
Hence, for each $i$ there holds $0 = u_i - C_i e^{-( \alpha_i W_i * u_i + \gamma W*u_j)}$ for some $C_i>0$, and since $u_i \in \mathcal{P}_{\textup{ac}}^+ (\mathbb{T} ) $, $c_i \equiv Z_i (\mathbf{u})$ as previously defined. Hence, $\mathbf{u}$ is a zero of $\widehat{G}$, completing the proof. 
\end{proof}

\end{document}